\documentclass[12pt]{amsart}
\usepackage{mathrsfs, amssymb, array, amsmath, upgreek, amsfonts, mathtools, mathcomp, longtable}
\usepackage{amsthm}
\usepackage{xcolor,graphicx}

\usepackage[matrix,arrow,curve]{xy}
\usepackage{hyperref}

\renewcommand{\emptyset}{\varnothing}

\newcommand{\ZZ}{\mathbb{Z}}
\newcommand{\PP}{\mathbb{P}}
\newcommand{\QQ}{\mathbb{Q}}

\newcommand{\cO}{{\mathscr{O}}}

\newcommand{\cF}{{\mathscr{F}}}

\newcommand{\g}{{\mathrm{g}}}
\newcommand{\kk}{{\mathsf{k}}}
\newcommand{\bkk}{{\bar{\mathsf{k}}}}
\newcommand{\Gal}{{\mathbf{G}}}
\newcommand{\dd}{{\mathrm{d}}}
\newcommand{\rG}{{\mathrm{G}}}
\newcommand{\rD}{{\mathrm{D}}}
\newcommand{\rR}{{\mathrm{R}}}

\DeclareMathOperator{\Br}{\mathrm{Br}}
\DeclareMathOperator{\alb}{\mathrm{alb}}
\DeclareMathOperator{\Alb}{\mathrm{Alb}}
\DeclareMathOperator{\Jac}{\mathrm{Jac}}
\DeclareMathOperator{\AJ}{\mathrm{AJ}}
\DeclareMathOperator{\NS}{\mathrm{NS}}
\DeclareMathOperator{\CH}{\mathrm{CH}}

\DeclareMathOperator{\Ext}{\mathrm{Ext}}
\DeclareMathOperator{\Hom}{\mathrm{Hom}}
\DeclareMathOperator{\Sing}{\mathrm{Sing}}
\DeclareMathOperator{\Spec}{\mathrm{Spec}}
\DeclareMathOperator{\Proj}{\mathrm{Proj}}
\DeclareMathOperator{\rank}{\mathrm{rank}}
\DeclareMathOperator{\Pic}{\mathrm{Pic}}
\DeclareMathOperator{\Cl}{\mathrm{Cl}}
\DeclareMathOperator{\Bs}{\mathrm{Bs}}
\DeclareMathOperator{\rk}{\mathrm{rk}}
\newcommand{\Tor}{\operatorname{Tor}}

\newcommand{\larrow}{\longrightarrow}

\newcommand{\bN}{{\mathbf{N}}}

\newcommand{\bX}{{\mathbf{X}}}

\DeclareMathOperator{\Gr}{\mathrm{Gr}}
\newcommand{\tX}{{\tilde{X}}}

\newcommand{\tpsi}{\psi}

\DeclareMathOperator{\Bl}{\mathrm{Bl}}

\DeclareMathOperator{\OGr}{\mathrm{OGr}}

\newcommand{\cP}{{\mathscr{P}}}

\newcommand{\bT}{{\mathbf{T}}}

\DeclareMathOperator{\Sym}{\mathrm{Sym}}
\DeclareMathOperator{\Ker}{\mathrm{Ker}}

\newcommand{\cN}{{\mathscr{N}}}

\newcommand{\cU}{{\mathscr{U}}}

\newcommand{\tL}{{\tilde{L}}}
\newcommand{\rc}{{\mathrm{c}}}
\newcommand{\rch}{{\mathrm{ch}}}
\newcommand{\rtd}{{\mathrm{td}}}
\newcommand{\cE}{{\mathscr{E}}}

\newcommand{\cM}{\mathcal{M}}
\newcommand{\tC}{\tilde{C}}
\newcommand{\tS}{\tilde{S}}
\newcommand{\tcM}{\widetilde{\mathcal{M}}}
\newcommand{\tcS}{\widetilde{\mathcal{S}}}

\newcommand{\tY}{{\tilde{Y}}}
\newcommand{\tD}{{\tilde{D}}}

\newcommand{\hX}{{\hat{X}}}
\newcommand{\hY}{{\hat{Y}}}

\newcommand{\hE}{{\hat{E}}}
\newcommand{\hQ}{{\hat{Q}}}

\newcommand{\barX}{{\bar{X}}}
\newcommand{\barY}{{\bar{Y}}}
\newcommand{\barE}{{\bar{E}}}
\newcommand{\barF}{{\bar{F}}}
\newcommand{\barH}{{\bar{H}}}

\DeclareMathOperator{\codim}{\mathrm{codim}}
\DeclareMathOperator{\id}{\mathrm{id}}

\DeclareMathOperator{\LGr}{\mathrm{LGr}}

\newcommand{\dc}{d_{\mathrm{C}}}
\newcommand{\dtor}{d_{\mathrm{T}}}

\newcommand{\cC}{\mathscr{C}} 
 
\newcommand{\cL}{\mathscr{L}} 
\newcommand{\rF}{\mathrm{F}}
\newcommand{\rM}{\mathrm{M}}

\DeclareMathOperator{\GTGr}{G_2Gr}

\DeclareMathOperator{\SL}{SL}

\newcommand{\nil}{\mathrm{null}}
\newcommand{\alg}{\mathrm{alg}}

\newcommand{\Db}{\mathrm{D}^{\mathrm{b}}}

\renewcommand\labelenumi{\rm (\roman{enumi})}
\renewcommand\theenumi{\rm (\roman{enumi})}

\makeatletter
\@addtoreset{equation}{subsection}
\makeatother
\renewcommand{\theequation}
{\arabic{section}.\arabic{subsection}.\arabic{equation}}
\renewcommand{\thesubsection}{\arabic{section}.\arabic{subsection}}

\theoremstyle{plain}

\newtheorem{theorem}{Theorem}[section]
\newtheorem{lemma}[theorem]{Lemma}

\newtheorem{proposition}[theorem]{Proposition}
\newtheorem{corollary}[theorem]{Corollary}

\theoremstyle{definition}
\newtheorem{definition}[theorem]{Definition}

\newtheorem*{definition*}{Definition}
\newtheorem{example-remark}{Remark-Example}

\newtheorem*{notation*}{Notation}

\newtheorem{remark}[theorem]{Remark}

\title{Rationality of Fano threefolds over non-closed fields}
\author{Alexander Kuznetsov}
\thanks{The authors were partially supported by the HSE University Basic Research Program.}
\address{\parbox{0.9\textwidth}{Steklov Mathematical Institute of Russian Academy of Sciences, Moscow, Russia 
\\[1pt]
Interdisciplinary Scientific Center J.-V. Poncelet (CNRS UMI 2615), Moscow, Russia
\\[1pt]
Laboratory of Algebraic Geometry, NRU HSE, Moscow, Russia\\}}

\email{akuznet@mi-ras.ru}

\author{Yuri Prokhorov}

\address{\parbox{0.9\textwidth}{Steklov Mathematical Institute of Russian Academy of Sciences, Moscow, Russia 
\\[1pt]
Laboratory of Algebraic Geometry, NRU HSE, Moscow, Russia
\\[1pt]
Department of Algebra, Moscow State University, Moscow, Russia
\\}}

\email{prokhoro@mi-ras.ru}

\date{}

\begin{document}

\begin{abstract}
We give necessary and sufficient conditions for unirationality and rationality of Fano threefolds
of geometric Picard rank~1 over an arbitrary field of zero characteristic.
\end{abstract}

\maketitle

\tableofcontents

\section{Introduction}

\subsection{The main result}

The goal of this paper is to explore the questions of rationality for smooth Fano threefolds
with geometric Picard rank one over an arbitrary field~$\kk$ of zero characteristic.
A~$\kk$-rational variety is a fortiori rational over the algebraic closure~$\bkk$ of~$\kk$.
Therefore we restrict our attention to the following eight families of geometrically rational Fano varieties:
\begin{itemize}
\item $\PP^3$, a projective 3-space;
\item $Q^3$, a 3-dimensional quadric in $\PP^4$;
\item $V_4$, a quartic del Pezzo threefold in $\PP^5$;
\item $V_5$, a quintic del Pezzo threefold in $\PP^6$;
\item $X_{12}$, a prime Fano threefold of genus~7 in $\PP^{8}$;
\item $X_{16}$, a prime Fano threefold of genus~9 in $\PP^{10}$;
\item $X_{18}$, a prime Fano threefold of genus~10 in $\PP^{11}$;
\item $X_{22}$, a prime Fano threefold of genus~12 in $\PP^{13}$;
\end{itemize}
see~\S\ref{subsection:fano} for a geometric description of varieties in these families.

The precise question that we address in the paper is: 
under what conditions a $\kk$-form of one of these varieties is $\kk$-rational or $\kk$-unirational 
(cf.~\cite[Exercises~2.3, 2.4]{Manin:arFano:CM}).
A complete answer to this question is given in the following theorem, which is the main result of our paper.

\begin{theorem}
\label{theorem:main}
Let $\kk$ be a field of characteristic zero.

\begin{enumerate}
\item 
\label{main:v5}
If $X$ is a $\kk$-form of $V_5$ then $X$ is $\kk$-rational.
\item 
\label{main:p3q3x12x22}
If $X$ is a $\kk$-form of $\PP^3$, $Q^3$, $X_{12}$, or $X_{22}$ then 
$X$ is $\kk$-rational if and only if $X(\kk) \ne \varnothing$.
\item 
\label{main:v4x16x18}
If $X$ is a $\kk$-form of $V_4$, $X_{16}$, or $X_{18}$ then 
$X$ is $\kk$-unirational if and only if $X(\kk) \ne \varnothing$.
Moreover, 
\begin{enumerate}
\item 
\label{main:v4}
a $\kk$-form $X$ of $V_4$ is $\kk$-rational if and only if $X(\kk) \ne \varnothing$ and $\rF_1(X)(\kk) \ne \varnothing$;
\item 
\label{main:x16}
a $\kk$-form $X$ of $X_{16}$ is $\kk$-rational if and only if $\rF_3(X)(\kk) \ne \varnothing$;
\item 
\label{main:x18}
a $\kk$-form $X$ of $X_{18}$ is $\kk$-rational if and only if $X(\kk) \ne \varnothing$ and $\rF_2(X)(\kk) \ne \varnothing$.
\end{enumerate}
\end{enumerate}
\end{theorem}

Here $\rF_d(X)$ denotes the Hilbert scheme of degree~$d$ (with respect to the ample generator of the Picard group of~$X_\bkk$) genus~0 curves on $X$.
The statement of the theorem is classical for~$\PP^3$ and~$Q^3$, and for~$V_4$ it has been proved in~\cite[Theorem~A]{BW2}.

The proof of this theorem takes up all the paper. 
Some parts are really easy.
The classical and well-known cases of forms of~$\PP^3$ and~$Q^3$ are discussed in Section~\ref{section:preliminaries}.
The criterion for quintic del Pezzo threefolds~$V_5$ can be considered as a higher-dimensional version of Enriques' theorem, 
and is deduced from it in Section~\ref{section:del-pezzo}, see Theorem~\ref{theorem:v5}.
The sufficiency of the rationality and unirationality conditions for quartic del Pezzo threefolds $V_4$ is also classical,
see Theorem~\ref{theorem:v4-unirational} in Section~\ref{section:del-pezzo}.
The other cases require more work.

In a subsequent paper~\cite{KP21} we discuss the question of $\kk$-rationality for geometrically rational Fano threefolds 
of geometric Picard number higher than~1.

\subsection{Rationality constructions}

The proofs of rationality and unirationality in the remaining cases of threefolds $X_{12}$, $X_{16}$, $X_{18}$, and $X_{22}$ 
is accomplished in Section~\ref{section:rationality-constructions}
after some preparations in Section~\ref{section:mmp-lemma}, 
where we remind some standard MMP results used in the paper and prove some criteria for unirationality,
and Section~\ref{section:sarkisov-links}, where we construct various Sarkisov links for these Fano threefolds.
For the reader's convenience we outline the argument of Section~\ref{section:rationality-constructions} here.
In what follows we denote by $g \in \{7,\, 9,\, 10,\, 12\}$ the genus of a prime Fano threefold $X = X_{2g - 2}$.

If $g \in \{7,12\}$ and there is a $\kk$-point $x \in X$ which is \emph{sufficiently general}, i.e., does not lie on any line, 
then the Sarkisov link with center at~$x$ (Theorem~\ref{th:sl:point}) 
transforms $X$ to a quintic del Pezzo threefold (for $g = 7$), or to $\PP^3$ (for $g = 12$), so rationality of~$X$ follows.

Similarly, for $g = 10$ the same link transforms $X$ to a sextic del Pezzo fibration with a rational 3-section.
If, moreover, the point $x$ is in a general position with respect to the conic corresponding to a $\kk$-point of $\rF_2(X)$,
the conic gives a 2-section of this fibration, and 
rationality of $X$ follows from~\cite[Proposition~8]{Add-Has-Tsch-VA}.

Finally, for~\mbox{$g = 9$}, let $C \subset X$ be a cubic curve defined over~$\kk$.
If $C$ is smooth and has no bisecants in~$X$, then the Sarkisov link with center at~$C$ (Theorem~\ref{th:sl:cubic})
transforms $X$ to a quintic del Pezzo threefold, so rationality follows.
If $C$ is the union of three lines meeting at a point, the double projection out of the meeting point
transforms $X$ to an intersection of three quadrics in~$\PP^6$ containing a plane,
and then the projection out of this plane gives a birational isomorphism onto~$\PP^3$ (Corollary~\ref{cor:projection-birational}).
In all other cases there exists a line on~$X$ defined over~$\kk$
and the Sarkisov link with center at it (Theorem~\ref{th:sl:line}) provides a birational isomorphism of~$X$ onto~$\PP^3$.

It is clear from the above explanation that for $g \in \{7,10,12\}$ for the proof of rationality 
it is essential to find a $\kk$-point of $X$ in a general position.
So, before proving rationality, we prove unirationality of $X$ under the assumption of existence of a $\kk$-point.
The proof uses more or less the same instruments as before:
Sarkisov link with center at a point (if the point does not lie on a line),
Sarkisov link with center at a line (if the point lies on a $\kk$-line),
Sarkisov link with center at a singular conic (if the point lies on two Galois-conjugate lines),
and the double projection from a point (if the point lies on three Galois-conjugate lines).
In the last case we use the unirationality criterion (Proposition~\ref{proposition:q-fano}) of Section~\ref{section:mmp-lemma}.

It should be mentioned that Sarkisov links of Theorems~\ref{th:sl:line}, \ref{th:sl:cubic}, and~\ref{th:sl:point} 
are well known over an algebraically closed field 
(except for the case of the Sarkisov link with center at a point and~\mbox{$g = 7$}, where we provide all necessary computations).
The case of the Sarkisov link with center at a singular conic (Theorem~\ref{th:sl:conic}) is not so standard,
so we discuss it with more details than the other links.

\subsection{Obstructions to rationality}

Finally, we explain how we prove that the conditions of rationality of Theorem~\ref{theorem:main} are necessary.
For forms of $V_5$ there is nothing to prove and for forms of $\PP^3$, $Q^3$, $X_{12}$, and $X_{22}$ the argument is standard.
So, the only interesting cases are those of varieties $V_4$, $X_{16}$, and $X_{18}$.
These cases are discussed in Section~\ref{section:obstructions}.

In fact, in this section we prove a more general necessary criterion for rationality (Theorem~\ref{theorem:obstruction-general}),
which is interesting by itself.
Roughly speaking, it tells the following.
Consider the Hilbert schemes~$\rF_d(X_\bkk)$, $d \in \ZZ_{>0}$, 
for the extension of scalars~$X_\bkk$ of~$X$ to the closure~$\bkk$ of~$\kk$,
parameterizing degree~$d$ genus~$0$ curves on~$X_\bkk$.
Assume that two particular Hilbert schemes $\rF_{\dc}(X_\bkk)$ and $\rF_{\dtor}(X_\bkk)$ among these 
have a structure of a rationally connected fibration over a curve $\Gamma_\bkk$ and of an abelian variety, respectively,
and such that the Abel--Jacobi maps induce isomorphisms 
\begin{equation}
\label{eq:intro:alb-jac}
\Alb(\rF_{\dc}(X_\bkk)) \cong \Pic^0(\Gamma_\bkk) \cong \Jac(X_\bkk) 
\qquad\text{and}\qquad
\rF_{\dtor}(X_\bkk) \cong \Jac(X_\bkk),
\end{equation}
where $\Jac(X_\bkk)$ is the intermediate Jacobian.
Assuming also a couple of technical conditions, we prove that if $X$ is $\kk$-rational, then~\mbox{$\rF_{\dtor}(X)(\kk) \ne \varnothing$}.

The proof of Theorem~\ref{theorem:obstruction-general} is based on the theory of intermediate Jacobians over non-closed fields
as developed by Benoist and Wittenberg in~\cite{BW,BW2}.
Using their results we associate to each Hilbert scheme $\rF_d(X)$ a torsor over $\Jac(X)$, the intermediate Jacobian of $X$ over~$\kk$.
Benoist and Wittenberg prove that if $X$ is $\kk$-rational, then every such torsor is isomorphic to~$\Pic^m(\Gamma)$,
a connected component of the Picard scheme of~$\Gamma$, where $\Gamma$ is a $\kk$-form of the curve $\Gamma_\bkk$.
Applying this to the Hilbert scheme of lines~$\rF_1(X)$ and using some general properties of this construction, 
we conclude that the torsor associated with~$\rF_{d}(X)$ is isomorphic to $\Pic^{md}(\Gamma)$.

On the other hand, the existence of a fibration~$\rF_{\dc}(X) \to \Gamma$ implies that 
the torsor associated with~$\rF_{\dc}(X)$ is isomorphic to $\Pic^1(\Gamma)$.
It follows that the class of the torsor~$\Pic^1(\Gamma)$ is annihilated by~\mbox{$m\dc - 1$}.
Since it is also killed by~\mbox{$2\g(\Gamma) - 2$}, it follows from a numerical assumption of the theorem that it is killed by $m\dtor$. 
Therefore, the torsor~$\rF_{\dtor}(X)$, isomorphic to~$\Pic^{m\dtor}(\Gamma)$, is trivial, hence has a $\kk$-point.

A funny feature of Theorem~\ref{theorem:obstruction-general} is that although it gives a criterion for rationality over~$\kk$,
all its conditions are formulated, and can be verified, over~$\bkk$.
In particular, one can use various results, such as Mukai's description of Fano threefolds, to check these conditions.
In fact, what we use, is the semiorthogonal decomposition of the bounded derived category $\Db(X_\bkk)$ of coherent sheaves on~$X_\bkk$.
The advantage of this approach is that besides proving abstract isomorphisms~\eqref{eq:intro:alb-jac}, 
it also shows that these isomorphisms are induced by the Abel--Jacobi maps, 
precisely as we need for Theorem~\ref{theorem:obstruction-general}.

We prove some general results in this direction in Section~\ref{section:sod-aj}; 
in particular, we show in Proposition~\ref{proposition:aj-c2-e} that the intermediate Jacobian
of a rationally connected threefold whose derived category has a semiorthogonal decomposition
with several exceptional objects and the derived category of a curve as components
is isomorphic (not just isogenous!) to the Jacobian of that curve.

At the first glance it may look that the conditions of Theorem~\ref{theorem:obstruction-general} are far too strong to be satisfied in practice.
However, it turns out that all these conditions hold for the three cases of our interest: $V_4$, $X_{16}$, and $X_{18}$.
This is proved in Section~\ref{section:sods} (see~\S\ref{subsection:genus-9} for $X_{16}$, \S\ref{subsection:genus-10} for $X_{18}$,
and \S\ref{subsection:v4} for~$V_4$).
The descriptions of the Hilbert schemes $\rF_3(X_{16})$ and $\rF_3(X_{18})$ of cubic curves on Fano threefolds of genus~9 and~10,
obtained in Theorems~\ref{theorem:x16-f3} and~\ref{theorem:x18-f3} are completely new.

\bigskip

As it is clear from the above discussion, the paper consists of two parts: 
the construction part (Sections~\ref{section:del-pezzo}--\ref{section:rationality-constructions})
and obstruction part (Sections~\ref{section:obstructions}--\ref{section:sods})
that use completely different techniques and are almost independent of each other.
The reader interested only in one of these parts can safely ignore the other.

It should be mentioned, that similar results 
were established independently by Hassett and Tschinkel, see~\cite{HT19:v4,HT19:x18}.

\subsection{Acknowledgements}

We would like to thank many people for their help.
First, this is Olivier Wittenberg, who informed us about the results of~\cite{BW2} and explained how they can be used.
We are also happy to thank 
Fabrizio Catanese, 
Vanya Cheltsov, 
Jean-Louis Colliot-Th\'el\`ene,
Sergey Gorchinskiy,
Christian Liedtke,
Dmitri Orlov,
Costya Shramov 
and other people for useful communications
that helped to improve our results.
We are also grateful to the anonymous referee for useful comments.

Finally, we would like to thank Scuola Internazionale Superiore di Studi Avanzati (Trieste)
and Max Planck Institute for Mathematics (Bonn),
where some parts of this work were accomplished.

\section{Preliminaries}
\label{section:preliminaries}

We work over a field $\kk$ of zero characteristic.
We denote by $\Gal$ the Galois group of $\bkk$ over $\kk$.

\subsection{Standard results}

We will frequently use the following standard facts.

\begin{lemma}[{\cite{nishimura-55}}]
\label{lemma:points}
Let $\psi \colon X \dashrightarrow Y$ be a rational map from a smooth variety $X$ to a proper variety~$Y$.
If $X(\kk) \ne \varnothing$ then $Y(\kk) \ne \varnothing$.
\end{lemma}

\begin{lemma}
\label{lemma:unirationality-point}
If $X$ is $\kk$-unirational, then $X(\kk) \ne \varnothing$ and the set $X(\kk)$ is Zariski dense in~$X$.
\end{lemma}
\begin{proof}
By definition of $\kk$-unirationality there is a dominant map $\PP^N_\kk \dashrightarrow X$.
Therefore, there is a Zariski open subset $U \subset \PP^N_\kk$ and a regular map $U \to X$ with Zariski dense image.
Since $\kk$-points are dense in $\PP^N_\kk$, they are dense in~$U$, and so they are dense in~$X$.
\end{proof}

The following result is also well known.

\begin{lemma}
\label{lemma:divisors}
Let $X$ be a proper $\kk$-variety with $X(\kk) \ne \varnothing$. 
Then the natural map 
\begin{equation*}
\Pic(X) \larrow \Pic(X_\bkk)^\Gal 
\end{equation*} 
is an isomorphism.
In particular, if $X$ is projective, $X(\kk) \ne \varnothing$, and $\Pic(X_\bkk) \cong \ZZ$, 
then every Cartier divisor class of~$X_\bkk$ is defined over~$\kk$.
\end{lemma}
\begin{proof}
The first part is~\cite[Ch.~4, Proposition~12]{C58}.
For the second part note that the Galois action on $\Pic(X_\bkk)$ is trivial since it preserves ampleness, 
hence $\Pic(X_\bkk)^\Gal = \Pic(X_\bkk)$.
\end{proof}

The following useful result can be extracted from~\cite[proof of Theorem~3]{Lichtenbaum1968} and~\cite[proof of Theorem~7]{Lichtenbaum1969};
we provide below a proof for completeness.

\begin{lemma}
\label{lemma:h-gal}
Assume $X$ is projective. 
If $D \in \Pic(X_\bkk)^\Gal$ then 
\begin{equation*}
\chi(\cO(D)) \cdot D \in \Pic(X),
\end{equation*}
where $\chi(\cO(D)) := \sum (-1)^i \dim H^i(X_\bkk,\cO_{X_\bkk}(D))$ is the Euler characteristic.
\end{lemma}

\begin{proof}
The Hochschild--Serre spectral sequence gives us the following standard exact sequence 
\begin{equation}
\label{eq:hs-sequence}
0 \larrow \Pic(X) \larrow \Pic(X_\bkk)^\Gal \xrightarrow{\ \beta\ } \Br(\kk) \larrow \Br(X) \larrow H^1(\Gal,\Pic(X_\bkk)).
\end{equation}
Let $\beta_D$ be the image of $D$ in $\Br(\kk)$.
We show below that $\chi(\cO(D)) \cdot \beta_D = 0$.

First, assume $H^{>0}(X_\bkk,\cO_{X_\bkk}(D)) = 0$, so that $\chi(\cO(D)) = \dim H^{0}(X_\bkk,\cO_{X_\bkk}(D))$, 
and~$D$ is very ample.
By~\cite[Theorem~3.4]{Liedtke2017} the class~$\beta_D$ can be represented by a Severi--Brauer variety 
that is a $\kk$-form of~$\PP(H^0(X_\bkk,\cO_{X_\bkk}(D)))$; in particular, it is annihilated by~$\chi(\cO(D))$. 

Now consider the case of arbitrary $D$.
Let $H \in \Pic(X)$ be an ample class.
Then for $m \gg 0$ the divisor class $D + mH$ satisfies the assumption of the lemma 
and $\beta_{D + mH} = \beta_D$ by~\eqref{eq:hs-sequence}, hence
\begin{equation*}
\chi(\cO(D + mH)) \cdot \beta_D = 0.
\end{equation*}
Since the above holds for all $m \gg 0$, the class $\beta_D$ is killed by
\begin{equation*}
\gcd \{ \chi(\cO(D + mH)) \mid m \gg 0 \} =
\gcd \{ \chi(\cO(D + mH)) \mid m \in \ZZ \},
\end{equation*}
hence it is killed by $\chi(\cO(D))$.

Finally, the equality $\chi(\cO(D)) \cdot \beta_D = 0$ proved above implies that the class $\beta_{\chi(\cO(D)) \cdot D} \in \Br(\kk)$ is zero,
hence the divisor class $\chi(\cO(D)) \cdot D \in \Pic(X_\bkk)^\Gal$ comes from $\Pic(X)$.
\end{proof}

\subsection{Severi--Brauer varieties and quadrics}

The following criteria for rationality of Severi--Brauer varieties and quadrics are classical.

\begin{proposition}[Severi--Brauer]
\label{proposition:sb}
Let $X$ be a $\kk$-form of $\PP^n_\bkk$ and let $H \in \Pic(\PP^n_\bkk)$ be the hyperplane class.
The following conditions are equivalent
\begin{enumerate}
\renewcommand\labelenumi{\rm (\alph{enumi})}
\renewcommand\theenumi{\rm (\alph{enumi})}
\item 
\label{proposition:sb-a}
$X(\kk) \ne \varnothing$;
\item 
\label{proposition:sb-b}
$\Pic(X) = \Pic(X_\bkk)$;
\item
\label{proposition:sb-c}
$dH \in \Pic(X)$ for some $d \in \ZZ$ such that $\gcd(d,n+1) = 1$;
\item 
\label{proposition:sb-d}
$X \cong \PP^n_\kk$.
\end{enumerate}
\end{proposition}
\begin{proof}
Implication \ref{proposition:sb-d} $\implies$ \ref{proposition:sb-a} is trivial
and by Lemma~\ref{lemma:divisors} condition~\ref{proposition:sb-a} implies~\ref{proposition:sb-b}.
Furthermore, since the canonical class $K_X = -(n+1)H$ is defined over~$\kk$, 
and since the image of~$\Pic(X)$ in~\mbox{$\Pic(X_\bkk)^\Gal = \Pic(X_\bkk)$} is a subgroup, it follows that~\ref{proposition:sb-b} is equivalent to~\ref{proposition:sb-c}.
So it remains to show that~\ref{proposition:sb-b} implies~\ref{proposition:sb-d}.

Indeed, if $H \in \Pic(X)$, then the map to the projective space $\PP(H^0(X,\cO_X(H))^\vee) \cong \PP^n$ associated with $H$ 
becomes an isomorphism after extension of scalars to~$\bkk$, hence it is an isomorphism over~$\kk$.
\end{proof}

A similar argument proves

\begin{proposition}
\label{proposition:quadric}
A form $X$ of a smooth quadric is rational if and only if $X(\kk) \ne \varnothing$. 
\end{proposition}
\begin{proof}
One direction follows from Lemma~\ref{lemma:unirationality-point}.
To prove the other, assume~\mbox{$X(\kk) \ne \varnothing$}.
The hyperplane class~$H$ of~$X_\bkk$ is Galois-invariant because $- \dim(X) H \sim K_{X_\bkk}$ and~$\Pic(X_\bkk)$ is torsion-free,
therefore, by Lemma~\ref{lemma:divisors} it is defined over~$\kk$.
The map 
\begin{equation*}
X \larrow \PP(H^0(X,\cO_X(H))^\vee) \cong \PP^{n+1}
\end{equation*}
after extension of scalars to $\bkk$ becomes the embedding of $X_\bkk$ into $\PP^{n+1}_\bkk$ as a quadric hypersurface,
hence it identifies $X$ with a quadric hypersurface in $\PP^{n+1}$.
The projection out of a $\kk$-rational point provides a birational isomorphism of $X$ and $\PP^n$.
\end{proof}

\subsection{Fano threefolds}
\label{subsection:fano}

Let $X$ be a smooth Fano threefold.
We denote by $\uprho(X)$ the Picard rank of~$X$ and by $\uprho(X_\bkk)$ its geometric Picard rank.
We define {\sf the index}~$\iota(X)$ as the maximal integer dividing the canonical class $K_{X_\bkk}$ in $\Pic(X_\bkk)$.

If $\iota(X) \ge 3$, then $X_\bkk \cong \PP^3$, or $X_\bkk \cong Q^3$ (see, e.g., \cite[Corollary~3.1.15]{IP}).

Fano threefolds $X$ with $\iota(X) = 2$ are usually called {\sf del Pezzo threefolds} (see~\cite{Fujita-book}).
We will denote by $H$ the class in $\Pic(X_\bkk)$ such that $2H \sim -K_{X_\bkk}$. 
Recall that the Picard group of a Fano variety is torsion free (\cite[Proposition~2.1.2(ii)]{IP}), 
so $H$ is defined unambiguously and is Galois-invariant. 
In particular, if $X(\kk) \ne \varnothing$, then $H$ is defined over~$\kk$ by Lemma~\ref{lemma:divisors}.
We set
\begin{equation*}
\dd(X) := H^3 = (-K_X)^3/8.
\end{equation*}
By Riemann--Roch one has (see~\cite[Ch. 1, Proposition~4.5]{Iskovskikh-1980-Anticanonical})
\begin{equation}
\label{eq:dP:dimH}
\dim H^0(X_\bkk,\cO_{X_\bkk}(H)) = \dd(X) + 2.
\end{equation}

Classification of del Pezzo threefolds is well known, in particular we have

\begin{theorem}[{\cite[Ch. 2, Theorem~1.1]{Iskovskikh-1980-Anticanonical}, \cite[Theorem~3.3.1]{IP}, \cite[Theorem~8.11]{Fujita-book}}]
\label{th:del-pezzo}
If $X$ is a smooth del Pezzo threefold with $\uprho(X_\bkk) = 1$ then~$1 \le \dd(X) \le 5$.
\end{theorem}

We will discuss here two types of del Pezzo threefolds:
\begin{itemize}
\item 
{\sf quintic del Pezzo threefolds}, i.e., del Pezzo threefolds $X$ with $\uprho(X_\bkk) = 1$ and $\dd(X) = 5$;
over an algebraically closed field any such threefold is isomorphic to a codimension~3 linear section of $\Gr(2,5)$ 
(see, e.g., \cite[Theorem~3.3.1]{IP});
and
\item 
{\sf quartic del Pezzo threefolds}, i.e., del Pezzo threefolds $X$ with $\uprho(X_\bkk) = 1$ and $\dd(X) = 4$;
over an algebraically closed field any such threefold is isomorphic to a complete intersection of two quadrics in~$\PP^5$ 
(see, e.g., \cite[Theorem~3.3.1]{IP}).
\end{itemize}

If a Fano threefold~$X$ satisfies~$\uprho(X_\bkk) = \iota(X) = 1$, i.e., if~$\Pic(X_\bkk) = \ZZ \cdot K_{X_\bkk}$, it is called {\sf prime}.
Note that the canonical class~$K_X$ of any variety is defined over the ground field~$\kk$.
The {\sf genus}~$\g(X)$ of a Fano threefold~$X$ is defined as
\begin{equation}
\label{def:genus}
\g(X) := \dim H^0(X,\cO_X(-K_X)) - 2 = 1 - K_X^3/2
\end{equation}
(the equality follows from Riemann--Roch).

We will frequently use the following classical result:

\begin{theorem}[{\cite[Ch.~1, Theorem~6.3; Ch.~2, Theorems~2.2 and~3.4]{Iskovskikh-1980-Anticanonical}}]
\label{th:bht}
Let $X$ be a Fano threefold with $\iota(X)=\uprho(X)=1$.
\begin{enumerate}
\item 
If $g := \g(X) \ge 5$ then the anicanonical class $-K_X$ is very ample and the anicanonical image 
\begin{equation*}
X = X_{2g-2} \subset \PP^{g+1}
\end{equation*}
is an intersection of quadrics \textup(as a scheme\textup). 
\item 
If additionally $\uprho(X_\bkk) = 1$ then $X_\bkk$ contains no surfaces of degree less than~$2g - 2$. 
\end{enumerate}
\end{theorem}

In this paper we will discuss threefolds $X_{12}$, $X_{16}$, $X_{18}$, and~$X_{22}$ of genus 7, 9, 10, and~12, respectively.
Over an algebraically closed field they
can be represented as zero loci of equivariant vector bundles on homogeneous spaces of simple algebraic groups, see~\cite{Mukai92}.

\subsection{Hilbert schemes}
\label{subsection:Hilbert}

If $X \subset \PP^n$ is a projective variety we denote by~$\rF_d(X)$ 
the Hilbert scheme of curves in~$X$ with Hilbert polynomial $h_d(t) = dt + 1$.
Recall that formation of Hilbert schemes commutes with base change, in particular
with extension of scalars, so that
\begin{equation*}
\rF_d(X_\bkk) \cong \rF_d(X)_\bkk.
\end{equation*}
Similarly, if $x \in X$ is a point, we denote by $\rF_d(X,x) \subset \rF_d(X)$ the subscheme of curves that pass through~$x$;
this is the fiber over~$x$ of the universal curve over~$\rF_d(X)$.

When $d = 1$ or $d = 2$ every curve parameterized by $\rF_d(X)$ 
is a line or conic in~$\PP^n$ lying on~$X$, respectively, (see~\cite[Lemma~2.1.1]{KPS}).
We will also need the following classification result.

\begin{lemma}
\label{lemma:f3}
Assume $\kk$ is algebraically closed.
Let $X \subset \PP^n$ be an intersection of quadrics which contains no planes 
and let $C \subset X$ be a curve that corresponds to a point of $\rF_3(X)$.
Then either 
\begin{enumerate}
\item 
\label{case:f3:smooth}
$C$ is a smooth rational twisted cubic curve, or
\item 
\label{case:f3:2-1}
$C$ is the union of a smooth conic and a line, or
\item 
\label{case:f3:chain}
$C$ is a chain of three lines, or
\item 
\label{case:f3:trident}
$C$ is the union of three lines that meet in a point but do not lie in the same plane, or
\item 
\label{case:f3:multiple}
$C$ is a purely one-dimensional curve that contain a unique multiple line as a component.
\end{enumerate}
\end{lemma}
\begin{proof}
See~\cite[\S1]{LLSS} and take into account that $X$ contains no plane cubic curves.
\end{proof}

The following easy observation allows to simplify the criterion of Theorem~\ref{theorem:main}\ref{main:x16}.

\begin{lemma}
\label{lemma:f3-point}
If $X \subset \PP^n$ is a projective variety such that either $\rF_1(X)(\kk) \ne \varnothing$ or $\rF_3(X)(\kk) \ne \varnothing$, 
then $X(\kk) \ne \varnothing$.
\end{lemma}

\begin{proof}
If $L$ is a line defined over~$\kk$, its intersection with a hyperplane in $\PP^n$ is a $\kk$-point on~$X$.

Now let $C \subset X$ be a curve corresponding to a $\kk$-point of $\rF_3(X)$.
First, assume that~$C_\bkk$ is one of the curves listed in Lemma~\ref{lemma:f3}.
If $C_\bkk$ is of types~\ref{case:f3:2-1}, \ref{case:f3:chain}, or~\ref{case:f3:multiple}, 
one of its line-components is Galois invariant, hence defined over~$\kk$, hence $X$ has a $\kk$-point by the first part of the lemma.
Similarly, if $C_\bkk$ is of type~\ref{case:f3:smooth}, then $C_\bkk \cong \PP^1_\bkk$ 
and the restriction of the hyperplane class of $\PP^n$ to~$C$ has degree~3, 
hence $C \cong \PP^1$ by Proposition~\ref{proposition:sb}, hence $C$ has $\kk$-points, hence $X(\kk) \ne \varnothing$.
If $C$ is of type~\ref{case:f3:trident}, the meeting point of the components of~$C$ is Galois-invariant, hence is defined over~$\kk$.

Otherwise, $C_\bkk$ is a plane cubic curve with an extra (possibly embedded) point (see~\cite{PS85} and~\cite[\S1]{LLSS}).
In this case the extra point is Galois-invariant, hence defined over~$\kk$.
\end{proof}

\subsection{Veronese surfaces and their projections}

Let $U$ be a vector space of dimension~3.
Let~\mbox{$\rR \subset \PP(S^2U) \cong \PP^5$} be {\sf the Veronese surface},
i.e., the image of the double Veronese embedding~$\PP(U) \to \PP(S^2U)$.
Recall that the group $\SL(U)$ acts on the space $\PP(S^2U)$ with three orbits
\begin{equation*}
\rR,
\qquad 
\rD \setminus \rR,
\qquad 
\PP(S^2U) \setminus \rD,
\end{equation*}
where $\rD \subset \PP(S^2U)$ is the symmetric determinantal cubic hypersurface (the secant variety of~$\rR$).

Recall from~\cite[\S13.4]{FH} that $\rR$ is an intersection of quadrics, 
and the space of quadratic forms vanishing on~$\rR$ is naturally identified with the space
\begin{equation*}
\Ker\left(S^2(S^2U^\vee) \larrow S^4U^\vee\right) \cong S^2U
\end{equation*}
as representations of the group~$\SL(U)$.
Thus, quadrics through~$\rR$ in~$\PP(S^2U)$ are parameterized by the same space~$\PP(S^2U)$.
To avoid confusion, we will use notation~$S^2U'$ for the space of quadratic forms vanishing on~$\rR$ and 
\begin{equation*}
\rR' \subset \rD' \subset \PP(S^2U')
\end{equation*}
for the corresponding $\SL(U)$-orbit stratification of the space of quadrics through~$\rR$.
Using this stratification it is easy to check the following facts about quadrics passing through $\rR$.

\begin{lemma}[{\cite[\S13.4]{FH}}]
\label{lemma:veronese-quadrics}
Let $Q$ be a quadric in $\PP(S^2U)$ containing~$\rR$.
\begin{enumerate}
\item 
If $Q$ corresponds to a point of $\PP(S^2U') \setminus \rD'$ then $Q$ is nonsingular.
\item
If $Q$ corresponds to a point $[u_1u_2] \in \rD' \setminus \rR'$, $u_1,u_2 \in U$, $u_1 \ne u_2$, 
then $Q$ has corank~$2$ and 
\begin{equation*}
\Sing(Q) = \PP(\langle u_1^2,u_2^2 \rangle) \subset \rD \subset \PP(S^2U),
\qquad 
\Sing(Q) \cap \rR = \{[u_1^2], [u_2^2]\}.
\end{equation*}
\item 
If $Q$ corresponds to a point $[u^2] \in \rR'$, $u \in U$, 
then $Q$ has corank~$3$ and 
\begin{equation*}
\Sing(Q) = \PP(u \cdot U) \subset \rD \subset \PP(S^2U),
\qquad 
\Sing(Q) \cap \rR = \{[u^2]\}.
\end{equation*}
\end{enumerate}
In particular the singular loci of all singular quadrics passing through~$\rR$ sweep the hypersurface~$\rD$,
and for any singular quadric~$Q$ containing~$\rR$ the intersection $\Sing(Q) \cap \rR$ is finite.
\end{lemma}

Using Lemma~\ref{lemma:veronese-quadrics} we deduce the following result that will be used in Section~\ref{section:sarkisov-links}.

\begin{lemma}\label{lemma:Veronese2}
Let~$Y \subset \PP^6$ be a reduced and irreducible complete intersection of three quadrics.
If~$Y$ contains a Veronese surface~$\rR \subset \PP^5\subset \PP^6$ then its hyperplane section~$Y \cap \langle \rR \rangle$ 
by the linear span of~$\rR$ cannot contain~$\rR$ with multiplicity~$2$.
\end{lemma}

\begin{proof}
We can identify the linear span $\langle \rR \rangle$ of $\rR$ with the space $\PP(S^2U)$.
Let~$N$ be the 3-dimensional space of quadratic forms on~$\PP^6$ vanishing on~$Y$.
Since~$Y$ is a reduced and irreducible threefold, none of the quadratic forms in~$N$ vanishes on~$\PP(S^2U)$, 
hence the restriction defines an embedding 
\begin{equation*}
N \hookrightarrow S^2U'
\subset S^2(S^2U^\vee)
\end{equation*}
of~$N$ into the space of quadratic forms on~$\PP(S^2U)$ vanishing on~$\rR$.
From now on we will identify~$\PP(N)$ with a net of quadrics in~$\PP(S^2U)$ containing~$\rR$.

If the hyperplane section $Y \cap \PP(S^2U)$ contains~$\rR$ with multiplicity~2, 
then the intersection of this net of quadrics in~$\PP(S^2U)$ 
has tangent space of dimension at least~4 at each point of~$\rR$.
Therefore, for every point of $\rR$ there is a quadric in~$\PP(N)$ singular at this point.
In particular, the set of pairs~$(Q,y)$, where~$Q$ is a quadric from the net~$\PP(N)$ and~$y \in \Sing(Q) \cap \rR$, is at least 2-dimensional,
hence by Lemma~\ref{lemma:veronese-quadrics} the set of singular quadrics in~$\PP(N)$ is at least 2-dimensional, 
hence any quadric in~$\PP(N)$ is singular.
On the other hand, by Lemma~\ref{lemma:veronese-quadrics}
every singular quadric containing $\rR$ corresponds to a point of the hypersurface~$\rD'$, 
so it follows that~$\PP(N) \subset \rD'$.

Now note, that the hypersurface~$\rD'$ contains two types of projective planes:
\begin{itemize}
\item 
planes of the form~$\PP(S^2U'_2) \subset \PP(S^2U')$, where $U'_2 \subset U'$ is a 2-dimensional subspace;
\item 
planes of the form~$\PP(U'_1 \cdot U') \subset \PP(S^2U')$, where $U'_1 \subset U'$ is a 1-dimensional subspace.
\end{itemize}
It is easy to check that the intersection of quadrics parameterized by the plane~$\PP(S^2U'_2)$
is the union of~$\rR$ with the infinitesimal neighborhood of a plane;
in particular, its multiplicity along~$\rR$ is 1.
Similarly, the intersection of quadrics parameterized by the plane $\PP(U'_1 \cdot U')$
is a cone over a cubic surface scroll;
in particular, it is three-dimensional and of degree~3, 
so it cannot be a hyperplane section of a reduced and irreducible complete intersection~$Y$ of three quadrics.
\end{proof}

Another fact that we need is the following:

\begin{lemma}
\label{lemma:Veronese0}
Let $S \subset \PP^4$ be a regular linear projection of the Veronese surface $\rR \subset \PP(S^2U)$. 
Any quadric containing~$S$ is singular, and if $S$ is smooth it is not contained in a quadric.
\end{lemma}
\begin{proof}
Let $\upsilon \in \PP(S^2U)$ be the center of the projection $\PP(S^2U) \dashrightarrow \PP^4$.
The preimage of a quadric containing the image $S$ of the Veronese surface $\rR$ is a singular quadric in $\PP(S^2U)$ containing $\rR$.
By Lemma~\ref{lemma:veronese-quadrics} this is a quadric of corank at least 2, 
hence the corresponding quadric in~$\PP^4$ has corank at least~1.
In particular, it is singular.
Furthermore, by Lemma~\ref{lemma:veronese-quadrics} in this case~$\upsilon$ lies 
on the secant variety~$\rD$ of~$\rR$, hence~$S$ is singular.
\end{proof}

Finally, we will need the following observation:

\begin{lemma}[{\cite[Ch. VII, \S 3.2]{Semple-Roth-1949}}]
\label{lemma:Veronese1}
If $S \subset \PP^k$, $k \in \{3,4\}$, is a regular linear projection of the Veronese surface $\rR \subset \PP(S^2U)$, 
then~$\Sing(S)$ is a union of a finite number of lines.
\end{lemma}

\section{Rationality of del Pezzo varieties}
\label{section:del-pezzo}

\subsection{Quintic del Pezzo varieties}

In this section we consider quintic del Pezzo threefolds, see~\S\ref{subsection:fano}.
The following observation is very useful:

\begin{lemma}
\label{lemma:dp5-h}
If $X$ is a quintic del Pezzo threefold
then the class $H$ is defined over~$\kk$.
\end{lemma}

\begin{proof}
As we explained in \S\ref{subsection:fano}, we have $H \in \Pic(X_\bkk)^\Gal$.
By~\eqref{eq:dP:dimH} 
we have $\chi(\cO_{X_\bkk}(H)) = 7$.
Therefore, Lemma~\ref{lemma:h-gal} shows that $7H \in \Pic(X)$.
On the other hand, since $K_{X_\bkk} = -2H$ is defined over~$\kk$, it follows 
that $2H \in \Pic(X)$.
Since $\gcd(7,2) = 1$, we conclude that $H \in \Pic(X)$.
\end{proof}

Recall the classical Enriques theorem.

\begin{theorem}[\cite{SB92}]
\label{theorem:dp5s}
If $X$ is a smooth quintic del Pezzo surface 
then $X$ is $\kk$-rational.
\end{theorem} 

Now we deduce from it the proof of Theorem~\ref{theorem:main}\ref{main:v5}.

\begin{theorem}
\label{theorem:v5}
Let $X$ be a smooth quintic del Pezzo threefold.
Then $X$ is $\kk$-rational.
\end{theorem}
\begin{proof}
By Lemma~\ref{lemma:dp5-h} the class $H$ is defined over $\kk$ and by~\cite[Ch. 2, Theorem~1.1(ii)]{Iskovskikh-1980-Anticanonical}
it induces an embedding 
\begin{equation*}
X \hookrightarrow \PP^6.
\end{equation*}
Consider a general projective subspace $\PP^4 \subset \PP^6$ defined over~$\kk$.
Restricting to $X$ the linear projection~$\PP^6 \dashrightarrow \PP^1$ with center in $\PP^4$,
we see that $X$ is birational to a quintic del Pezzo surface fibration over $\PP^1$,
hence by Theorem~\ref{theorem:dp5s} it is rational over $\PP^1$, hence is $\kk$-rational.
\end{proof}

\begin{remark}
One can also define a quintic del Pezzo variety of any dimension~$n \le 6$ 
as a smooth Fano variety $X$ with $\uprho(X_\bkk) = 1$, $\iota(X) = n - 1$, and $(-K_X)^n/(n-1)^n = 5$.
If $2 \le n \le 5$ the argument of Lemma~\ref{lemma:dp5-h} shows that the ample generator $H$ of $\Pic(X_\bkk)$ is defined over~$\kk$
and the argument of Theorem~\ref{theorem:v5} shows that $X$ is $\kk$-rational.

Note that both results fail for $n = 6$ --- it is easy to construct a form of $\Gr(2,5)$ 
for which the Pl\"ucker class is not defined over~$\kk$,
and which has no $\kk$-points, hence is not $\kk$-rational.
\end{remark}

\subsection{Quartic del Pezzo threefolds}

Let $X$ be a quartic del Pezzo threefold, i.e., an intersection of two quadrics in a Severi--Brauer variety of dimension~5.
The following rationality result is very classical.

\begin{theorem}[{\cite[Proposition~2.2, 2.3]{CTSSD}}]
\label{theorem:v4-unirational}
Let $X$ be a smooth quartic del Pezzo threefold.
If~$X(\kk) \ne \varnothing$ then~$X$ is $\kk$-unirational 
and if $X(\kk) \ne \varnothing$ and~$\rF_1(X)(\kk) \ne \varnothing$ then~$X$ is $\kk$-rational.
\end{theorem}

This proves that the conditions of (uni)rationality of $X$ from Theorem~\ref{theorem:main}\ref{main:v4} are sufficient.

\section{Results from Minimal Model Program}
\label{section:mmp-lemma}

In this section we remind some results from MMP and prove a couple of unirationality criteria.
The first of them (Lemma~\ref{lemma:unirational-fibration}) is quite standard.
The second (Proposition~\ref{proposition:q-fano}) is a bit technical 
and will be only used at the end of \S~\ref{subsection:unirationality}, see the proof of Proposition~\ref{proposition:f1xx-3}.

\subsection{Terminal singularities}
\label{subsec:terminal}

In this subsection we remind some basic definitions and results of the Minimal Model Program, for more details see e.g. 
\cite{Kollar-Mori:88}, \cite{Mori-1988}, \cite{Kollar95:pairs}.

For a morphism $f \colon Y' \to Y$ of normal varieties and a $\QQ$-Cartier divisor class~$D$ on~$Y$ 
we denote by~$f^*D$ its pullback to~$Y'$.
Similarly, for a rational map $\psi \colon Y' \dashrightarrow Y$ and a subvariety $Z \subset Y'$ 
not lying in the indeterminacy locus of~$\psi$ we write $\psi_*(Z) \subset Y$ for the strict transform of~$Z$.

For a normal variety $Y$ we denote by $K_Y$ the canonical class (in general it is a Weil divisor class).
Recall that $Y$ {\sf has terminal singularities} if it is normal, $\QQ$-Gorenstein 
(i.e., $K_Y$ is $\QQ$-Cartier)
and for any birational morphism~\mbox{$\pi \colon \tY \to Y$} from a normal variety one has 
\begin{equation}
\label{eq:discrepancy}
K_\tY = \pi^*K_Y + \sum a_iE_i
\end{equation}
with $a_i > 0$ for all $i$, where $E_i$ are the exceptional divisors of $\pi$.
The coefficients $a_i$ are called {\sf discrepancies} of the divisors~$E_i$ over~$Y$.
The discrepancy of a divisor $E$ over~$Y$ only depends 
on the valuation of the field~$\kk(Y)$ of rational functions on~$Y$ defined by~$E$.

We will need the following standard results about terminal singularities.

\begin{lemma}[{\cite[Lemma~3.10]{Kollar95:pairs}}]
\label{lemma:small-terminality}
If $f \colon Y' \to Y$ is a small proper birational morphism of normal $\QQ$-Gorenstein varieties 
then~$Y$ has terminal singularities if and only if~$Y'$ has.
\end{lemma}

\begin{lemma}[{\cite[Corollaries~5.18, 5.38]{Kollar-Mori:88}}]
\label{lemma:terminal-singularities}
If $Y$ is a threefold with terminal singularities then the 
set~$\Sing(Y)$ is finite.
If, moreover, $Y$ is Gorenstein, then for any point~$y \in \Sing(Y)$ we have~\mbox{$\dim T_yY \le 4$}.
\end{lemma}

\begin{lemma}[{\cite[Lemma~5.1, Corollary~5.2]{Kawamata88-crep}}]
\label{lemma:factoriality}
Let~$Y$ be a threefold with terminal singularities and let~$D$ be any Weil $\QQ$-Cartier divisor.
If~$rK_X$ is Cartier for some integer~$r$ then~$rD$ is Cartier as well.
In particular, if the singularities of~$Y$ are 
Gorenstein then any Weil $\QQ$-Cartier divisor on~$Y$ is Cartier.
\end{lemma}

\begin{lemma}[{\cite[Corollary~4.5]{Kawamata88-crep}}]
\label{lemma:factorialization}
If~$Y$ is a threefold with
terminal singularities then there exists a small birational morphism~\mbox{$\tau \colon \tY \to Y$} 
such that $\tY$ has $\QQ$-factorial terminal singularities.
\end{lemma}

Let $Y'$ be a variety with $\QQ$-factorial terminal singularities.
Recall that an {\sf extremal Mori contraction} of~$Y'$ is a proper morphism~$f \colon Y' \to Y$ with connected fibers
such that~$Y$ is normal, the relative Picard number is~$\uprho(Y'/Y) = 1$, and~$-K_{Y'}$ is~$f$-ample.

We need the following result which is known over an algebraically closed ground field~\cite{Cutkosky-1988}
(for more details see~\cite[Theorems~7.1.1, 8.2.4, and~10.2 and Proposition~9.1]{P21}).

\begin{theorem}
\label{theorem:smooth-contractions}
Let $Y'$ be a projective threefold with terminal $\QQ$-factorial Gorenstein singularities and let $f \colon Y' \to Y$
be an extremal Mori
contraction.
Then one of the following possibilities takes place:
\begin{enumerate}
\item 
\label{theorem:smooth-contractions:divisorial:point}
The map $f$ contracts a $\kk$-irreducible divisor $D$ onto a $\kk$-irreducible $0$-dimensional subvariety of~$Y$;
the discrepancy of~$D$ takes values in the set $\left\{\tfrac12,1,2\right\}$.
\item 
\label{theorem:smooth-contractions:divisorial:curve}
The map $f$ is the blowup of a $\kk$-irreducible locally complete intersection curve $\Gamma$ lying in the smooth locus of~$Y$;
the discrepancy of the exceptional divisor~$D$ equals~$1$.
\item 
\label{theorem:smooth-contractions:conic-bundle}
The map $f$ is a flat generically smooth conic bundle over a smooth surface $Y$.
\item 
\label{theorem:smooth-contractions:dP}
The map $f$ is a flat generically smooth del Pezzo fibration over a smooth curve $Y$;
every fiber of~$f$ is $\kk$-irreducible and reduced.
\item 
\label{theorem:smooth-contractions:Fano}
$Y$ is a point and $Y'$ is a Fano threefold with $\uprho(Y') =1$.
\end{enumerate}
\end{theorem}

\begin{remark}
\label{remark:connected}
In cases~\ref{theorem:smooth-contractions:divisorial:point} and~\ref{theorem:smooth-contractions:divisorial:curve}
if $\uprho(Y'_\bkk) - \uprho(Y_\bkk) = 1$ then $D$ is connected and $D_{\bkk}$ is irreducible~\cite[Theorem~4]{Cutkosky-1988}.
\end{remark}

\begin{proof}
First, if $Y$ is a point, then $Y'$ is a Fano variety with $\uprho(Y')=1$ because $f$ is extremal.
This proves~\ref{theorem:smooth-contractions:Fano}.

Assume $Y$ is a curve. 
Recall that $Y$ is normal, hence smooth.
Since $Y'$ has isolated singularities, a general fiber of $f$ is smooth, 
and since $-K_{Y'}$ is relatively ample, it is a smooth del Pezzo surface.
Let $D \subset f^{-1}(y)$ be a $\kk$-irreducible component 
with reduced structure of a fiber of~$f$.
Since $Y'$ is locally factorial (Lemma~\ref{lemma:factoriality}), $D$ is a Cartier divisor.
Since~$f$ is extremal, $D = f^{-1}(y)$ (see~\cite[Theorem~3.7(4)]{Kollar-Mori:88}).
Therefore, every fiber of $f$ is $\kk$-irreducible and reduced.
This proves~\ref{theorem:smooth-contractions:dP}.

Next, assume that~$Y$ is a surface. 
Let us show that~$Y$ is smooth;
in fact, this is proved in~\cite[Theorem~7]{Cutkosky-1988}, but this part of the argument is very vague,
so we provide an alternative argument for the reader's convenience.

Assume that~$Y$ has a singular point, say~$y$.
Shrinking~$Y$ if necessary, we can find a smooth anticanonical divisor $S \subset Y'$ 
which does not contain fibers of~$f$, see~\cite[Proposition~1]{Cutkosky-1988}. 
Then the restriction $f_S \colon S \to Y$ is a finite morphism of degree~2.
Therefore, \mbox{$f_*\cO_S \cong \cO_Y \oplus \cL$},
where~$\cL$ is a reflexive sheaf of rank~1. 
But~$Y$ is locally factorial by~\cite[Lemma~3]{Cutkosky-1988}, hence~$\cL$ is locally free,
and hence the morphism~$f_S$ is flat.
Since~$S$ is smooth, it follows from~\cite[Proposition~6.5.2(i)]{EGA} that~$Y$ is smooth as well.
The rest of the argument just repeats that of~\cite[Theorem~7]{Cutkosky-1988}.
This proves \ref{theorem:smooth-contractions:conic-bundle}.

Finally, assume that the contraction $f$ is birational.
Then the exceptional locus~$D$ of~$f$ is a divisor;
indeed, if the exceptional locus is the union of a finite number of curves, the contraction is minimal in the sense of~\cite[Definition~0.4]{Benveniste85},
which by~\cite[Theorem~0(ii)]{Benveniste85} implies that~$Y$ is not Gorenstein, thus contradicting the assumptions of the theorem.
Since $Y'$ is locally factorial and the contraction is extremal, $D$ is $\kk$-irreducible. 
Then $Y$ has only ($\QQ$-factorial) terminal singularities (see e.g.~\cite[Corollary~3.43]{Kollar-Mori:88}).

Assume that $f(D)$ is a curve. 
Then all the fibers over points in $f(D)$ are one-dimensional.
In this case, the proof of~\cite[Theorem~4]{Cutkosky-1988} works without any changes.
This proves~\ref{theorem:smooth-contractions:divisorial:curve}.

It remains to consider the case when $f(D)$ is 0-dimensional.
As $\kk$-irreducibility of $f(D)$ follows from that of~$D$, we only need to compute the discrepancy.

Consider the morphism~\mbox{$f_\bkk : Y'_\bkk \to Y_\bkk$}. 
If $D_\bkk$ is irreducible, then $f_\bkk$ is a birational extremal Mori contraction, hence~\cite[Theorem~5]{Cutkosky-1988} applies.
So, from now on assume that $D_\bkk$ is not irreducible.

If $a$ is the discrepancy of~$D$ then $K_{Y'_\bkk} = f^* K_{ Y_\bkk}+a D_\bkk$.
Let $\tau \colon \tY_\bkk\to Y'_\bkk$ be a $\QQ$-factorialization (Lemma~\ref{lemma:factorialization})
and let~$\tD_\bkk = (\tau^{-1})_*(D_\bkk)$ be the strict transform of $D_\bkk$. 
Note that $\tau$ is small, hence
\begin{equation}
\label{eq:ktybkk}
K_{\tY_\bkk} =\tau^*f^* K_{ Y_\bkk}+a \tD_\bkk,
\end{equation}
Since $-K_{\tY_\bkk} = \tau^*(-K_{Y'_\bkk})$ is nef and not trivial over $Y_\bkk$,
there exists a $K_{\tY_\bkk}$-negative extremal ray~$R$, trivial over~$Y_\bkk$.
By~\cite[Theorem~6.2(i)]{Mori-1988} the contraction $\gamma \colon \tY_\bkk \to Y''_\bkk$ of $R$ fits into a diagram
\begin{equation*}
\xymatrix@R=10pt{
& \tY_\bkk \ar[dl]_\tau \ar[dr]^\gamma 
\\
Y'_\bkk \ar[dr]_f &&
Y''_\bkk \ar[dl] 
\\
& Y_\bkk
}
\end{equation*}
and must be a divisorial contraction.
Its exceptional divisor is, therefore, a component $D^{0}_\bkk\subset \tD_\bkk$. 
By~\eqref{eq:ktybkk} we have 
\begin{equation}
\label{eq:ktybkk-r}
K_{\tY_\bkk}\cdot R= a \left( D^{0}_\bkk\cdot R + D^{1}_\bkk\cdot R \right ),
\end{equation}
where $D^{1}_\bkk := \tD_\bkk - D^{0}_\bkk$.
According to~\cite[Theorem~5]{Cutkosky-1988} we have the following possibilities:
\begin{enumerate}
\renewcommand\labelenumi{\rm \alph{enumi})}
\renewcommand\theenumi{\rm \alph{enumi})}
\item \label{ea}
$\gamma$ is the blowup of a smooth point; then $K_{\tY_\bkk}\cdot R=-2$ and $D^{0}_\bkk\cdot R=-1$;
\item \label{eb}
$\gamma$ is the blowup of a Gorenstein terminal point or a curve; 
then $K_{\tY_\bkk}\cdot R=-1$, $D^{0}_\bkk\cdot R=-1$;
\item \label{ec}
$\gamma$ is the blowup of a non-Gorenstein terminal point 
then $K_{\tY_\bkk}\cdot R = -1$, $D^{0}_\bkk\cdot R=-2$.
\end{enumerate}
Since $a > 0$ (by terminality), we have $D^{0}_\bkk\cdot R + D^{1}_\bkk \cdot R < 0$ in all these cases.
On the other hand, $D^{1}_\bkk \cdot R$ is a nonnegative integer 
(because $D^{1}_\bkk$ is a Cartier divisor by Lemma~\ref{lemma:factoriality}
and the curves in the ray $R$ sweep out $D^{0}_\bkk$), 
hence it must be equal to~0 in cases~\ref{ea} and~\ref{eb}, and to~$0$ or~$1$ in case~\ref{ec}.
Using~\eqref{eq:ktybkk-r} we compute $a = 2$ in case~\ref{ea}, $a = 1$ in case~\ref{eb}, and $a \in \{\tfrac12, 1\}$ in case~\ref{ec}.
\end{proof}

The following notion will be useful later (see Theorem~\ref{th:sl:conic}\ref{th:sl:conic:g=10} and Proposition~\ref{proposition:f1xx-2}).

\begin{definition}
\label{def:inf-rat}
We will say that a $\kk$-point $y \in Y$ is {\sf infinitesimally $\kk$-rational} 
if the exceptional divisor of the blowup of $Y$ at $y$ is a $\kk$-rational variety.
\end{definition}

A smooth $\kk$-point is always infinitesimally $\kk$-rational.
Moreover, an ordinary double point is infinitesimally $\kk$-rational if there is a smooth $\kk$-curve passing through it.

\begin{lemma}
\label{lemma:flop-odp}
Assume $Y$ is a threefold with a unique singular point $y$, which is an ordinary double point,
and let $Y\dashrightarrow Y'$ be a flop.
Then $Y'$ has a unique singular point $y'$ which is also an ordinary double point.
Furthermore, let $Q$ and $Q'$ be the exceptional divisors of the blowups~\mbox{$\mu \colon \tY \to Y$} and~\mbox{$\mu' \colon \tY'\to Y'$} 
at $y$ and $y'$ respectively.
Then $Q$ is birational to $Q'$; in particular, $y'$ is infinitesimally rational if and only if $y$ is.
\end{lemma}

\begin{proof}
Consider a common resolution $\hY$ of $\tY$ and $\tY'$.
We have a commutative diagram:
\begin{equation*}
\xymatrix@R=1.5em{
& \hQ \ar@{^{(}->}[r] \ar[dl] &
\hY \ar[dr]^{\theta'} \ar[dl]_{\theta} &
\hQ' \ar@{_{(}->}[l] \ar[dr] &
\\
Q \ar@{^{(}->}[r] &
\tY \ar[d]_{\mu} &&
\tY'\ar[d]^{\mu'} &
Q' \ar@{_{(}->}[l] &
\\
&
Y\ar[dr]_\pi \ar@{<-->}[rr] && 
Y'\ar[dl]^{\pi'}
\\
&& \barY,
}
\end{equation*}
where $\pi$ and $\pi'$ are small crepant contractions of $Y$ and $Y'$, and the dashed arrow is the flop.
By~\cite[Theorem~3.25.5]{Kollar-Mori:88} the canonical divisor $K_{\bar Y}$ of $\bar Y$ is Cartier. 
The singularity of $Y$ and $Y'$ are the same by~\cite[Theorem~2.4]{Kollar1989a},
hence $Y'$ has a single ordinary double point $y'$.
Moreover, we have~$\pi(y) = \pi'(y')$ and we denote this point by~$\bar{y} \in \barY$.

Let $\hQ \subset \hY$ and $\hQ' \subset \hY$ be the strict transforms of $Q$ and $Q'$.
We have
\begin{equation*}
K_{\tilde Y} = \mu^* K_Y + Q = \mu^*\pi^*K_\barY + Q,
\qquad 
K_{\hat Y}= \theta^* K_{\tilde Y}+\sum \hat a_i \hat E_i, 
\end{equation*}
where $\hat E_i\subset \hat Y$ are prime exceptional divisors and $\hat a_i\in \ZZ_{>0}$. 
Then 
\begin{equation*}
K_{\hat Y} = 
\theta^* \mu^* \pi^* K_{\barY} + \theta^* Q +\sum \hat a_i \hat E_i = 
\theta^* \mu^* \pi^* K_{\barY} + \hat Q +\sum (\hat a_i+ \hat m_i ) \hat E_i,
\end{equation*}
where $\hat m_i$ is the multiplicity of $\hat E_i$ in $\theta^* Q$. 
For any prime divisor $\hat E_i$ with center at $Q$ we have~\mbox{$\hat m_i \ge 1$} and so $\hat a_i + \hat m_i > 1$. 
Therefore, $\hQ$ is the only divisor in~$\hY$ with discrepancy~1 over~$\barY$ whose image in~$Y$ is~$\{y\}$.
A similar computation shows that the discrepancy of~$\hQ'$ over~$\barY$ is also~1.
So, if~$\mu(\theta(\hQ')) = \{y\}$, it follows that $\hQ' = \hQ$.
Similarly, if $\mu'(\theta'(\hQ)) = \{y'\}$, it follows that $\hQ' = \hQ$.
In both cases, it follows that~$Q$ is birational to $Q'$.

Finally, assume $\mu(\theta(\hQ')) \ne \{y\}$ and $\mu'(\theta'(\hQ)) \ne \{y'\}$.
Note that $\mu(\theta(\hQ')) \subset \pi^{-1}(\bar{y})$.
On the other hand, $\pi^{-1}(\bar{y})$ is a tree of smooth rational curves (this follows, e.g., from~\cite[Theorem~1.14]{Reid1983}).
Hence there is a unique minimal chain of components of $\pi^{-1}(\bar{y})$ connecting~$y$ with~$\mu(\theta(\hQ'))$,
and in this chain there is a unique curve passing through~$y$. 
By uniqueness this curve is defined over~$\kk$ and gives a $\kk$-point on~$Q$.
Similarly, we find a $\kk$-point of~$Q'$.
But~$Q$ and~$Q'$ are smooth quadrics in~$\PP^3$, because~$y$ and~$y'$ are ordinary double points,
hence both~$Q$ and~$Q'$ are $\kk$-rational, and a fortiori, birational to each other.
\end{proof}

\subsection{Terminal and canonical pairs}

For a linear system $\cM$ of Weil divisors we denote by~$\Bs(\cM)$ the base locus of $\cM$ (as a subscheme).

\begin{definition}[{\cite[\S1.9]{Alexeev-1994ge}}, {\cite[\S4.6]{Kollar95:pairs}}]
Let $Y$ be a normal projective variety with a linear system $\cM$ of Weil divisors on~$Y$ such that $K_Y + \cM$ is $\QQ$-Cartier.
The pair $(Y,\cM)$ 
is called {\sf terminal} (resp.\ {\sf canonical}), 
if for any birational morphism~\mbox{$\pi \colon \tY \to Y$} from a normal variety~$\tY$
and for a general member $M \in \cM$ one has
\begin{equation}
\label{eq:discr}
K_\tY + (\pi^{-1})_*(M) = \pi^*(K_Y + M) + \sum a_i^\cM E_i
\end{equation}
with $a^\cM_i > 0$ (resp.\ $a^\cM_i \ge 0$) for all $i$, where 
$E_i$ are the prime exceptional divisors of $\pi$,
and where recall that~$(\pi^{-1})_*$ denotes the strict transform under the rational map~$\pi^{-1}$, see~\S\ref{subsec:terminal}.
\end{definition}

\begin{remark}
\label{remark:logMMP}
Assume that $Y$ is $\QQ$-factorial and the pair $(Y,\cM)$ has terminal singularities.
Then the MMP for pairs~\cite{Alexeev-1994ge} applied to $(Y,\cM)$ preserves these properties. 
The standard arguments apply
because the steps of the MMP do not contract $\cM$,
see~\cite[Lemma~3.38, Corollaries~3.42--3.43]{Kollar-Mori:88}.
In particular, MMP for pairs over a non-closed field~$\kk$ works well 
as soon as the linear system~$\cM$ is defined over~$\kk$ (see, e.g., \cite[Remark~0.3.14]{Mori-1988}).
\end{remark}

The following lemma relates terminality of a pair to that of the underlying variety.

\begin{lemma}[{\cite[Lemma~1.13]{Alexeev-1994ge}}]
\label{lemma:terminal-pair}
Assume that members of $\cM$ are $\QQ$-Cartier.
\begin{enumerate}
\item 
\label{lemma:terminal-pair1}
If $(Y,\cM)$ is terminal then $Y$ has terminal singularities.
\item 
\label{lemma:terminal-pair2}
Moreover, if $\dim Y = 3$ then $\Bs(\cM)$ is finite,
$\Bs(\cM) \cap \Sing(Y) = \varnothing$,
and a general member $M \in \cM$ is smooth.
In particular, $\cM$ is a linear system of Cartier divisors.
\item 
\label{lemma:terminal-pair3}
Conversely, if $Y$ has terminal singularities and $\cM$ is base point free then $(Y,\cM)$ is terminal.
\end{enumerate}
\end{lemma}
\begin{proof}
Let $\pi \colon \tY \to Y$ be a birational morphism from a normal variety~$\tY$.
The statement~\ref{lemma:terminal-pair1} is straightforward:
besides~\eqref{eq:discrepancy} and~\eqref{eq:discr} write
\begin{equation*}
(\pi^{-1})_*(M) = \pi^*M - \sum m_i E_i.
\end{equation*}
Clearly, $m_i \ge 0$. 
Therefore, $a_i = a_i^\cM + m_i \ge a_i^\cM$, hence $Y$ has terminal singularities if $(Y,\cM)$ is terminal.
This also proves statement~\ref{lemma:terminal-pair3} since~$m_i=0$ if $\cM$ is base point free.

For the proof of~\ref{lemma:terminal-pair2}, let $y\in Y$ be a singular (terminal) point of Gorenstein index~$r$.
Assume that $y\in \Bs(\cM)$. 
By Lemma~\ref{lemma:factoriality}
the divisor class~$rM$ is Cartier at~$y$.
Hence, $m_i\in \frac 1r\ZZ_{>0}$ for any divisor~$E_i$ with center $y$, in particular $m_i \ge 1/r$.
Furthermore, by~\cite{Kawamata92:discr,Markushevich96:discr}
there exists a divisor~$E_i$ with center~$y$ such that $a_i = 1/r$. 
Therefore, for this $i$ we have
\begin{equation*}
0 < a_i^\cM = a_i - m_i \le 0.
\end{equation*}
This contradiction shows that $y\notin \Bs(\cM)$
and we conclude that $\Bs(\cM) \cap \Sing(Y) = \varnothing$.
In particular, a general divisor in~$\cM$ does not pass through~$\Sing(Y)$, hence~$\cM$ is Cartier.

Now assume that $y\in \Bs(\cM)$ is a singular point of a general member $M \in \cM$ which is smooth on~$Y$.
Since $M$ is Cartier, for the exceptional divisor of the blowup of $y$ 
we have~\mbox{$m \ge 2$} and~\mbox{$a = 2$}, hence $a^\cM = a - m \le 0$, a contradiction.
Therefore, $M$ is smooth.

If $\dim \Bs(\cM)\ge 1$, then we obtain a contradiction as above by blowing up a curve in $\Bs(\cM)$.
\end{proof}

\begin{lemma}[{\cite[Proposition~2.8]{Corti95:Sark}}]
\label{lemma:factorialization-pairs}
If $(Y,\cM)$ is canonical then there exists a birational morphism~\mbox{$f \colon \tY \to Y$}
such that for the strict transform $\tcM = (f^{-1})_*(\cM)$ the pair $(\tY,\tcM)$ is terminal, the variety $\tY$ is $\QQ$-factorial, and 
\begin{equation}
\label{eq:maximal-crepant-blowup}
K _\tY + \tcM = f^*(K_Y + \cM).
\end{equation}
\end{lemma}

\begin{remark}\label{rem:canonical-pair}
A map $f$ as in Lemma~\ref{lemma:factorialization-pairs} is called {\sf maximal crepant blowup of $(Y,\cM)$}.
A maximal crepant blowup is defined over the base field~$\kk$, but is not unique.
Indeed, 
it can be constructed by first resolving the singularities of the pair $(Y,\cM)$
and then running the MMP for pairs relative over $(Y,\cM)$ as in~\cite[Prop.~2.8]{Corti95:Sark}.

If in the assumptions of Lemma~\ref{lemma:factorialization-pairs} the pair $(Y,\cM)$ is terminal, then 
the morphism~$f$ is small, i.e., it does not contract divisors 
(indeed, a combination of~\eqref{eq:discr} and~\eqref{eq:maximal-crepant-blowup} shows that~$f$ has no exceptional divisors). 
In this case $f$ is called {\sf $\QQ$-factorialization}.
\end{remark}

\subsection{Unirationality criteria}

As we already mentioned, the first criterion is well known.

\begin{lemma}
\label{lemma:unirational-fibration}
Let $Y \to B$ be a proper dominant morphism of irreducible varieties and let $Z \subset Y$ be a $\kk$-unirational subvariety dominating~$B$.
Assume one of the following conditions hold
\begin{enumerate}
\item 
\label{lemma:unirational-fibration:curve}
the general fiber of $Y \to B$ is a smooth geometrically rational curve, or 
\item 
\label{lemma:unirational-fibration:conic-bundle}
the general fiber of $Y \to B$ is a conic bundle over~$\PP^1$ with at most~$3$ singular fibers, or
\item 
\label{lemma:unirational-fibration:dP}
the general fiber of $Y \to B$ is a smooth del Pezzo surface 
of degree~$d \ge 3$, or
\item 
\label{lemma:unirational-fibration:dP2}
the general fiber of $Y \to B$ is a smooth del Pezzo surface of degree~$d = 2$ 
and the general fiber of $Z \to B$ is not contained in the union of lines on the general fiber of $Y \to B$, or
\item 
\label{lemma:unirational-fibration:cubic}
the general fiber of $Y \to B$ is a cubic hypersurface, which is not a cone.
\end{enumerate}
Then $Y$ is $\kk$-unirational.
\end{lemma}

\begin{proof}
Consider the base change $Y \times_B Z \to Z$ of the morphism along~$Z \to B$.
Its general fiber is a geometrically rational curve, or a conic bundle, 
or a del Pezzo surface of degree~$d \ge 2$, or a cubic hypersurface
over the field $\kk(Z)$. 
On the other hand, the diagonal morphism $Z \to Y \times_B Z$ provides it with a~$\kk(Z)$-point
(and in the case of a degree~2 del Pezzo surface this point does not lie on a line).
Applying~\cite[Theorems~IV.7.8, IV.8.1]{Manin-Cubic-forms-e-I}, \cite[Theorem~2]{Iskovskikh:70e}, and~\cite[Theorem~1.2]{Kollar:cubic}
we conclude that the general fiber is $\kk(Z)$-unirational, hence~$Y \times_B Z$ is unirational over~$Z$.
Since~$Z$ is unirational over~$\kk$, it follows that~$Y \times_B Z$ is unirational over~$\kk$ as well.
\end{proof}

In the rest of the section we prove another unirationality criterion.

\begin{proposition}
\label{proposition:q-fano}
Let $Y$ be a normal projective threefold and let $S \subset Y$ be a $\kk$-unirational surface.
Assume there is a linear system~$\cM$ of Weil divisors such that
\begin{equation}
\label{eq:s}
K_Y + \cM + S \sim 0,
\end{equation} 
the pair $(Y,\cM)$ is terminal, and the map $Y \dashrightarrow \PP^N$ given by the linear system $\cM$ is generically finite onto its image.
Then $Y$ is $\kk$-unirational.
\end{proposition}

Note that the assumptions of the proposition imply that the surface $S$ is a $\QQ$-Cartier divisor on~$Y$
(because $K_Y + \cM$ is $\QQ$-Cartier by terminality of $(Y,\cM)$).
Note also that~\eqref{eq:s} implies that~$\cM$ is a subsystem in $|- K_Y - S|$, 
and if some subsystem satisfies the other assumptions, then the full system also does. 
Replacing~$\cM$ with~$|- K_Y - S|$ we may (and will) assume that~$\cM$ is a complete linear system.
So, $\cM$ is not really an additional data.

The proof takes up the rest of the section.
The techniques applied in the proof are more or less standard (cf.~\cite{CampanaFlenner1993}).

\begin{lemma}
\label{lemma:q-factorialization}
Let $f \colon \tY \to Y$ be the $\QQ$-factorialization and set $\tS := (f^{-1})_*(S)$.
If the assumptions of Proposition~\textup{\ref{proposition:q-fano}} hold for $(Y,S)$, they also hold for $(\tY,\tS)$.
\end{lemma}
\begin{proof}
Set $\tcM = (f^{-1})_*(\cM)$.
Since $f$ is small, we have $\tS = f^*(S)$; moreover, $\tS$ is birational to~$S$, hence~$\tS$ is $\kk$-unirational. 
The pair $(\tY,\tcM)$ is terminal by the construction (see Lemma~\ref{lemma:factorialization-pairs} and Remark~\ref{rem:canonical-pair}).
Furthermore, the map given by~$\tcM$ coincides with the composition $\tY \xrightarrow{\ f\ } Y \dashrightarrow \PP^N$, 
hence is generically finite, and 
\begin{equation*}
K_\tY + \tcM + \tS = 
f^*(K_Y + \cM) + f^*(S) = f^*(K_Y + \cM + S) \sim 0,
\end{equation*}
so the condition~\eqref{eq:s} is satisfied for $(\tY,\tS)$.
\end{proof}

From now on we assume that~$Y$ is $\QQ$-factorial.

\begin{lemma}
\label{lemma:mmp}
The assumptions of Proposition~\textup{\ref{proposition:q-fano}} are preserved by $(K_Y + \cM)$-MMP.
In other words, if $f \colon Y \dashrightarrow Y'$ is a divisorial contraction or a flip of a 
$(K_Y + \cM)$-negative extremal ray then the strict transform
\begin{equation*}
S' := f_*(S) \subset Y'
\end{equation*}
is a $\kk$-unirational surface, the pair $(Y',f_*(\cM))$ is terminal, $K_{Y'} + f_*(\cM) + S' \sim 0$,
and the map~\mbox{$Y' \dashrightarrow \PP^N$} given by $f_*\cM$ is generically finite.
\end{lemma}
\begin{proof}
Assume the map $f$ is obtained from the contraction of a $(K_Y + \cM)$-extremal ray $R$, so that~$(K_Y + \cM) \cdot R < 0$.
From~\eqref{eq:s} we obtain $S \cdot R > 0$.
Therefore, $S$ is not contracted by $f$, and so $S'$ is a surface. 

The map $f \colon S \dashrightarrow S'$ is dominant, hence $S'$ is $\kk$-unirational.
Furthermore, the map given by~$f_*\cM$ is equal to the composition of $f^{-1}$ and the map given by $\cM$, hence it is generically finite.
Finally, the linear equivalence $K_{Y'} + f_*(\cM) + S' \sim 0$ follows from~\eqref{eq:s} and
terminality of the pair $(Y,f_*(\cM))$ follows from Remark~\ref{remark:logMMP}.
\end{proof} 

The assumption~\eqref{eq:s} excludes the possibility that~$(K_Y + \cM)$ is nef,
therefore, after running the~$(K_Y + \cM)$-MMP 
we can assume that~$Y$ is a Mori fiber space.

\begin{lemma}
\label{lemma:divisors-f-ample}
In the situation of Proposition~\textup{\ref{proposition:q-fano}} 
assume that $f \colon Y \to B$ is a morphism to a normal variety~$B$ with~$\dim(B) \le 2$
such that $-(K_Y + \cM)$ is $f$-ample and $\uprho(Y) = \uprho(B) + 1$.
Then~$-K_Y$, $\cM$, and~$S$ are all $f$-ample.
\end{lemma}
\begin{proof}
Since the map defined by $\cM$ is generically finite, 
the restriction of $\cM$ to the general fiber of~$f$ is big.
On the other hand, $-(K_Y + \cM)$ is ample on the general fiber.
Therefore, $-K_Y$ is also ample on the general fiber of $f$.
Since $\uprho(Y) = \uprho(B) + 1$ and $\cM$ is
big on the general fiber, it is also $f$-ample.
Finally, $S$ is $f$-ample by~\eqref{eq:s}.
\end{proof}

The case when the base of the Mori fiber space has positive dimension is easy.

\begin{lemma}
\label{lemma:mfs}
In the situation of Proposition~\textup{\ref{proposition:q-fano}} 
assume that $f \colon Y \to B$ is a morphism with connected fibers to a normal variety $B$ of dimension~$1$ or~$2$
such that $-(K_Y + \cM)$ is $f$-ample and~\mbox{$\uprho(Y) = \uprho(B) + 1$}.
Then $Y$ is $\kk$-unirational.
\end{lemma}

\begin{proof}
Since the singularities of $Y$ are terminal (Lemma~\ref{lemma:terminal-pair}\ref{lemma:terminal-pair1}), 
they are isolated (Lemma~\ref{lemma:terminal-singularities}). 
Therefore, the general fiber of $f$ is smooth.
If $\dim(B) = 2$ then the general fiber of $f$ is a geometrically rational curve.
Similarly, if~$\dim(B) = 1$ then for a general geometric fiber $D$ of $f$ we have
\begin{equation*}
-K_D \sim S\vert_D + \cM\vert_D
\end{equation*}
is a sum of two ample (by Lemma~\ref{lemma:divisors-f-ample}) divisor classes, 
hence $D$ is a smooth del Pezzo surface of index at least~2, i.e., of degree $d \ge 8$.
Moreover, in both cases $S$ is $f$-ample, hence dominates $B$.
Therefore, $Y$ is $\kk$-unirational 
by Lemma~\ref{lemma:unirational-fibration}\ref{lemma:unirational-fibration:curve}, \ref{lemma:unirational-fibration:dP}.
\end{proof}

From now on we may assume that $Y$ is a (terminal) Fano threefold.

\begin{lemma}
\label{lemma:a-d}
In the situation of Proposition~\textup{\ref{proposition:q-fano}} 
assume that~$Y$ is a terminal~$\QQ$-factorial Fano threefold with~$\uprho(Y) = 1$.
Then~$\cM$ is ample, every divisor in~$\cM$ is Cartier, and a general divisor~$M$ in~$\cM$ is a smooth del Pezzo surface.
Moreover, there are a positive rational number~$a$ and
a positive integer~$d$ such that
\begin{equation}
\label{eq:a-d}
\cM \equiv aS,
\qquad 
S^3 = d/a,\quad 
S^2 \cdot M = d,\quad 
S \cdot M^2 = da,\quad 
M^3 = da^2,
\end{equation}
where $a$ is also integer if $d \le 7$.
Finally,
\begin{equation}
\label{eq:a-d-inequality}
da(a+1) \ge 4.
\end{equation} 
\end{lemma}

\begin{proof}
By Lemma~\ref{lemma:terminal-pair} any $M$ is Cartier and a general $M$ is smooth.
By adjunction~\eqref{eq:s} implies
\begin{equation*}
K_M = - S\vert_M,
\end{equation*}
and since $S$ is ample (by Lemma~\ref{lemma:divisors-f-ample})
it follows that $M$ is a smooth del Pezzo surface.
We denote by~$d := K_M^2$ its degree (so that $d$ is an integer); then $d = S^2 \cdot M$.
Since $\uprho(Y) = 1$, we have a numerical equivalence 
\[
M \equiv aS
\]
for some (a priori rational) positive number~$a$;
then~$M$ is ample and all the equalities in~\eqref{eq:a-d} follow.
Furthermore, if $d \le 7$, then the canonical class of $M$ is primitive, 
hence~$S$ is a primitive element of the group of Weil $\QQ$-Cartier divisors modulo 
numerical equivalence, hence~$a$ is an integer.

Finally, to prove~\eqref{eq:a-d-inequality} we compute the dimension of~$\cM$.
In the exact sequence
\begin{equation*}
0 \longrightarrow \cO_Y \longrightarrow \cO_Y(M) \longrightarrow \cO_M(M) \longrightarrow 0
\end{equation*}
we have $H^1(Y,\cO_Y) = 0$ since $Y$ is a Fano variety.
Furthermore, 
\begin{equation*}
\dim \bigl|M\vert_M\bigr| = |-aK_M| = da(a+1)/2, 
\end{equation*}
hence $\dim |M| = 1 + da(a+1)/2$.
Since $\cM$ defines a generically finite map, we have $\dim |M| \ge 3$, and~\eqref{eq:a-d-inequality} follows.
\end{proof}

Unirationality of $Y$ now follows from the following two lemmas.

\begin{lemma}
\label{lemma:d-big}
In the situation of Proposition~\textup{\ref{proposition:q-fano}} 
assume that $Y$ is a terminal $\QQ$-factorial Fano threefold with $\uprho(Y) = 1$
and let $d$ be defined as in Lemma~\textup{\ref{lemma:a-d}}.
If $d \ge 2$ then $Y$ is $\kk$-unirational.
\end{lemma}

\begin{proof}
We claim that $S\cap\Bs\cM=\varnothing$.
Indeed, by~\eqref{eq:s} we have~$M - S \sim K_Y + 2M$ and since~$M$ is ample by Lemma~\ref{lemma:a-d}
the Kawamata--Viehweg vanishing (see, e.g., \cite[Theorem~2.70]{Kollar-Mori:88}) implies~$H^1(Y,\cO_Y(M-S)) = 0$
and so from the exact sequence
\begin{equation*}
0 \larrow \cO_Y(M-S) \larrow \cO_Y(M) \larrow \cO_S(M) \larrow 0
\end{equation*}
we obtain $S\cap\Bs\cM=\Bs \bigl|M\vert_S\bigr|$. 
Furthermore, \eqref{eq:s} and the adjunction formula imply
\begin{equation*}
-K_S=M|_{S}, 
\end{equation*}
hence~$S\cap\Bs\cM=\Bs |-K_S|$ and~$S$ is a Gorenstein del Pezzo surface (possibly non-normal) of degree~$M^2\cdot S=da$.
Finally, $da\ge 2$ because, being the degree of a Gorenstein del Pezzo surface, it is a positive integer, 
and~$da = 1$ is impossible since then~$a \ge 3$ by~\eqref{eq:a-d-inequality} while~$d \ge 2$ by assumption.
Now we obtain~$\Bs |-K_S|=\varnothing$ by~\cite[\S4]{Hidaka-Watanabe-1981} and~\cite[Corollary~4.10]{Reid:dP94}, hence~$S\cap\Bs\cM=\varnothing$.

Since $\cM$ is ample by Lemma~\ref{lemma:a-d}, the linear system~$\cM\vert_S$ defines a finite morphism,
hence by Bertini's theorem for a general~$M$ in~$\cM$ the intersection~$M \cap S$ is geometrically irreducible.
Take two general members $M,\, M'\in \cM$.
Again
by Bertini's theorem the curve $C := M\cap M'$ is geometrically irreducible
and has no embedded components because~$Y$ is Cohen--Macaulay \cite[Corollary~5.25]{Kollar-Mori:88}.
Moreover, $S\cap C$ is a finite set.
By Lemma~\ref{lemma:a-d} the divisor~$M$ is a smooth del Pezzo surface of degree~$d \ge 2$ and~$C$ is a Cartier divisor on~$M$.

Consider the pencil~$\cP\subset \cM$ generated by~$M$ and~$M'$ and the corresponding map
\begin{equation*}
\psi \colon Y \dashrightarrow \PP^1.
\end{equation*}
The intersection of~$S$ and the general member~$M\in \cP$ is a curve that is not contained in~$\Bs \cP$,
hence the restriction of~$\psi$ to~$S$ is dominant.
Since the curve~$M\cap S$ is irreducible and it is a member of the anticanonical linear system~$|-K_{M}|$, 
in the case $d=2$ this curve is not contained in the union of lines on~$M$.
Now applying Lemma~\ref{lemma:unirational-fibration}\ref{lemma:unirational-fibration:dP2} to the blowup~$\tY \to Y$ of~$Y$ 
along the curve~$C=\Bs \cP$ we conclude that~$\tY$ (and therefore~$Y$ as well) is $\kk$-unirational.
\end{proof}

\begin{lemma}
\label{lemma:d-small}
In the situation of Proposition~\textup{\ref{proposition:q-fano}} 
assume that $Y$ is a terminal $\QQ$-factorial Fano threefold with~$\uprho(Y) = 1$
and let $d$ be defined as in Lemma~\textup{\ref{lemma:a-d}}.
If $d = 1$ then $Y$ is $\kk$-unirational.
\end{lemma}

\begin{proof}
By Lemma~\ref{lemma:a-d}
the general member $M$ of $\cM$ is a smooth del Pezzo surface of degree~$d = 1$.
Besides, we have $a \ge 2$ by~\eqref{eq:a-d-inequality} and since $M \equiv aS$ we conclude that~$|S - M| = \varnothing$.
Consider the exact sequence
\begin{equation*}
0 \larrow \cO_Y(S-M) \larrow \cO_Y(S) \larrow \cO_M(S) \larrow 0.
\end{equation*}
Note that $S - M \sim K_Y + 2S$ and the surface $S$ is 
an ample divisor 
(Lemma~\ref{lemma:divisors-f-ample}), 
hence we have~$H^1(Y,\cO_Y(S-M)) = 0$ by the Kawamata--Viehweg vanishing
(see, e.g., \cite[Theorem~2.70]{Kollar-Mori:88}).
On the other hand, we have~$\cO_M(S) \cong \cO_M(-K_M)$, hence $\dim |S| = |S\vert_M| = 1$.

Furthermore, the class of~$S$ cannot be represented as a sum of effective classes in~$\Cl(Y)$ 
(because any effective class on~$Y$ is ample by the assumption~$\uprho(Y) = 1$, hence restricts nontrivially to~$M$,
while on the del Pezzo surface~$M$ the anticanonical class~$-K_M = S\vert_M$ has degree~1 hence cannot be represented as a sum of effective classes)
hence the pencil $|S|$ has no fixed components.
Let~\mbox{$S',\, S''\in |S|$} be general members; then $\Bs|S| = S' \cap S''$.
From~\eqref{eq:a-d} we deduce
\begin{equation}
\label{eq:deg-gamma}
(S' \cap S'') \cdot M = S'\cdot S''\cdot M = S^2 \cdot M = d = 1.
\end{equation}
This implies that 
$S'\cap S''$ is $\bkk$-irreducible, and generically reduced curve
(on the other hand, $S'\cap S''$ may have some 0-dimensional embedded components).
By Bertini's theorem~$S'$ is smooth outside of $S'\cap S''$ and has only isolated singularities on $S'\cap S''$.
Moreover, $S'$ is Cohen--Macaulay by~\cite[Corollary~5.25]{Kollar-Mori:88}, hence it is normal.
By adjunction and~\eqref{eq:s} we have
\begin{equation}
\label{eq:k-sprime}
K_{S'}= -M |_{S'}=-aS|_{S'},
\qquad\text{and}\qquad
K_{S'}^2 = M^2 \cdot S = a \ge 2.
\end{equation}
Since~$M$ is a Cartier divisor class (Lemma~\ref{lemma:a-d}), it follows that~$S'$ is a normal Gorenstein del Pezzo surface.
Hence, by~\cite{Hidaka-Watanabe-1981} either the surface~$S'$ has only du Val singularities 
or~$S'$ is a {\sf generalized cone} over a smooth elliptic curve~$C$,
i.e., $S'$ is obtained by contracting the negative section on a ruled surface $\PP_C(\cO \oplus \cL)$ with $\deg(\cL) = -a$;
below we will discuss these two cases separately.

First, assume that~$S'$ is a del Pezzo surface with du Val singularities.
Then $(Y,|S|)$ is a canonical pair (see~\cite[4.8.3${}'$]{Kollar95:pairs} and~\cite[Lemma~1.21]{Alexeev-1994ge}), 
hence by Lemma~\ref{lemma:factorialization-pairs} there exists a crepant birational morphism~\mbox{$f \colon \tY \to Y$} 
such that the pair $(\tY,\tcS)$ is terminal, where~\mbox{$\tcS = (f^{-1})_*(|S|)$}.
In particular, the base locus of the pencil $\tcS$ is at most finite (Lemma~\ref{lemma:terminal-pair}), hence empty.
Therefore, the linear system~$\tcS$ induces a morphism 
\begin{equation*}
\pi \colon \tY \larrow \PP^1
\end{equation*}
that resolves the rational map $Y \dashrightarrow \PP^1$ induced by $|S|$.

Let $\tS' = (f^{-1})_*(S')$. 
It is a general member of $\tcS$.
By Lemma~\ref{lemma:factorialization-pairs} we have $K_\tY + \tS' = f^*(K_Y + S')$, 
hence by adjunction
\begin{equation}
\label{eq:k-tsp}
K_{\tS'} = f^*K_{S'} = -f^*M.
\end{equation}
Therefore, $\tS'$ is a crepant partial resolution of singularities of $S'$.
In particular,
\begin{equation*}
K_{\tS'}^2 = K_{S'}^2 = S \cdot M^2 = a,
\end{equation*}
so $\pi$ is a family of weak del Pezzo surfaces of degree~$a \ge 2$ with du Val singularities.

Let $\tilde{\Gamma} \subset \tS'$ be the strict transform of~$(S'\cap S'')_{\mathrm{red}}$ with respect to the map $\tS' \to S'$.
By~\eqref{eq:k-tsp} and~\eqref{eq:deg-gamma} it is a $(-1)$-curve. 
Thus, the family of weak del Pezzo surfaces $\pi \colon \tY \to \PP^1$ contains a distinguished family of $(-1)$-curves.
In particular, the schematic generic fiber of $\pi$ is a weak del Pezzo surface with du Val singularities of degree~$a \ge 2$
over $\kk(\PP^1)$ with a $(-1)$-curve.
Contracting the curve, we obtain a weak del Pezzo surface with du Val singularities of degree~$a + 1 \ge 3$
over~$\kk(\PP^1)$ with a point that does not lie on $(-2)$-curves (if there are any).
By~\cite[Theorem~IV.7.8]{Manin-Cubic-forms-e-I}, \cite[Theorems~A,B,C]{Coray-Tsfasman-1988}, and~\cite[Theorem~1.2]{Kollar:cubic} it is unirational over $\kk(\PP^1)$, 
hence $\tY$ (and therefore~$Y$ as well) is $\kk$-unirational.

To conclude, we show that the case where a general member $S'$ of $|S|$ is a generalized cone over an elliptic curve is impossible.
By Bertini's theorem the vertex~$\upsilon$ of~$S'$ lies on the curve $S' \cap S''$, 
If $\upsilon$ is a smooth point of~$Y$, then~$S'$ is a Cartier divisor at~$\upsilon$ 
and~\mbox{$S'\cap S''$} is a Cartier divisor on~$S'$ at~$\upsilon$.
In particular, the curve $S'\cap S''$ is Cohen--Macaulay, hence reduced.
By~\cite[Theorem~4.4]{Hidaka-Watanabe-1981} the linear system $|- K_{S'}|$ is base point free, 
hence by~\eqref{eq:k-sprime} and~\eqref{eq:deg-gamma} the curve~\mbox{$S' \cap S''$} is smooth.
But any Cartier divisor is singular at any singular point of the ambient variety.
This contradiction shows that $\upsilon \in \Sing(Y)$.
Since $Y$ has only a finite number of singular points, it follows that a general member of $|S|$ 
has vertex at a fixed singular point $\upsilon \in Y$.

Let $\pi \colon \tY \to Y$ be a resolution of the pair $(Y,S')$ at the point~$\upsilon$ 
and let $\tS' = (\pi^{-1})_*(S') \subset \tY'$ be the strict transform of $S'$.
We have
\begin{equation*}
K_{\tY}=\pi^*K_Y +\sum a_i E_i,\qquad \tS'=\pi^*S' -\sum m_i E_i,
\end{equation*}
where $E_i$ are the exceptional divisors.
Summing up and using adjunction, we obtain
\begin{equation*}
K_{\tS'} = (\pi\vert_{\tS'})^*K_{S'} + \sum (a_i - m_i)E_i|_{\tS'}. 
\end{equation*}
On the other hand, the singularity $\upsilon \in S'$ is strictly log canonical~\cite[Example~3.8]{Kollar95:pairs} 
and has a unique exceptional divisor (the negative section on $\PP_C(\cO \oplus \cL)$) with discrepancy~$-1$.
Moreover, this divisor is the unique divisor that intersects nontrivially the strict transform $\tilde{l}$ 
of a general ruling~$l$ of~$S'$.
We can assume that this divisor is $E_0 \cap \tS'$.
Then
\begin{equation*}
a_0 - m_0 = -1.
\end{equation*}
On the other hand, since $|\tS'|$ is a movable linear system and $1=M\cdot l=aS'\cdot l$, we have
\begin{equation*}
0 < \tS'\cdot \tilde{l} = 
\pi^*S'\cdot \tilde{l} - \sum m_i E_i\cdot \tilde{l} =
S' \cdot l - m_0 = 
1/a - m_0.
\end{equation*}
Therefore, $m_0 < 1/a < 1$, hence $a_0 = m_0 - 1 < 0$.
Thus $Y$ is not terminal at $\upsilon$ (and even not canonical).
This contradiction completes the proof.
\end{proof}

A combination of the above results gives the proof of the proposition.

\begin{proof}[Proof of Proposition~\textup{\ref{proposition:q-fano}}]
By Lemma~\ref{lemma:q-factorialization} we may assume that the variety $Y$ is $\QQ$-factorial.
Running the~\mbox{$(K_Y + \cM)$}-MMP for pairs~(see Remark~\ref{remark:logMMP})
and using Lemma~\ref{lemma:mmp} we may assume that $Y$ is a Mori fiber space.
If the dimension of the base is positive, we apply Lemma~\ref{lemma:mfs}.
If $Y$ is a Fano variety and the parameter $d$ defined in Lemma~\ref{lemma:a-d} is at least~2, we apply Lemma~\ref{lemma:d-big}.
Otherwise, we apply Lemma~\ref{lemma:d-small}.
In all cases we deduce $\kk$-unirationality of~$Y$.
\end{proof}

\section{Birational transformations}
\label{section:sarkisov-links}

In this section we describe some Sarkisov links for prime Fano threefolds discussed in Theorem~\ref{theorem:main},
i.e., for smooth Fano threefolds $X$ such that
\begin{equation}
\label{ass:fano}
\Pic(X_\bkk) = \ZZ\cdot K_{X_\bkk}
\qquad\text{and}\qquad 
\g(X) \in \{ 7, 9, 10, 12 \},
\end{equation} 
where $\g(X)$ is the genus of~$X$ defined in~\eqref{def:genus}, i.e., for threefolds $X_{12}$, $X_{16}$, $X_{18}$, and~$X_{22}$.
Throughout the section we denote~$\g(X)$ simply by~$g$.

\subsection{Sarkisov links}

In this subsection we prove some results about certain Sarkisov links for Fano threefolds.
We will only need Sarkisov links that start with a blowup.

\begin{definition}
\label{def:sl}
Let $X$ be a smooth Fano threefold with~$\uprho(X) = 1$ and let $Z \subset X$ be a reduced $\kk$-irreducible subvariety of codimension at least~$2$.
A {\sf Sarkisov link of $X$ with center at~$Z$} is a diagram
\begin{equation}
\label{eq:sl}
\vcenter{
\xymatrix@C=2em{
\tX \ar[d]_{\sigma} \ar[drr]_(.5)\phi \ar@{-->}^{\tpsi}[rrrr] &&&&
\tX^+ \ar[d]^{\sigma_+} \ar[dll]^(.5){\phi_+}
\\
X &&
\bar{X} &&
X^+,
} }
\end{equation} 
where $\sigma$ is the blowup along~$Z$,
$\sigma_+$ is an extremal Mori contraction,
$\phi$ and $\phi_+$ are small crepant birational contractions, 
so that 
\begin{equation*}
K_\tX = \phi^*K_\barX
\qquad\text{and}\qquad 
K_{\tX^+} = \phi_+^*K_\barX,
\end{equation*}
and~$\tpsi = \phi_+^{-1} \circ \phi$ is a flop.
In particular, we allow $\phi$, $\phi_+$, and $\psi$ to be isomorphisms;
in that case we require~$\sigma_+$ to be distinct from~$\sigma$.
\end{definition}

\begin{remark}\label{remark:def:k}
We restrict to the case where the singularities of $Z$ are at worst ordinary double points. 
Then~$\tilde X$ also has a unique ordinary double point over each singularity of~$Z$ and is smooth elsewhere.
By~\cite{Kollar1989a} the same is true for $\tX^+$.
\end{remark}

The following result holds for all Fano threefolds of Picard number~1 and is quite useful.

\begin{proposition}
\label{proposition:link-k}
Let $X$ be a smooth Fano threefold with $\uprho(X) = 1$ and let $\sigma \colon \tX \to X$ be the blowup 
of a reduced $\kk$-irreducible subvariety $Z \subset X$ of codimension at least~$2$.
If the linear system~$|-nK_\tX|$ is base point free for some $n > 0$ 
and the induced morphism $\phi \colon \tX \to \PP^N$ is a small birational contraction \textup(or an isomorphism\textup)
then there exists a Sarkisov link~\eqref{eq:sl} with center at~$Z$ and it is defined over~$\kk$.
\end{proposition}
\begin{proof}
The blowup $\tX$ of $X$ along $Z$ is obviously defined over $\kk$.
The $\sigma$-exceptional divisor $E$ is $\kk$-irreducible,
hence 
\begin{equation}
\label{eq:rho-tx}
\uprho(\tX) = \rk\Cl(\tX) = 2; 
\end{equation} 
in particular $\tX$ is $\QQ$-factorial.

First, assume that $-K_\tX$ is not ample.
Since $|-n K_\tX|$ is base point free, $-K_\tX$
is orthogonal to one of the two rays of the Mori cone of $\tX$.
The other ray is orthogonal to~$\sigma^*(-K_X)$.
Then the flop~$\tpsi$ exists by the standard MMP argument:
it can be directly defined by
\begin{equation*}
\tX^+ \cong \Proj \left( \bigoplus_{n\ge 0} H^0\left(\tilde X,\, \cO_{\tX}\left(n\big(\sigma^*K_X - aK_{\tilde X}\big)\right)\right) \right),
\end{equation*}
for any sufficiently positive~$a$ (see, e.g.,~\cite[Corollary~6.4]{Kollar-Mori:88}).
In particular, $\tX^+$ is defined over~$\kk$.
Furthermore, $\uprho(\tX^+) = \uprho(\tX) = 2$, hence the Mori cone of~$\tX^+$ also has two rays.
One of them corresponds to the contraction $\phi_+$, hence is $K_{\tilde X^+}$-trivial.
The other should be~$K_{\tX^+}$-negative and gives the extremal Mori contraction $\sigma_+$ which must be defined over~$\kk$.

If $-K_\tX$ is ample, we set $\tX^+ = \barX = \tX$ with the identity maps $\phi$, $\phi_+$, and~$\psi$,
and define $\sigma_+$ as the extremal Mori contraction 
corresponding to the second ray
of the Mori cone (the one, on which~$\sigma^*(-K_X)$ is positive).
Since the first ray is Galois-invariant, so is the second ray, hence $\sigma_+$ is defined over~$\kk$.
\end{proof}

In the rest of this subsection we work under slightly more general than~\eqref{ass:fano} assumptions 
\begin{equation}
\label{ass:fano-weak}
\Pic(X) = \ZZ\cdot K_{X}
\qquad\text{and}\qquad 
\g(X) \ge 6.
\end{equation} 
Let~$X$ be a Fano threefold satisfying~\eqref{ass:fano-weak}, $g = \g(X)$.
We consider the anticanonical embedding
\begin{equation*}
X \subset \PP^{g + 1}
\end{equation*}
(note that $-K_X$ is very ample and~$X$ is an intersection of quadrics by Theorem~\ref{th:bht}) and write
\begin{equation*}
H := -K_X,
\end{equation*}
so that $H$ is the restriction of the hyperplane class.
Considering the blowup of~$X$ along~$Z$ and the Sarkisov link as in~\eqref{eq:sl}
we abuse notation by abbreviating~$\sigma^*H$ to just~$H$;
as it will always be clear from the context on which variety this divisor class is considered, this should not cause any confusion.
We also denote by~$E$ the exceptional divisor of~$\sigma$.

We will need some standard results about linear systems on blowups of $X$.

\begin{lemma}
\label{lemma:linear-systems}
Let $X$ be a Fano threefold satisfying~\eqref{ass:fano-weak}.
\begin{enumerate}
\item 
\label{lemma:linear-systems:conic}
If $\sigma \colon \tX \to X$ is the blowup of a reduced conic~$C$ such that~$X \cap \langle C \rangle = C$
with the exceptional divisor $E \subset \tX$ then
\begin{equation*}
\dim |H - E| = g - 2,
\qquad
\dim |H - 2E| \ge g - 8,
\qquad 
\dim |2H - 3E| \ge 5g - 31.
\end{equation*}
\item 
\label{lemma:linear-systems:cubic}
If $\sigma \colon \tX \to X$ is the blowup of a smooth twisted cubic curve~$C$ such that~$X \cap \langle C \rangle = C$
with the exceptional divisor~\mbox{$E \subset \tX$} then
\begin{equation*}
\dim |H - E| = g - 3,
\qquad
\dim |2H - 3E| \ge 5g - 39.
\end{equation*}
\item 
\label{lemma:linear-systems:point}
If $\sigma \colon \tX \to X$ is the blowup of a point with the exceptional divisor $E \subset \tX$ then
\begin{align*}
\dim |H-2E| &= g - 3,
& \dim |2H - 5E| &\ge 5(g - 7),
\\
\dim |H - 3E| &\ge g - 9,
& \dim |3H - 7E| &\ge 14g - 92.
\end{align*}
\end{enumerate}
\end{lemma}
\begin{proof}
By Riemann--Roch we have $\dim |H| = g + 1$, $\dim |2H| = 5g - 1$, $\dim |3H| = 14g - 8$.
Furthermore, for all integers~$a$ and~$b$ we have an exact sequence
\begin{equation*}
0 \larrow \cO_{\tilde X}(aH - (b+1)E) \larrow \cO_{\tilde X}(aH - bE) \larrow \cO_{E}(aH-bE) \larrow 0.
\end{equation*}
Therefore
\begin{equation}
\label{eq:h0-inequality}
\dim|aH - (b+1)E| \ge \dim|aH - bE| - \dim\left(H^0(E,\cO_E(aH - bE))\right).
\end{equation}

If the blowup center is a curve $C$, then $\cO_E(-E)$ is the bundle $\cO(1)$ on $E \cong \PP_C(\cN_{C/X})$, hence
\begin{equation*}
H^0(E,\cO_{E}(aH-bE)) \cong H^0(C,\Sym^b (\cN^\vee_{C/X}) \otimes \cO_X(aH)\vert_C).
\end{equation*}
Note that the assumption~$X \cap \langle C \rangle = C$
implies that the sheaf~$I_C(H)$ is globally generated,
and hence so is the sheaf $\cN^\vee_{C/X}(H) \cong (I_C/I_C^2)(H)$.
Therefore, for~$b \le a$ we have the vanishing~$H^1(C,\Sym^b (\cN^\vee_{C/X}) \otimes \cO_X(aH)\vert_C) = 0$,
hence by Riemann--Roch 
\begin{equation*}
\dim(H^0(C,\Sym^b (\cN^\vee_{C/X}) \otimes \cO_X(aH)\vert_C)) = (b + 1)((a - b/2)\deg(C) + b + 1).
\end{equation*}

Similarly, if the blowup center is a point, then $E \cong \PP^2$ and $\cO_E(-E) \cong \cO_{\PP^2}(1)$, hence
\begin{equation*}
\dim (H^0(E,\cO_{E}(aH-bE))) = \dim (H^0(\PP^2,\cO_{\PP^2}(b))) = (b+1)(b+2)/2.
\end{equation*}

Applying~\eqref{eq:h0-inequality} several times, and using very ampleness of $|H|$ 
and linear normality of $C$ at the first step to obtain the equality, we deduce the lemma.
\end{proof}

The following lemma is useful for the construction of Sarkisov links with center at a curve.

\begin{lemma}
\label{lemma:sl-construction}
Let $X$ be a Fano threefold satisfying~\eqref{ass:fano-weak}.
Let $C \subset X$ be a $\kk$-irreducible reduced conic or a smooth twisted cubic curve and set $d =\deg(C)$.
If there is an equality of schemes
\begin{equation}
\label{eq:span-condition}
X \cap \langle C \rangle = C
\end{equation} 
then the anticanonical linear system $|-K_\tX| = |H - E|$ on the blowup $\tX$ of $X$ along~$C$ is base point free and~$\dim |-K_\tX| = g - d$.
If 
\begin{equation*}
\phi \colon \tX \larrow \PP^{g - d}
\end{equation*}
is the induced morphism then one of the following two options holds:
\begin{enumerate}
\item 
either $\phi$ is a small contraction and there exists a Sarkisov link~\eqref{eq:sl} with center at~$C$,
\item 
or $\phi$ contracts an irreducible divisor $D$ such that
\begin{equation}
\label{eq:d-curve}
D \sim t\left(H - \frac{2g - 2 - d}{d+2}E\right)
\end{equation}
for some $t \in \ZZ_{>0}$.
Moreover, if $|aH - bE|$ is a nonempty linear system such that $b > a$
then $D$ is its fixed component.
\end{enumerate}
\end{lemma}

\begin{proof}
Since $|H|$ is very ample, the base locus of $|H-C|$ is $X \cap \langle C \rangle$, 
so if~\eqref{eq:span-condition} holds, this is $C$, hence on~$\tX$ the linear system $|H - E|$ is base point free.
The standard intersection theory gives
\begin{equation}
\label{eq:intersection-blowup-curve}
H^3 = 2g - 2,
\qquad 
H^2\cdot E = 0,
\qquad 
H\cdot E^2 = -d,
\qquad 
E^3 = 2 - d,
\end{equation}
therefore $(H - E)^3 = 2(g - 2 - d) > 0$ (because~$g \ge 6$ by~\eqref{ass:fano-weak} and~$d \le 3$), 
hence the morphism $\phi$ defined by $|H - E|$ is generically finite.
The computation of Lemma~\ref{lemma:linear-systems} shows that~$\dim|H-E| = g - d$.

If $\phi$ is a small contraction, a Sarkisov link exists by Proposition~\ref{proposition:link-k}. 
Otherwise $\phi$ contracts a divisor~$D$.
By~\eqref{eq:rho-tx} the divisor~$D$ is $\kk$-irreducible.
Furthermore, we must have~\mbox{$(H-E)^2 \cdot D = 0$}, hence equalities~\eqref{eq:intersection-blowup-curve} imply~\eqref{eq:d-curve}.
Finally, if~$R$ is a curve in a fiber of~\mbox{$D \to \phi(D)$}, we have~\mbox{$(H - E) \cdot R = 0$} and~$H \cdot R > 0$, 
hence for $b > a$ we have~$(aH - bE) \cdot R < 0$ and so any divisor in~\mbox{$|aH - bE|$} contains~$D$,
hence $D$ is a fixed component of that linear system.
\end{proof}

A similar result can be proved for a point instead of a curve.
Before stating it we will need a general result about the Hilbert scheme~$\rF_1(X,x)$ of lines 
of a projective variety~$X \subset \PP^N$ passing through a point~$x$.
We denote by~$T_xX$ the tangent space to~$X$ at~$x$ and by~$\bT_xX \subset \PP^N$ the embedded tangent space.
Note that~$\bT_xX$ can be thought of as a cone over~$\PP(T_xX)$ with vertex~$x$.

\begin{lemma}
\label{lemma:f1-x-x}
If $X \subset \PP^N$ is an intersection of quadrics there is a natural embedding
\begin{equation*}
\rF_1(X,x) \subset \PP(T_xX).
\end{equation*}
Moreover, if~$\rF_1(X,x) = \varnothing$ the intersection~$X \cap \bT_xX$ is the first-order neighborhood of~$x$,
and otherwise~$X \cap \bT_xX$ is the cone over~$\rF_1(X,x)$.
\end{lemma}

\begin{proof}
For every quadric containing~$X$ the embedded tangent space~$\bT_xX$ is contained in the tangent space to the quadric,
hence the intersection of the quadric with~$\bT_xX$ is a cone with vertex at~$x$.
Since~$X$ is an intersection of quadrics, it follows that~\mbox{$X \cap \bT_xX$} 
is the cone over an intersection of quadrics in~\mbox{$\PP(T_xX)$}
and since~$\rF_1(\PP^{N},x) = \PP(T_x\PP^{N})$, it follows that the base of this cone is isomorphic to~$\rF_1(X,x)$ and~\mbox{$X \cap \bT_xX$} is the cone over it.
Moreover, if~$\rF_1(X,x) = \varnothing$ it follows that~$X \cap \bT_xX$ is the first-order neighborhood of~$x$.
\end{proof}

\begin{lemma}
\label{lemma:sl-construction-point}
Let~$X$ be a Fano threefold satisfying~\eqref{ass:fano-weak} and
assume~$\rF_1(X,x) = \varnothing$.
Then the anticanonical linear system $|-K_\tX| = |H - 2E|$ on the blowup~$\tX$ of~$X$ at~$x$ is base point free and if
\begin{equation*}
\phi \colon \tX \larrow \PP^{g - 3}
\end{equation*}
is the induced morphism then one of the following two options holds:
\begin{enumerate}
\item 
either $\phi$ is small and there exists a Sarkisov link~\eqref{eq:sl} with center at~$x$,

\item 
or $\phi$ contracts an irreducible divisor $D$ such that 
\begin{equation}
\label{eq:td-class}
D \sim t\left(H - \frac{g - 1}2 E\right),
\end{equation}
for some $t \in \ZZ_{>0}$.
Moreover, if $|aH - bE|$ is a nonempty linear system such that $b > 2a$
then $D$ is its fixed component.
\end{enumerate}
\end{lemma}

\begin{proof}
Since~$X \subset \PP^{g+1}$ is an intersection of quadrics and~$\rF_1(X,x) = \varnothing$, 
Lemma~\ref{lemma:f1-x-x} proves that the intersection~$X \cap \bT_xX$ is the first-order neighborhood of~$x$, 
hence~$|H - 2E|$ is base point free.
The standard intersection theory gives
\begin{equation}
\label{eq:intersections-h-e}
H^3 = 2g - 2,
\qquad 
H^2 \cdot E = H \cdot E^2 = 0,
\qquad 
E^3 = 1,
\end{equation}
therefore $(H - 2E)^3 = 2g - 10 > 0$ (because~$g \ge 6$ by~\eqref{ass:fano-weak}), 
hence the morphism~$\phi$ defined by~$|H - 2E|$ is generically finite.
Using Lemma~\ref{lemma:linear-systems}(iii) we conclude that the target of~$\phi$ is~$\PP^{g-3}$.

If $\phi$ is a small contraction, a Sarkisov link exists by Proposition~\ref{proposition:link-k}.
Otherwise $\phi$ contracts a divisor~$D$.
By~\eqref{eq:rho-tx} the divisor~$D$ is $\kk$-irreducible.
Furthermore, we must have~$(H - 2E)^2 \cdot D = 0$, hence~\eqref{eq:intersections-h-e} imply~\eqref{eq:td-class}.
Finally, if~$R$ is a curve in a fiber of~\mbox{$D \to \phi(D)$}, we have~$(H - 2E) \cdot R = 0$ and~$H \cdot R > 0$, 
hence for $b > 2a$ we have~$(aH - bE) \cdot R < 0$ and so any divisor in~\mbox{$|aH - bE|$} contains~$D$,
hence $D$ is a fixed component of that linear system.
\end{proof}

\subsection{Sarkisov links with center at curves}
\label{subsection:sl-curves}

From now on we return to Fano threefolds satisfying~\eqref{ass:fano}.
The simplest link starts with a blowup of a line.
The following result is well-known in the case of an algebraically closed field.

\begin{theorem}[cf.~\cite{Iskovskikh1989}, {\cite[\S 4.3]{IP}}] 
\label{th:sl:line}
Let $X$ be a Fano threefold satisfying~\eqref{ass:fano} and
let~\mbox{$L \subset X$} be a line defined over~$\kk$.
There exists a Sarkisov link~\eqref{eq:sl} with center at~$L$
and for $\sigma_+$ the following hold:
\begin{enumerate}
\item 
\label{th:sl:line:g=7}
If $\g(X)=7$, then $X^+\simeq \PP^1$ and $\sigma_+$ is a del Pezzo fibration of degree $5$.

\item 
\label{th:sl:line:g=9}
If $\g(X)=9$, then $X^+\simeq \PP^3$ and~$\sigma_+$ is the blowup of a smooth connected curve~$\Gamma$ of genus~$3$ and degree~$7$.

\item 
\label{th:sl:line:g=10}
If $\g(X)=10$, then $X^+ \subset \PP^4$ is a smooth $\kk$-rational quadric 
and $\sigma_+$ is the blowup of a smooth connected curve~$\Gamma$ of genus~$2$ and degree~$7$.

\item 
\label{th:sl:line:g=12}
If $\g(X)=12$, then $X^+ \subset \PP^6$ is a smooth quintic del Pezzo threefold 
and $\sigma_+$ is the blowup of a smooth connected rational curve~$\Gamma$ of degree~$5$.
\end{enumerate}
\end{theorem}

\begin{proof}
By Proposition~\ref{proposition:link-k} the diagram~\eqref{eq:sl}
constructed over $\bkk$ in~\cite[\S 4.3]{IP} 
is defined over~$\kk$.
Furthermore, it follows from~\cite[\S 4.3]{IP} that $X^+$ is a $\kk$-form of $\PP^1$, or $\PP^3$, 
or a smooth 3-dimensional quadric, 
or a quintic del Pezzo threefold, respectively.
It remains to show that the forms of~$\PP^1$ and~$\PP^3$ are trivial and the quadric is $\kk$-rational.
By Propositions~\ref{proposition:sb} and~\ref{proposition:quadric} it is enough 
to note that~$X(\kk) \ne \varnothing$ by Lemma~\ref{lemma:f3-point}, hence $X^+(\kk) \ne \varnothing$ by Lemma~\ref{lemma:points}.
\end{proof}

The second link we consider has center at a singular conic.

\begin{theorem}[cf.~{\cite{Takeuchi-1989}}]
\label{th:sl:conic}
Let $X$ be a Fano threefold satisfying~\eqref{ass:fano} and
let $C\subset X$ be a reduced 
singular $\kk$-irreducible conic.
There exists a Sarkisov link~\eqref{eq:sl} with center at~$C$
and for~$\sigma_+$ the following hold:
\begin{enumerate}

\item 
\label{th:sl:conic:g=7}
If $\g(X)=7$, then $X^+ \subset \PP^{4}$ is a $\kk$-rational quadric and $\sigma_+$ is the blowup 
of a connected curve $\Gamma$ of degree $10$ and arithmetic genus~$7$.

\item 
\label{th:sl:conic:g=9}
If $\g(X)=9$, then $X^+ \simeq \PP^1$ and $\sigma_+$ is a del Pezzo fibration of degree $6$. 

\item 
\label{th:sl:conic:g=10}
If $\g(X)=10$, then $X^+\simeq \PP^2$
and $\sigma_+$ is a conic bundle with discriminant curve~\mbox{$\Delta \subset X^+$} of degree~$4$.
The variety $\tX^+$ has a single ordinary double point which is infinitesimally rational 
and whose image in~$X^+$ is a singular point of~$\Delta$.
Moreover, conics on $X$ intersecting~$C$ are parameterized by a curve.

\item 
\label{th:sl:conic:g=12}
If $\g(X)=12$, then $X^+ \subset \PP^4$ is a $\kk$-rational quadric and $\sigma_+$ is the blowup 
of a connected curve $\Gamma$ of degree $6$ and arithmetic genus~$0$.
\end{enumerate}
\end{theorem}

\begin{proof}
We apply Lemma~\ref{lemma:sl-construction}.
Since $X$ is an intersection of quadrics and contains no planes by Theorem~\ref{th:bht}, condition~\eqref{eq:span-condition} holds.
If~$\phi$ contracts a divisor~$D$, then~$D \sim t(H - \tfrac{g - 2}2 E)$ for some~\mbox{$t \in \ZZ_{>0}$} by~\eqref{eq:d-curve}.
By Lemma~\ref{lemma:linear-systems}(i) we have~$\dim |H - 2E| \ge g - 8$, hence for~$g \ge 9$ the divisor~$D$ is a fixed component of~$|H - 2E|$.
If~$g = 9$ this contradicts the equivalence~$D \sim t(H - \tfrac72 E)$ proved above 
and for~$g \in \{10,12\}$ it follows that the movable part of~$|H - 2E|$ must be equal to
\begin{equation*}
|H - 2E - D| = \left|H - 2E - H + \tfrac{g-2}2 E\right| = \left|\tfrac{g - 6}2E\right|,
\end{equation*}
which is absurd since~$E$ is the exceptional divisor of a blowup.

Similarly, if $g = 7$ we have $\dim |2H - 3E| > 0$ (again by~Lemma~\ref{lemma:linear-systems}(i)),
hence $D$ is a fixed component of $|2H - 3E|$. 
The equivalence $D \sim t(H - \tfrac52 E)$ then reduces to~$D \sim 2H - 5E$, hence the movable part of $|2H - 3E|$ must be equal to
\begin{equation*}
|2H - 3E - D| = |2H - 3E - 2H + 5E| = |2E|
\end{equation*}
which is again impossible.

Thus $\phi$ is a small contraction and so by Lemma~\ref{lemma:sl-construction} there exists a Sarkisov link~\eqref{eq:sl}.
It remains to determine the type of the contraction $\sigma_+$.
We start with some preparations.

Let $x\in C$ be the singular point of the conic; it is defined over~$\kk$.
By Remark~\ref{remark:def:k} the variety~$\tX$ has a unique ordinary double point $\tilde{x} \in \tX$ over $x$
and the variety $\tX^+$ also has a unique ordinary double point $\tilde{x}_+ \in \tX^+$.
The point $\tilde{x}$ is infinitesimally rational, because the smooth curve $\sigma^{-1}(x) \cong \PP^1$ passes through~$\tilde{x}$
and gives a $\kk$-point on the exceptional divisor of the blowup of $\tX$ at $\tilde{x}$ 
(which is a $\kk$-form of a smooth quadric, since $\tilde{x}$ is an ordinary double point).
By Lemma~\ref{lemma:flop-odp} the point $\tilde{x}_+$ is also infinitesimally rational.

Furthermore, $\tX$ and $\tX^+$ have the same ranks of the Picard and class groups. 
Since the exceptional divisor $E$ of $\sigma$ is irreducible over~$\kk$ and has two components over~$\bkk$, we conclude that
\begin{equation*}
\rk\Cl (\tX^+) = 2
\qquad\text{and}\qquad
\rk\Cl (\tX^+_\bkk)= 3.
\end{equation*}
On the other hand, over~$\bkk$ there is a small contraction from the consecutive blowup of the components of $C$ to $\tX_\bkk$,
which shows that the components of $E$ are not $\QQ$-Cartier, hence
\begin{equation}
\label{eq:rho-txplus}
\uprho(\tX^+) = \uprho(\tX^+_\bkk) = 2.
\end{equation}
Since under the contraction~$\sigma_+$ the Picard number must decrease, it follows that
\begin{equation}
\label{eq:rho-xplus}
\uprho(X^+)= \uprho(X^+_\bkk)= 1.
\end{equation}

To proceed we will also need to know some intersection numbers on $\tX^+$.
For this note that by Definition~\ref{def:sl} the classes $K_\tX$ and $K_{\tX^+}$ are the pullbacks 
of the same Cartier divisor class on $\barX$.
Therefore, using the projection formulas for $\phi$ and $\phi_+$ we deduce 
$K_{\tX^+}^i \cdot (E^+)^{3-i} = K_{\tX}^i \cdot (E)^{3-i}$ for~\mbox{$i > 0$}, where $E^+ = \psi_*(E)$ is the strict transform.
Therefore, for an arbitrary divisor class 
\begin{equation}
\label{eq:divisor}
D = u(-K_{\tX^+}) - vE^+
\end{equation} 
in $\Pic(\tX^+)$ using~\eqref{eq:intersection-blowup-curve} we obtain 
\begin{equation}
\label{eq:intersections-txplus}
\begin{aligned} 
D^2 \cdot (-K_{\tX^+})\hphantom{{}^2} &= 2(gu^2 - (2u + v)^2),\\
D\hphantom{{}^2} \cdot (-K_{\tX^+})^2 &= 2(gu - 2(2u + v)),\\
(-K_{\tX^+})^3 &= 2(g - 4).
\end{aligned}
\end{equation} 
Now we are ready to consider the possibilities for $\sigma_+$.
We will use Theorem~\ref{theorem:smooth-contractions}.

First, assume that $\sigma_+$ is a del Pezzo fibration, i.e., $\dim X^+=1$.
Since $X^+$ is geometrically rationally connected (because $X$ is) 
and has a $\kk$-point (by Lemma~\ref{lemma:points}), we have $X^+\simeq \PP^1$.
If~$D$ is the class of a general fiber of $\sigma_+$ then $D^2\cdot (-K_{\tilde X^+}) = 0$,
and writing $D$ in the form~\eqref{eq:divisor} with~$u,\, v \in \ZZ_{>0}$ and $\gcd(u,v) = 1$ and using~\eqref{eq:intersections-txplus},
we obtain the equation
\begin{equation*}
gu^2 - (2u + v)^2 = 0.
\end{equation*}
For $g \in \{7,9,10,12\}$ this equation has only one solution: $g=9$, $u=v=1$.
Using~\eqref{eq:intersections-txplus} again we obtain 
$K_D^2	= D\cdot (-K_{\tilde X^+})^2 =(2g-8)u-4v = 6$,
hence $\sigma_+$ is a sextic del Pezzo fibration.
This gives case~\ref{th:sl:conic:g=9} of the theorem.

Now assume that $\sigma_+$ is a conic bundle, i.e., $\dim X^+=2$.
The surface $X^+$ is smooth by Theorem~\ref{theorem:smooth-contractions}\ref{theorem:smooth-contractions:conic-bundle}
and geometrically rationally connected (because $X$ is),
hence geometrically rational.
Since it also satisfies~\eqref{eq:rho-xplus}, we see that $X^+_\bkk \cong \PP^2_\bkk$.
Finally, $X^+(\kk) \ne \varnothing$ by Lemma~\ref{lemma:points}, hence 
\begin{equation*}
X^+ \cong \PP^2.
\end{equation*}

If~$D$ is the pullback of the class of a line on $X^+$ then $D^2\cdot (-K_{\tilde X^+}) = 2$,
and writing $D$ in the form~\eqref{eq:divisor} with~$u,\, v \in \ZZ_{>0}$ and $\gcd(u,v) = 1$ and using~\eqref{eq:intersections-txplus},
we obtain the equation
\begin{equation*}
gu^2 - (2u + v)^2 = 1.
\end{equation*}
This time for $g \in \{7,9,10,12\}$ the only solution is $g = 10$, $u = v = 1$ 
(see, e.g., \cite{Takeuchi-1989} or~\cite[Lemma~4.1.6]{IP}).
Using~\eqref{eq:intersections-txplus} again we compute the degree of the discriminant curve $\Delta$ as
\begin{equation*}
\deg \Delta = 12 - D\cdot (-K_{\tilde X^+})^2 = 12 - 2(gu - 2(2u + v)) = 4.
\end{equation*}
Note that the total space of a conic bundle is smooth away from the preimage of the singular locus of the discriminant curve.
Therefore, the image $x_+ = \sigma_+(\tilde{x}_+) \in X^+$ of the singular point~\mbox{$\tilde{x}_+ \in \tX^+$}
is a $\kk$-rational point on $\Sing(\Delta)$.

Let us show that conics on~$X$ intersecting~$C$ are parameterized by a curve.
Indeed, by the above computation the map~$\sigma_+ \circ \tpsi \colon \tX \dashrightarrow X^+$ is given by the linear system $|D| = |H - 2E|$.
If~$C'$ is a conic intersecting~$C$ (and having no common components with~$C$), 
the strict transform~$\tC'$ of~$C'$ in~$\tX$ 
has nonpositive intersection index with~$H - 2E$.
So, excluding the finite number of flopping curves of~$\psi$, we see that this intersection index should be~$0$,
hence the intersection index of~$\tC'$ with~$K_\tX = - H + E$, or equivalently with~$K_{\tX^+}$, should be~$-1$.
This means that the strict transforms of such~$C'$ on~$\tX^+$ become components of the fibers of~$\sigma_+^{-1}(\Delta) \to \Delta$,
hence are parameterized by a double covering of~$\Delta$.
This gives case~\ref{th:sl:conic:g=10} of the theorem.

Now assume that $\sigma_+$ is birational. 
By Theorem~\ref{theorem:smooth-contractions} it contracts a $\kk$-irreducible divisor~$D$. 
Moreover, a combination of~\eqref{eq:rho-txplus}, \eqref{eq:rho-xplus}, and Remark~\ref{remark:connected} shows that~$D$ is connected.

Consider the case where $\sigma_+(D)$ is 0-dimensional.
Let $S\in |-K_{\tilde X^+}|$ be a general member. 
Then~$S$ is a smooth K3 surface and $\Upsilon:=S \cap D$ is a connected reduced contractible curve on~$S$. 
In this situation, $\Upsilon^2<0$, $K_S\cdot \Upsilon=0$ and $\frac 12 \Upsilon^2 + 1 = p_a(\Upsilon) = 0$ by the genus formula. 
Therefore, we have $D^2\cdot (-K_{\tilde X^+}) = \Upsilon^2 = -2$.
We also have
\begin{equation*}
K_{\tilde X^+}=\sigma_+^* K_{X^+}+aD,
\qquad 
a \in \{\tfrac12, 1, 2\}
\end{equation*}
by Theorem~\ref{theorem:smooth-contractions}\ref{theorem:smooth-contractions:divisorial:point}.
Furthermore,
\begin{equation*}
D\cdot (-K_{\tilde X^+})^2 = D\cdot (-K_{\tilde X^+}) \cdot (-\sigma_+^* K_{X^+}-aD) =-aD^2\cdot (-K_{\tilde X^+}) =2a.
\end{equation*}
Writing $D$ in the form~\eqref{eq:divisor} with~$u,\, v \in \ZZ_{>0}$
and using~\eqref{eq:intersections-txplus} we obtain the equations
\begin{equation*}
\left\{
\begin{aligned}
gu^2 - (2u + v)^2 &= -1,\\
2(gu - 2(2u + v)) &\in \{1,2,4\}.
\end{aligned}
\right.
\end{equation*}
One can check that this system has no solutions for $g \in \{7,9,10,12\}$, so this case is impossible.

Finally, assume that $\sigma_+(D)$ is a curve $\Gamma \subset X^+$.
Note that $\Gamma$ is connected because~$D$ is.
The threefold $X^+$ is a terminal Gorenstein Fano variety with Picard number~1, hence 
\begin{equation*}
(-K_{X^+})^3 \in \{2,4,6,8,10,12,14,16,18,22,24,32,40,54,64\},
\quad 
(-K_{X^+}) \cdot \Gamma > 0,
\quad 
p_a(\Gamma) \ge 0.
\end{equation*}
On the other hand, in this case $K_{\tilde X^+} = \sigma_+^* K_{X^+} + D$, hence we have 
\begin{equation*}
\begin{aligned}
(- \sigma_+^*K_{X^+})^3 &= (-K_{\tilde X^+}+D)^2\cdot(-K_{\tilde X^+}),
\\
(- \sigma_+^*K_{X^+})\cdot \Gamma &= (-K_{\tilde X^+}+D)\cdot D\cdot (-K_{\tilde X^+}),
\\
2p_a(\Gamma)-2 &= (-K_{\tilde X^+})\cdot D^2
\end{aligned}
\end{equation*}
(for the first and third equalities we use the fact that $(- \sigma_+^*K_{X^+})^2 \cdot D = 0$,
and for the second that~$X^+$ is smooth along $\Gamma$).
Writing $D$ in the form~\eqref{eq:divisor} with~$u,\, v \in \ZZ_{>0}$
and using~\eqref{eq:intersections-txplus} we obtain the following equations
\begin{equation}
\label{eq:xplus}
\left\{
\begin{aligned}
(2g-8)(u+1)^2-8(u+1)v-2v^2 &\in \{2,4,6,8,10,12,14,16,18,22,24,32,40,54,64\},
\\
(2g-8)(u+1)u-8uv -4v -2v^2 &> 0
\\
(2g-8)u^2-8uv-2v^2 &\ge -2.
\end{aligned}
\right.
\end{equation}
Furthermore, rewriting $K_{\tilde X^+} = \sigma_+^* K_{X^+} + D$ in the form 
\begin{equation}
\label{eq:kxplus}
\sigma_+^* K_{X^+} = (1 + u)K_{\tX^+} + vE^+
\end{equation}
and taking modulo $K_{\tX^+}$, we see that $v$ is equal to the Fano index of~$X^+$, hence~$v \in \{1,2,3,4\}$.

One can check that the only solutions to~\eqref{eq:xplus} satisfying all these restrictions are:

\begin{itemize}
\item $g = 7$, $D\sim 5(-K_{\tilde X^+})-3E$, $X^+ \subset \PP^4$ is a quadric, $\deg \Gamma = 10$, $p_a(\Gamma) = 7$;

\item $g=7$, $D\sim 3(-K_{\tilde X^+})-2E$, $X^+ \subset \PP^4$ is a cubic, $\deg \Gamma = 4$, $p_a(\Gamma) = 0$;

\item $g=9$, $D\sim -K_{\tilde X^+}-E$, $X^+ \subset \PP^{13}$ is a Fano threefold of genus 12, $\deg \Gamma = 6$, $p_a(\Gamma) = 0$;

\item $g = 12$, $D\sim 2(-K_{\tilde X^+})-3E$, $X^+ \subset \PP^4$ is a quadric, $\deg \Gamma = 6$, $p_a(\Gamma) = 0$.
\end{itemize}
When $X^+$ is a quadric, its singular locus consists of at most one point (because the singularities of $X^+$ are at worst terminal),
and $X^+$ contains the $\kk$-point $x_+ = \sigma_+(\tilde{x}_+)$.
If $x_+$ is a smooth point, then $X^+$ is $\kk$-rational.
If $x_+$ is singular, then the map $\sigma_+$ is an isomorphism over a neighborhood of $x_+$,
hence $x_+$ is infinitesimally $\kk$-rational (because $\tilde{x}_+$ is).
Thus, $X^+$ is the cone over a $\kk$-rational quadric, hence is again $\kk$-rational.
This gives cases~\ref{th:sl:conic:g=7} and~\ref{th:sl:conic:g=12}.

The case where $g =9$ is impossible,
because 
\begin{equation*}
\dim|D| = \dim |-K_{\tilde X^+}-E^+| = \dim |H - 2E| > 0
\end{equation*}
(see Lemma~\ref{lemma:linear-systems}\ref{lemma:linear-systems:conic}).
Finally, if $g = 7$, then the cubic hypersurface $X^+$ is geometrically rational.
Then $X^+$ is singular and its singular locus consists of a unique node. 
In this case, $X^+$ is locally factorial (see, e.g., \cite{Finkelnberg}).
Therefore, the curve $\Gamma$ is smooth and lies in the smooth locus. 
Then~\mbox{$\rk \Cl (\tilde X^+_{\bkk})=2$}, a contradiction.
\end{proof}

The third link we consider is the one that starts with the blowup of a (smooth) rational twisted cubic curve.
Here we restrict ourselves to the case~\mbox{$g = 9$}.

\begin{theorem}[{\cite[Proposition~4.6.3]{IP}}]
\label{th:sl:cubic}
Let~$X$ be a prime Fano threefold with~$\g(X) = 9$. 
Let~$C\subset X$ be a smooth rational twisted cubic curve such that no bisecant line to~$C$ is contained in~$X$.
There exists a Sarkisov link~\eqref{eq:sl} with center at~$C$,
where~$X^+$ is a quintic del Pezzo threefold and~$\sigma_+$ is the blowup of a smooth connected curve~$\Gamma$ of degree~$9$ and genus~$3$.
\end{theorem}

\begin{proof}
We apply Lemma~\ref{lemma:sl-construction}.
Since $X$ is an intersection of quadrics and contains no quadric surfaces by Theorem~\ref{th:bht}, 
the intersection~$X \cap \langle C \rangle$ is either~$C$ or the union of~$C$ with a bisecant.
So, the assumption of the theorem implies that~\eqref{eq:span-condition} holds.

If~$\phi$ contracts a divisor, then~\mbox{$D \sim t(H - \tfrac{13}5E)$} for some~\mbox{$t \in \ZZ_{>0}$} by~\eqref{eq:d-curve}.
By Lemma~\ref{lemma:linear-systems}(ii) we have $\dim |2H - 3E| > 0$, hence the divisor~$D$ is a fixed component of~$|2H - 3E|$.
These two observations contradict each other, hence~$\phi$ is a small contraction and so there exists a Sarkisov link~\eqref{eq:sl}.
Now we can apply the standard computations (see \cite[\S4.1]{IP}) and determine the type of~$\sigma_+$.
\end{proof}

\subsection{Double projection from a point}
\label{subsection:double-projection}

In this subsection we discuss the double projection from a point;
under an additional assumption it provides the half-link~$\phi \circ \sigma^{-1} \colon X \dashrightarrow \barX$ for the Sarkisov link with center at~$x$,
and it is also important by itself.
Recall that by Corollary~\ref{cor:length-f1} the length of the Hilbert scheme~$\rF_1(X,x)$ of lines on~$X$ passing through a point~$x \in X$ 
is bounded by~3, so, a fortiori, there is at most~3 distinct $\bkk$-lines through~$x$.

\begin{lemma}
\label{lemma:double-projection}
Let $X$ be a Fano threefold satisfying~\eqref{ass:fano} with a $\kk$-point~$x$ such that $\rF_1(X,x)$ is reduced.
Let $L_1,\dots,L_k$ be the lines \textup(defined over~$\bkk$\textup) on~$X$ passing through~$x$, where $k \le 3$ is the length of~$\rF_1(X,x)$.
Let $\sigma \colon \tX \to X$ be the blowup of~$x$ with exceptional divisor~$E$,
let $\tL_i \subset \tX$ be the strict transforms of $L_i$.
\par\smallskip
\noindent\textup{(i)} 
The curves $\tL_i$ are disjoint, intersect $E$ at $k$ distinct points $x_1,\dots,x_k$,
and the base locus of the anticanonical linear system $|-K_{\tilde X}|= |H-2E|$ is equal to $\sqcup \tL_i$.
Consider the blowup
\begin{equation*}
\hat{\sigma} \colon \hat{X} \longrightarrow \tX
\end{equation*}
of all $\tL_i$ with exceptional divisors $F_i$. 
The anticanonical linear system $|-K_{\hat{X}}|= |H-2E - \sum F_i|$ is base point free and defines a regular morphism 
\begin{equation*}
\phi \colon \hat{X} \larrow \bar X\subset \PP^{g-3}.
\end{equation*}

\par\smallskip
\noindent\textup{(ii)} 
There exist no divisors $D \subset \hX$ such that $\phi(D)$ is contained in a linear subspace $\PP^{g-5} \subset \PP^{g-3}$;
in particular $\phi$ does not contract divisors to points or lines in~$\barX$.
If
\begin{equation}
\label{eq:stein-factorization}
\hX \xrightarrow{\ \phi'\ } \barX' \xrightarrow{\ \phi''\ } \bar X \subset \PP^{g-3}
\end{equation}
is the Stein factorization for $\phi$, where $\phi''$ is a finite morphism, then 
\begin{equation}
\label{eq:deg-barx}
\deg(\phi'') \deg (\bar{X}) = 2g - 10
\end{equation} 
and $\deg(\phi'') \le 2$.
\par\smallskip
\noindent\textup{(iii)} 
The strict transform $\hat{E}$ of $E$ in $\hat{X}$ is isomorphic to the blowup of $E$ at points $x_1,\dots,x_k$, and 
\begin{equation*}
\barE := \phi(E)
\end{equation*}
is a $\kk$-unirational surface.
If $k = 0$ then $\barE$ is a regular projection of the Veronese surface, 
and if additionally $\dim |H - 3E| = g - 9$ then $\barE$ is embedded into $\PP^{g-3}$ by the complete linear system.
Finally, if~$k > 0$ then $\barF_i := \phi(F_i)$ are $\bkk$-planes, 
the morphisms $\phi\vert_\hE \colon \hE \to \bar{E}$ and $\phi\vert_{F_i} \colon F_i \to \barF_i$ are generically finite, 
and if~$k = 3$ then~$\barE$ is a $\kk$-plane.
\end{lemma}

\begin{proof}
The curves $\tL_i$ are evidently disjoint away from $E$, and the points $x_i$ are distinct, 
because different lines have different tangent vectors at~$x$.
By Lemma~\ref{lemma:f1-x-x} the intersection~$X \cap \bT_xX$
is the cone over the reduced scheme~$\rF_1(X,x)$ of length~$k$, if~$k > 0$, or the first-order neighborhood of~$x$, if~$k = 0$.
Therefore, the base locus of~$|H - 2E|$ is the union of~$\tL_i$,
hence the base locus of~\mbox{$|H - 2E - \sum F_i|$} on~$\hat{X}$ is empty, and hence this linear system defines a morphism~$\phi$.
The dimension of the linear system~$|H - 2E|$ has been computed in Lemma~\ref{lemma:linear-systems}\ref{lemma:linear-systems:point}.
This proves part~(i).

If the linear projection $\phi$ out of $\bT_xX$ contracts a divisor $D$ and $\phi(D)$ is contained in~$\PP^{g-5}$,
then~$\sigma(D) \subset X$ is a divisor contained in the linear span of $\bT_xX$ and~$\PP^{g-5}$, 
i.e., in a projective space of dimension at most~$4 + (g - 5) = g - 1$.
But $X \subset \PP^{g+1}$, so the codimension of the linear span of~$\sigma(D)$ is bigger than~1.
Therefore, a general hyperplane section of $X$ through $\sigma(D)$ contains~$\sigma(D)$ as a fixed component,
which contradicts the assumption $\Pic(X) = \ZZ \cdot H$.

To compute the degree of $\phi$ we note that 
\begin{equation}
\label{eq:intersections-h-e-f}
(H - 2E)^2\cdot F_i = 0,
\qquad 
(H - 2E)\cdot F_i^2 = 1,
\qquad 
F_i^3 = 3,
\quad\text{and}\quad 
F_i\cdot F_j = 0\quad\text{for $i \ne j$}
\end{equation}
(the second equality follows from $(H - 2E)\cdot \tL_i = 1 - 2 = -1$, the third from $\deg(\cN_{\tL_i/\tX}) = -3$,
and the last from disjointness of $\tL_i$ and $\tL_j$).
Taking into account equalities~\eqref{eq:intersections-h-e}, we compute 
\begin{equation*}
\left(H - 2E - \sum F_i\right)^3 = (2g - 2) - 8 + 3k - 3k = 2g - 10, 
\end{equation*}
and~\eqref{eq:deg-barx} follows.
Furthermore, since $\bar{X} \subset \PP^{g-3}$ by definition is not contained in a hyperplane, 
we have $\deg(\bar{X}) \ge \codim(\bar{X}) + 1 = g - 5$, hence $\deg(\phi') \le 2$.
This proves part~(ii).

Since the curves $\tL_i$ intersect $E$ transversely, the strict transform $\hat{E}$ of $E$ is isomorphic to the blowup $\Bl_{x_1,\dots,x_k}(E)$.
Therefore, its image $\barE$ is $\kk$-unirational.

If $k = 0$ the map $\phi\vert_\hE$ is given by a base point free linear system of conics on $E = \PP^2$, 
hence $\barE$ is a regular projection of the Veronese surface.
If moreover $\dim |H - 3E| = g - 9$, the argument of Lemma~\ref{lemma:linear-systems}\ref{lemma:linear-systems:point} shows 
that $\phi\vert_\hE$ is given by the complete linear system of conics.

Using~\eqref{eq:intersections-h-e-f} we obtain~$(H - 2E - \sum F_i)^2\cdot F_j = - 2 + 3 = 1$.
This proves that each $\barF_i$ is a $\bkk$-plane and the morphisms $\phi\vert_{F_i} \colon F_i \to \barF_i$ are generically finite.
Similarly, $(H - 2E - \sum F_i)^2\cdot E = 4 - k$;
since~$k \le 3$ 
the morphism~$\phi\vert_\hE \colon \hE \to \bar{E}$ is generically finite,
and, if~$k = 3$, its image is a $\kk$-plane.
\end{proof}

Now we will show that the morphism~$\phi''$ in~\eqref{eq:stein-factorization} is an isomorphism.

\begin{lemma}
\label{lemma:no-stein}
Assume we are in the situation of Lemma~\textup{\ref{lemma:double-projection}}.
The morphism~$\phi''$ in the Stein factorization~\eqref{eq:stein-factorization} is an isomorphism, 
hence~$\barX$ is normal, $\deg(\barX) = 2g - 10$, and~$\phi = \phi'$ has connected fibers.
\end{lemma}
\begin{proof}
Recall that $\deg(\phi'') \le 2$ by Lemma~\ref{lemma:double-projection}.
Assume $\deg(\phi'') = 2$, so that $\deg(\barX) = g - 5$ by~\eqref{eq:deg-barx}.
Since $\codim(\barX) = g - 6$, we see that $W := \barX$ is a variety of minimal degree
hence by~\cite[Theorem~1.11]{Saint-Donat:74} it is either a quadric, or a scroll, or a cone over a Veronese surface.

First, note that~$W$ cannot be a smooth quadric in~$\PP^4$.
Indeed, if~$k = 0$ then the regular projection~$\barE$ of the Veronese surface is contained in~$W$
and if~$k > 0$ then~$W$ contains a $\bkk$-plane~$\barF_1$.
It remains to note that a smooth quadric can contain neither of these surfaces (see Lemma~\ref{lemma:Veronese0}).

To show that the other cases are also impossible, 
note that the singular locus of a 3-dimensional variety of minimal degree is a point or a line.
Since $\phi$ does not contract divisors to points or lines by Lemma~\ref{lemma:double-projection}, we conclude that
\begin{equation}
\label{eq:dim-sing}
\dim(\phi^{-1}(\Sing(W))) \le 1.
\end{equation}
On the other hand, if~$W$ is not a smooth quadric, 
the hyperplane class of~$W \setminus \Sing(W)$ can be represented as the sum of two movable classes 
(e.g., if~$W$ is a scroll or a cone over a scroll, one of these can be taken to be the class of the fiber of the projection to~$\PP^1$),
so~\eqref{eq:dim-sing} implies that the same is true for the class~$H - 2E - \sum F_i$ 
in~$\Pic(\hX) = \Pic(\hX \setminus \phi^{-1}(\Sing(W)))$, 
and hence the same is true for the hyperplane class of~$X$, which is absurd.
Thus~$\deg(\phi'') = 1$.

It remains to show that~$\phi''$ is an isomorphism. 
Let~$S\in |-K_\hX|$ be a general member. By Bertini's theorem~$S$ is a smooth K3 surface.
The restriction~$A:=-K_\hX\vert_S$ is a nef and big divisor, and the linear system~$|A|$ is base point free
and defines a morphism which is birational onto its image.
By~\cite[Theorem~6.1]{Saint-Donat:74} the graded algebra 
\begin{equation*}
R(S,A):=\bigoplus_{n\ge 0} H^0(S,nA) 
\end{equation*}
is generated by its component of degree $1$.
Moreover, Kawamata--Viehweg vanishing shows that the natural restriction homomorphism
\begin{equation*}
R(\hat X,-K_{\hat X}):=\bigoplus_{n\ge 0} H^0(\hat X,-nK_{\hat X}) \larrow R(S,A)
\end{equation*}
is surjective. Therefore, the algebra $R(\hat X,-K_{\hat X})$ is also generated by its component of degree~$1$.
Since $\phi''$ is defined by the linear system $|H - 2E - \sum F_i| = |-K_\hX|$,
it follows that $\phi''$ is an isomorphism.
Finally, since~$\hX$ is smooth and~\eqref{eq:stein-factorization} is a Stein factorization, $\bar{X} \cong \bar{X}'$ is normal.
\end{proof}

\begin{lemma}
\label{lemma:barx}
Assume again we are in the situation of Lemma~\textup{\ref{lemma:double-projection}}.
Then $\Sing(\barX)$ contains no lines.
Moreover, if $\g(X) \ge 9$ then $\barX \subset \PP^{g-3}$ is an intersection of quadrics.
\end{lemma}
\begin{proof}
Assume~$\Sing(\barX)$ contains a line $\bar{L}$.
Since~$\bar{X}$ is normal, over every point of~$\Sing(\barX)$ the fiber of~$\phi$ is at least 1-dimensional,
hence~$\phi^{-1}(\bar{L})$ contains a divisor, 
which is impossible by Lemma~\ref{lemma:double-projection}.

Now assume $\g(X) \ge 9$.
Let $W\subset \PP^{g-3}$ be the (scheme) intersection of all quadrics passing through $\barX$. 
Assume $W \ne \barX$.
Then we apply the argument of~\cite[Ch.~2, Proposition~3.3]{Iskovskikh-1980-Anticanonical}
and conclude that $W \subset \PP^{g-3}$ is a~4-dimensional variety of minimal degree smooth outside of~$\Sing(\barX)$.
Indeed, \cite[Ch.~2, Proposition~3.3]{Iskovskikh-1980-Anticanonical} is stated and proved for a smooth Fano variety, 
but smoothness is only used when constructing a smooth curve section (whose linear span contains a given point of $W$),
which can be also done for a normal variety (when the given point of $W$ does not lie on~$\Sing(\barX)$).
In particular, we see that $\deg(W) = g - 6$.

Note that $W$ is not a quadric since $g \ge 9$.
Since the singular locus of a variety of minimal degree is a linear subspace,
and~$\Sing(\barX)$ contains no lines, it follows from Lemma~\ref{lemma:double-projection} that~\eqref{eq:dim-sing} holds,
and the movable decomposition argument of Lemma~\ref{lemma:no-stein} provides a contradiction.
\end{proof}

\begin{lemma}
\label{lemma:barx-detailed}
Assume again we are in the situation of Lemma~\textup{\ref{lemma:double-projection}} and either $k = 3$ or~$k = 2$ and~\mbox{$g \ge 9$}.
Then~$\phi$ is a small contraction and~$\barX \subset \PP^{g-3}$ is a Fano variety with terminal Gorenstein singularities, not a cone, and $K_\barX = -\barH$.
In particular, the variety~$\barX$ is not covered by lines.
\end{lemma}
\begin{proof}
Assume that~$\phi$ contracts an irreducible divisor $D$.
Let $C \subset D$ be a general fiber of this contraction, so that $(H - 2\hE - \sum F_i) \cdot C = 0$.
If $H \cdot C = 0$, i.e., $\sigma \circ \hat\sigma \colon \hX \to X$ contracts $C$, 
then $C$ is contained in the union of~$\hE$ and~$F_i$.
But the map $\phi$ is generically finite on $\hE$ and $F_i$ by Lemma~\ref{lemma:double-projection}(iii), 
hence there is at most finite number of curves contracted by~$\phi$
in the union of $\hE$ and $F_i$, while $C$ varies in a family.
Thus~$H \cdot C > 0$, hence either $\hE \cdot C > 0$, or $F_i \cdot C > 0$ for some $i$.
Therefore, the curve~$\phi(D)$ is contained either in $\barE$ or in~$\barF_i$.
The latter is impossible for all $g \ge 7$ by Lemma~\ref{lemma:double-projection}(ii) since~\mbox{$g - 5 \ge 2$}, 
and the former is impossible either if~$k = 3$ or if~$k = 2$ and~$g \ge 9$, since~$g - 5 \ge 3$,
while the linear span of the quadric surface~$\barE$ is~$\PP^3$.

Thus, the morphism~$\phi$ is small. 
The equality $\barH = -K_\barX$ follows from~
\begin{equation*}
\phi^*\barH = H - 2\hE - \sum F_i = -K_\hX 
\end{equation*}
since $\phi$ is small.
Since $\hX$ is smooth and $\phi$ is small and crepant, it follows that $\barX$ is Gorenstein and
the singularities of~$\barX$ are terminal (Lemma~\ref{lemma:small-terminality}).
By Lemma~\ref{lemma:terminal-singularities} the tangent space to~$\barX$ at any singular point is 4-dimensional.
So, if~$\barX$ is a cone, its linear span is $\PP^4$, hence $g = 7$ and~\mbox{$\barX \subset \PP^4$} is a quartic hypersurface.
But then the base of the cone is a smooth quartic surface in~$\PP^3$, hence~$\barX$ is not geometrically rationally connected, which is absurd.

Assume that $\barX$ is covered by lines, so that~$\rF_1(\barX)$ contains a component~$\Sigma_0$ of dimension at least~2,
such that lines parameterized by this component cover~$\barX$.
Since~$\barX$ is not a cone, we have~$\dim \rF_1(\barX,\bar{x}) \le 1$ for any point~$\bar{x} \in \barX$, 
hence a general line parameterized by~$\Sigma_0$ is contained in the smooth locus of~$\barX$.
The argument of~\cite[Lemma~2.2.3]{KPS} then shows that lines parameterized by~$\Sigma_0$ could only cover a surface in~$\barX$.
This contradiction completes the proof.
\end{proof}

\begin{corollary}
\label{cor:projection-birational}
Assume again we are in the situation of Lemma~\textup{\ref{lemma:double-projection}}, $k = 3$, and $g \ge 9$.
The linear projection out of $\barE$
\begin{equation*}
\pi_\barE \colon \barX \dashrightarrow \PP^{g - 6}
\end{equation*}
is birational onto its image.
\end{corollary}
\begin{proof}
Recall that $\barX \subset \PP^{g-3}$ is an intersection of quadrics by Lemma~\ref{lemma:barx}.
The fibers of~$\pi_\barE$ can be described as follows.
Let~$\Pi_t$ be a 3-space in~$\PP^{g-3}$ containing the plane~$\barE$ that corresponds to a point~$t \in \PP^{g-6}$.
For each quadric~$Q$ passing through~$\barX$ the intersection~$Q \cap \Pi_t$ is the union of~$\barE$ and the residual plane.
Intersecting these residual planes for all quadrics we obtain a linear subspace in~$\Pi_t$.
Thus, the general fiber of the map~$\pi_\barE$ is a linear space.

On the other hand, the general fiber of $\pi_\barE$ cannot be a linear space of positive dimension,
because otherwise $\barX$ would be covered by lines, contradicting Lemma~\ref{lemma:barx-detailed}.
Therefore, the map~$\pi_\barE$ is birational onto its image.
\end{proof}

\begin{lemma}
\label{lemma:ci3}
Assume again we are in the situation of Lemma~\textup{\ref{lemma:double-projection}}.
If~$g = 9$ then $\barX \subset \PP^6$ is a complete intersection of three quadrics and if~$k = 2$ then~$\bar{E}$ is a smooth quadric surface.
\end{lemma}
\begin{proof}
By Lemma~\ref{lemma:barx} the subvariety $\barX \subset \PP^6$ is an intersection of quadrics.
Its codimension is~3, and by Lemma~\ref{lemma:no-stein} its degree is $2g - 10 = 8$.
Therefore, it is a complete intersection of three quadrics.

If $k = 2$ the map $\phi\vert_\hE$ is given by the linear system of conics on $E = \PP^2$ passing through two points $x_i$.
Since~$\phi$ is regular and small (Lemma~\ref{lemma:barx-detailed}), $\bar{E}$ 
is a smooth $\kk$-quadric.
\end{proof}

\subsection{Sarkisov link with center at point}

Now we impose an additional assumption on the point~$x \in X$ and construct a Sarkisov link.

\begin{theorem}[cf.~{\cite{Takeuchi-1989}}] 
\label{th:sl:point}
Let $X$ be a Fano threefold satisfying~\eqref{ass:fano}.
Let $x$ be a $\kk$-point on~$X$ such that
\begin{equation}
\label{eq:no-line}
\rF_1(X,x) = \varnothing.
\end{equation} 
There exists a Sarkisov link~\eqref{eq:sl} with center at~$x$
and for~$\sigma_+$ the following hold:
\begin{enumerate}
\item 
\label{th:sl:point:g=7}
If $\g(X)=7$, then $X^+ \subset \PP^{6}$ is a smooth del Pezzo threefold of degree~$5$ 
and~$\sigma_+$ is the blowup of a smooth connected curve~$\Gamma$ of degree~$12$ and genus~$7$.

\item 
\label{th:sl:point:g=9}
If $\g(X)=9$, then $X^+ \subset \PP^{10}$ is a smooth Fano threefold of the same type 
and~$\sigma_+$ is the blowup of a $\kk$-point~$x_+ \in X^+$.
The map~$\sigma_+ \circ \psi \circ \sigma^{-1} \colon X \dashrightarrow X^+$ contracts the irreducible divisor~$D_x$ 
which is a unique member of the linear system~$|H-3x|$.
This divisor is a~$\kk$-rational surface.
\item 
\label{th:sl:point:g=10}
If $\g(X)=10$, then~$X^+ \simeq \PP^1$ and $\sigma_+$ is a del Pezzo fibration of degree~$6$. 
Moreover, the map $\sigma_+ \circ \tpsi \colon \tX \dashrightarrow X^+$ is given by the linear system~$|H-3E|$,
and its restriction to a general line in~$E \cong \PP^2$ is regular and has degree~$3$.
Finally, the flopping locus of~$\psi$ is the union of strict transforms of conics in~$X$ passing through~$x$.

\item 
\label{th:sl:point:g=12}
If $\g(X)=12$, then~$X^+ \simeq \PP^3$ and~$\sigma_+$ is the blowup of a smooth connected rational curve~$\Gamma$ of degree~$6$.
\end{enumerate}
\end{theorem}

\begin{proof}
We apply Lemma~\ref{lemma:sl-construction-point}.
Assume the morphism $\phi$ contracts a divisor $D$.
Since $\uprho(\tX) = 2$, we have $\uprho(\tX) = \uprho(\barX) + 1$, 
hence all curves in fibers of $\phi$ are proportional, hence intersect~$D$ negatively, hence are contained in~$D$.
Thus, away from $D$ the morphism $\phi$ is an isomorphism, hence~$\barX$ and~$\barE$ are nonsingular away from~$\phi(D)$.
Furthermore, since $\phi$ is the anticanonical morphism, we have~$\Sing(\barX) = \phi(D)$ {(because the contraction $\phi$ is crepant)}, hence 
\begin{equation}
\label{eq:sing-bare}
\Sing(\barE) \subset \Sing(\barX).
\end{equation} 

Assume $g \in \{10,12\}$.
By Lemma~\ref{lemma:linear-systems}(iii) we have $\dim |H - 3E| \ge g - 9$, 
hence by Lemma~\ref{lemma:sl-construction-point}(ii) the divisor~$D$ is a fixed component of the linear system~$|H - 3E|$.
On the other hand, by~\eqref{eq:td-class} we have $D \sim t(H - \tfrac{g-1}2E)$ for $t \in \ZZ$.
These two observations clearly contradict each other.

Next assume $g = 9$.
Again by Lemma~\ref{lemma:linear-systems}(iii) the divisor $D$ is a fixed component of $|H - 3E|$,
hence by~\eqref{eq:td-class} we have
\begin{equation}
\label{eq:td-class-9}
D \sim H - 4E.
\end{equation}
If the linear system $|H - 3E|$ is movable, then $|H - 3E - D| = |E|$ must be movable, which is absurd, hence~\mbox{$\dim |H - 3E| = 0$}.
By Lemma~\ref{lemma:double-projection}(iii) this implies that $\barE \subset \PP^{g - 3} = \PP^6$ is the Veronese surface.

On the other hand, $\bar X \subset \PP^6$ is a complete intersection of three quadrics by Lemma~\ref{lemma:ci3}.
Since the pullback to~$\tX$ of the hyperplane class of~$\barX$ is equal to $H - 2E$, 
it follows from~\eqref{eq:td-class-9} that there exists a hyperplane section of $\bar{X}$ that contains $\bar{E}$ with multiplicity~2,
which is impossible by Lemma~\ref{lemma:Veronese2}.
This contradiction proves that in the case $g = 9$ the map $\phi$ does not contract divisors.

Finally, assume $g = 7$.
This time by Lemma~\ref{lemma:linear-systems}(iii) and Lemma~\ref{lemma:sl-construction-point}(ii) 
the divisor $D$ is a fixed component of the linear system~$|3H - 7E|$,
hence by~\eqref{eq:td-class} we have
\begin{equation}
\label{eq:td-class-7}
\text{either $D \sim H - 3E$},
\qquad 
\text{or $D \sim 2H - 6E$},
\qquad 
\text{or $D \sim 3H - 9E$}.
\end{equation}
On the other hand, $\bar X \subset \PP^4$ is a quartic hypersurface and $\barE \subset \barX$ is a projection of the Veronese surface.
The first equality in~\eqref{eq:td-class-7} means that $\barE \subset \PP^4$ is contained in a hyperplane,
while the second means that it is contained in a quadric.
In both cases $\barE$ is singular, and its singular locus contains a line (see Lemma~\ref{lemma:Veronese1}).
By~\eqref{eq:sing-bare} we conclude that $\Sing(\barX)$ contains a line, which is impossible by Lemma~\ref{lemma:barx}.
Finally, if $D \sim 3H - 9E$ then the linear system $|3H - 7E - D| = |2E|$ must be movable, which is absurd.

Thus $\phi$ is a small morphism and by Lemma~\ref{lemma:sl-construction-point}(i) there exists a Sarkisov link~\eqref{eq:sl}
and it remains to describe the contraction~$\sigma_+$.
For $g \in \{9,10,12\}$ we refer to~\cite{Takeuchi-1989}.

Let us discuss the case $g = 10$ in more details.
Let $C \subset \tX$ be a flopping curve, so that $C$ is contracted by~$\phi$.
Then $(H - 2E) \cdot C = 0$.
If $C \cdot E > 1$, then $\phi(C)$ is a singular point of $\barE$.
But then a combination of~\eqref{eq:sing-bare} and Lemma~\ref{lemma:Veronese1} proves that the singular locus of $\barX$ contains a line,
which is impossible by Lemma~\ref{lemma:barx}.
Thus $E \cdot C = 1$, hence $H \cdot C = 2$, so $C$ is the strict transform of a conic passing through~$x$.

Furthermore, by~\cite{Takeuchi-1989} the map $\sigma_+ \circ \tpsi \colon \tX \dashrightarrow \PP^1$ is given by the linear system $|H - 3E|$,
which is base point free away from a finite number of flopping curves, neither of which is contained in~$E$.
Therefore, the restriction of this map to $E \cong \PP^2$ is given by a linear system of cubics
that has no fixed components, hence its restriction to a general line in~$E$ has degree~3.

Finally, assume $g = 7$. 
Arguing as in the proof of Theorem~\ref{th:sl:conic} we can check that additionally to~\ref{th:sl:point:g=7} 
we have the following numerical possibilities:
\begin{enumerate}
\item[(i$^{\prime}$)] 
the morphism $\sigma_+$ is a contraction of a divisor $D \sim 2(-K_{\tilde X^+})-E^+$ 
to a point on a smooth prime Fano threefold $X^+$ of index~1 and genus~7;
\item [(i$^{\prime\prime}$)]
the morphism $\sigma_+$ is a contraction of a divisor $D \sim 2(-K_{\tilde X^+})-E^+$
to a conic on a smooth prime Fano threefold $X^+$ of index~1 and genus~6.
\end{enumerate}
But in both these cases the image $\phi_+(D)$ is a quadric containing the surface \mbox{$\barE = \phi(E) = \phi_+(E^+)$}.
Then the surface $\barE$ cannot be smooth by Lemma~\ref{lemma:double-projection}(iii) and Lemma~\ref{lemma:Veronese0},
hence by Lemma~\ref{lemma:Veronese1} its singular locus $\Sing(\barE)$ contains a line, 
and by~\eqref{eq:sing-bare} the same is true for $\Sing(\barX)$,
which contradicts Lemma~\ref{lemma:barx-detailed} and Lemma~\ref{lemma:terminal-singularities}.
This contradiction shows that neither of the cases~(i$^\prime$) and~(i$^{\prime\prime}$) is possible.
\end{proof}

\section{Rationality and unirationality}
\label{section:rationality-constructions}

In this section we prove the sufficient conditions for rationality and unirationality of Fano threefolds satisfying~\eqref{ass:fano}.

We start with an elementary consequence of Theorem~\ref{th:sl:line}.

\begin{lemma}\label{corollary:rat:line}
If a Fano threefold $X$ satisfying~\eqref{ass:fano} contains a line defined over~$\kk$, then $X$ is $\kk$-rational.
\end{lemma}
\begin{proof}
Consider the Sarkisov link described in Theorem~\ref{th:sl:line}.
If $\g(X) = 7$ we deduce from Theorem~\ref{theorem:dp5s}
that $\tX^+$ is rational over $X^+ \cong \PP^1$, hence $X$ is $\kk$-rational.
For $\g(X) \in \{9,10,12\}$ we conclude that $X$ is birational to $X^+$ which is a $\PP^3$, a $\kk$-rational quadric, 
or a quintic del Pezzo threefold, respectively, hence is rational (in the last case we use Theorem~\ref{theorem:v5}).
\end{proof}

\subsection{Unirationality}
\label{subsection:unirationality}

Now we will prove unirationality of $X$ under the assumption $X(\kk) \ne \varnothing$.
We fix a $\kk$-point $x$ and consider all possibilities for the Hilbert scheme $\rF_1(X,x)$ of lines through~$x$.

\begin{lemma}
\label{lemma:f1xx}
If a Fano threefold $X$ satisfies~\eqref{ass:fano} and $x \in X$ is a $\kk$-point 
then one of the following possibilities hold:
\begin{enumerate}
\item\label{cases:0}
$\rF_1(X,x) = \varnothing$;
\item\label{cases:1}
$\rF_1(X,x)$ contains a $\kk$-point;
\item\label{cases:2}
$\rF_1(X,x)$ is $\kk$-irreducible and reduced of length~$2$.
\item\label{cases:3}
$\rF_1(X,x)$ is $\kk$-irreducible and reduced of length~$3$.
\end{enumerate}
\end{lemma}

\begin{proof}
By Corollary~\ref{cor:length-f1} the scheme $\rF_1(X,x)$ is finite and its length is bounded by~3.
Therefore, the scheme~$\rF_1(X,x)$ has a $\kk$-point, unless the scheme is reduced of length~2 or~3 
and the Galois group action on $\rF_1(X,x)_\bkk$ is transitive.
\end{proof}

Below we prove $\kk$-unirationality of~$X$ in all cases of Lemma~\ref{lemma:f1xx}.
The easiest case is~\ref{cases:1}.

\begin{lemma}
\label{lemma:f1xx-1}
Let $X$ be a Fano threefold satisfying~\eqref{ass:fano} with a $\kk$-point $x$ such that $\rF_1(X,x)$ contains a $\kk$-point.
Then $X$ is $\kk$-rational.
\end{lemma}
\begin{proof}
The line on $X$ corresponding to a $\kk$-point of $\rF_1(X,x)$ is defined over~$\kk$, hence Lemma~\ref{corollary:rat:line} applies.
\end{proof}

The next case is~\ref{cases:0}.

\begin{proposition}
\label{proposition:f1xx-0}
Let $X$ be a Fano threefold satisfying~\eqref{ass:fano} with a $\kk$-point $x$ such that 
\begin{equation*}
\rF_1(X,x) = \varnothing.
\end{equation*}
Then $X$ is $\kk$-unirational.
\end{proposition}
\begin{proof}
First, assume $\g(X) \in \{7,12\}$.
It follows from Theorem~\ref{th:sl:point}\ref{th:sl:point:g=7} and~\ref{th:sl:point:g=12} 
that $X$ is birational to a quintic del Pezzo threefold, or to $\PP^3$, hence is $\kk$-rational 
(by Theorem~\ref{theorem:v5} in the former case).

Next, assume $\g(X) = 10$.
Then Theorem~\ref{th:sl:point}\ref{th:sl:point:g=10} shows that $X$ is birational to a sextic del Pezzo fibration $\sigma_+ \colon \tX^+ \to X^+$ 
and a general line in $E \cong \PP^2$ gives a $\kk$-rational curve in~$\tX^+$ that dominates~$X^+ \cong \PP^1$.
The general fiber of $\sigma_+$ is smooth because $\tX^+$ is smooth.
We conclude that~$\tX^+$ is $\kk$-unirational by Lemma~\ref{lemma:unirational-fibration}.

Finally, assume $\g(X) = 9$.
In this case Theorem~\ref{th:sl:point}\ref{th:sl:point:g=9} provides an irreducible $\kk$-rational divisor $D_x \in |H - 3x|$.
This provides many new $\kk$-points on $X$, and we complete the proof of unirationality by a ``spreading out'' argument.

Let $D_x^\circ \subset D_x$ be the open subset of points of $D_x$ not lying on lines (it contains~$x$, hence is non-empty).
Applying to each point~\mbox{$x' \in D_x^\circ$} the above construction, 
we obtain a family of $\kk$-rational surfaces $\{D_{x'}\}_{x' \in D_x^\circ}$ in~$X$.
The surface $D_{x'}$ has multiplicity at least~3 at point $x'$, hence this is not a constant family of surfaces.
This means that the surface $D_\eta$ corresponding to the general point $\eta$ of $D_x^\circ$ dominates~$X$.
Since $D_\eta$ is rational over the residue field of $\eta$, 
which is a pure transcendental extension of~$\kk$ (since $D_x$ is $\kk$-rational),
we conclude that $D_\eta$ is rational over~$\kk$, hence $X$ is $\kk$-unirational.
\end{proof}

Now we discuss case~\ref{cases:2} of Lemma~\ref{lemma:f1xx}.

\begin{proposition}
\label{proposition:f1xx-2}
Let $X$ be a Fano threefold satisfying~\eqref{ass:fano} with a $\kk$-point $x$ 
such that~$\rF_1(X,x)$ is reduced and $\kk$-irreducible of length~$2$.
Then $X$ is $\kk$-unirational.
\end{proposition}

\begin{proof}
We consider the union of the two lines corresponding to $\bkk$-points of $\rF_1(X,x)_{\mathrm{red}}$ 
as a $\kk$-irreducible singular conic and apply the Sarkisov link of Theorem~\ref{th:sl:conic}.
If $\g(X) \in \{7,12\}$, it follows that~$X$ is birational to a $\kk$-rational quadric, hence is $\kk$-rational.

Now assume $\g(X) = 9$.
Then, instead of the Sarkisov link we consider the double projection from the point $x$.
We obtain a birational isomorphism of ~$X$ 
onto a complete intersection of three quadrics $\barX \subset \PP^6$ (Lemma~\ref{lemma:ci3}) which
contains the smooth $\kk$-rational quadric surface $\barE \subset \barX$ 
defined over~$\kk$.
Note that the linear span $\langle \barE \rangle = \PP^3$ of $\barE$ is not contained in $\barX$ (because~$\barX$ is irreducible),
but in the net of quadrics defining~$\barX$ there is a pencil of quadrics containing~$\langle \barE \rangle$.
In other words, we can choose three quadrics $Q_0,Q_1,Q_2 \subset \PP^6$, such that 
\begin{equation*}
\barX = Q_0 \cap Q_1 \cap Q_2,
\qquad 
\langle \barE \rangle \subset Y := Q_1 \cap Q_2,
\quad\text{and}\quad 
Q_0 \cap \langle \barE \rangle = \barE.
\end{equation*}

We consider the linear projections
\begin{equation*}
\pi_Y \colon \Bl_{\langle \barE \rangle}(Y) \larrow \PP^2
\quad\text{and}\quad
\pi_\barX \colon \Bl_\barE(\barX) \larrow \PP^2
\end{equation*}
out of the linear span of $\barE$
and check that the general fiber of~$\pi_Y$ is $\PP^2$, and that of~$\pi_\barX$ is a conic.

Indeed, a general fiber of~$\pi_Y$ can be described as follows.
For a point~$t \in \PP^2$ let~$\Pi_t \subset \PP^6$ be the corresponding 4-space containing the linear span of~$\barE$.
The quadrics~$Q_1$ and~$Q_2$ intersects~$\Pi_t$ along quadrics, containing the hyperplane~$\langle \barE \rangle \subset \Pi_t$, 
hence we can write
\begin{equation*}
Q_1 \cap \Pi_t = \langle \barE \rangle \cup D_{1t},
\qquad 
Q_2 \cap \Pi_t = \langle \barE \rangle \cup D_{2t},
\end{equation*}
where $D_{it} \subset \Pi_t$ are hyperplanes.
Therefore, every non-empty fiber $Y_t$ of $\pi_Y$ can be described as
\begin{equation*}
Y_t = D_{1t} \cap D_{2t},
\end{equation*}
i.e., $Y_t$ is a linear subspace in $\Pi_t$ of codimension at most~2.
Thus, if the map $\pi_Y$ is dominant, its general fiber is a $\PP^2$, and otherwise its general fiber is $\PP^n$ with $n \ge 3$.

Now, the general fiber $X_t$ of $\pi_\barX$ is the intersection
\begin{equation*}
X_t = Y_t \cap Q_0.
\end{equation*}
If the map $\pi_Y$ is not dominant, $X_t$ is a quadric of dimension $n - 1 \ge 2$.
Note that any such quadric is covered by lines, hence $\barX$ is covered by lines, which is impossible by Lemma~\ref{lemma:barx-detailed}.
Therefore, $\pi_Y$ is dominant, its general fiber is $\PP^2$, and the general fiber of $\pi_\barX$ is a conic in it.

Furthermore, $Y_t \cap \langle \barE \rangle$ is a hyperplane section of $Y_t$, i.e., a line, 
and $X_t \cap \barE$ is the intersection of this line with the quadric $Q_0$, hence is nonempty.
This means that the exceptional divisor of the blowup~$\Bl_\barE(\barX)$ dominates the base $\PP^2$ of the conic bundle~$\pi_\barX$.
It remains to note that since~$\barE \subset \barX$ is a Weil divisor, 
it is birational to the exceptional divisor of~$\Bl_\barE(\barX)$,
hence the latter
is a $\kk$-rational surface.
So, applying Lemma~\ref{lemma:unirational-fibration}\ref{lemma:unirational-fibration:conic-bundle}, 
we conclude that $\Bl_\barE(\barX)$ is $\kk$-unirational, and hence so is $X$.

Finally, assume $\g(X) = 10$.
Then $X$ is birational to a conic bundle $\tX^+ \to X^+$ over~\mbox{$X^+ \cong \PP^2$},
Moreover, $\tX^+$ has a single ordinary double point $\tilde{x}_+$ and $x_+ = \sigma_+(\tilde{x}_+) \subset X^+$
is a singular point of the discriminant curve $\Delta \subset X^+$ of this conic bundle.
Let $\hX^+ \to \tX^+$ be the blowup of~$\tilde{x}_+$.
Following the argument of~\cite[Lemma~8]{Avilov:cb}, we apply the relative MMP to $\hX^+$ over $X^+$.
Since~$-K_{\hX^+}$ is nef over $X^+$ (\cite[Lemma~7]{Avilov:cb}),
it follows that there is a commutative diagram 
\begin{equation*}
\xymatrix@R=3ex{
\hX^+ \ar@{<-->}[rr]^{\hat\psi} \ar[d]_\rho &&
\hX^\natural \ar[d]^{\rho_\natural}
\\
\tX^+ \ar[dr]^{\sigma_+} &&
X^\natural \ar[dl]_{\sigma_\natural} \ar[dr]^\pi
\\
& X^+ \ar@{-->}[rr] &&
\PP^1,
}
\end{equation*}
where $\rho_\natural$ is a flat conic bundle and~$\sigma_\natural$ is a divisorial contraction.
The surface~$X^\natural$ is smooth by~Theorem~\ref{theorem:smooth-contractions}\ref{theorem:smooth-contractions:conic-bundle}
and the map~$\sigma_\natural$
is an isomorphism over the complement of~$x_+$, hence it is the blowup of~$x_+$.
Finally, the strict transform~$D^\natural := \hat\psi_*(D^+)$ of the exceptional divisor $D^+ \subset \hX^+$ of the blowup~$\rho$ 
dominates the exceptional line~$L = \sigma_\natural^{-1}(x_+)$ (because~$\rho_\natural$ is flat),
and the discriminant curve~$\Delta^\natural \subset X^\natural$ of the conic bundle~$\rho_\natural$ 
does not contain~$L$ (cf.~\cite[Lemma~10.10]{Prokhorov:UMN}), hence it is equal to the strict transform of~$\Delta$ under~$\sigma_\natural$.

Now we consider the morphism~$\pi \colon X^\natural \to \PP^1$ induced by the linear projection of~$X^+ = \PP^2$ out of~$x_+$,
and the composition
\begin{equation*}
\pi \circ \rho_\natural \colon \hX^\natural \to \PP^1.
\end{equation*}
Note that~$L$ is dominant over~$\PP^1$, hence so is~$D^\natural$.
Note also that the point~$\tilde{x}_+$ is infinitesimally rational (Theorem~\ref{th:sl:conic}\ref{th:sl:conic:g=10}), 
i.e., $D^+$ is a $\kk$-rational surface.
Thus~$\PP^1$ is dominated by a rational surface in~$\hX^\natural$.

On the other hand, the fiber of $\pi \circ \rho_\natural$ over a point $t \in \PP^1$ is a conic bundle over $\pi^{-1}(t) = \PP^1$. 
Since $x_+$ is a singular point of $\Delta$, we have
\begin{equation*}
\Delta^\natural \cdot \pi^{-1}(t) = (\sigma_{\natural})^{-1}_*(\Delta) \cdot \pi^{-1}(t) \le \Delta \cdot \sigma_\natural(\pi^{-1}(t)) - 2 = 4 - 2 = 2,
\end{equation*}
so this conic bundle has at most two singular fibers.
Now we apply Lemma~\ref{lemma:unirational-fibration}\ref{lemma:unirational-fibration:conic-bundle} and conclude that $\hX^\natural$ is $\kk$-unirational.
By construction $X$ is birational to $\hX^\natural$, hence it is also $\kk$-unirational.
\end{proof}

Finally, we consider the case~\ref{cases:3} of Lemma~\ref{lemma:f1xx}.

\begin{proposition}
\label{proposition:f1xx-3}
Let $X$ be a Fano threefold satisfying~\eqref{ass:fano} with a $\kk$-point $x$ 
such that~$\rF_1(X,x)$ is reduced
and $\kk$-irreducible of length~$3$.
Then $X$ is $\kk$-unirational.
\end{proposition}
\begin{proof}
Consider the image $\barX \subset \PP^{g-3}$ 
of the double projection $\phi \colon \hX \to \barX$ described in~\S\ref{subsection:double-projection}.
By Lemma~\ref{lemma:double-projection} the variety~$\bar X$ contains a $\kk$-plane $\barE \cong \PP^2$ 
and the anticanonical class of~$\bar X$ is the hyperplane class (Lemma~\ref{lemma:barx-detailed}).
Let~$Y$ be the blowup of~$\barX$ along~$\barE$ and let~$S$ be the strict transform of~$\bar E$.
Then we have a commutative diagram
\begin{equation}
\label{eq:barx-y}
\vcenter{\xymatrix{
\barE \ar[r] &
\barX \ar@{-->}[dr] &&
Y \ar[ll] \ar[dl]^\pi &
S \ar[l]
\\
&& \PP^{g-6},
}}
\end{equation}
where the middle horizontal arrow is the blowup map 
and the dashed arrow is the restriction of the linear projection~$\PP^{g-3} \dashrightarrow \PP^{g-6}$ with center at~$\barE$.
Since~$\barE$ is a Weil divisor on~$\barX$, the blowup map~$Y \to \barX$ is small.
Moreover, the exceptional divisor~$S$ of the blowup is a Cartier divisor, 
which is birational to~$\barE$, hence is a $\kk$-rational surface.

First, assume $\g(X) \ge 9$. 
We check below that the assumptions of Proposition~\ref{proposition:q-fano} are satisfied by the pair $(Y,S)$.
Indeed, $Y$ has terminal singularities by Lemma~\ref{lemma:small-terminality}, 
because $\barX$ has terminal singularities (Lemma~\ref{lemma:barx-detailed}), 
and the morphism $Y \to \barX$ is small.
Furthermore, the linear system~\mbox{$|- K_Y - S|$} is the pullback of $|\cO(1)|$ from $\PP^{g-6}$, hence it is base point free and Cartier,
hence the pair~$(Y,|-K_Y - S|)$ is terminal by Lemma~\ref{lemma:terminal-pair}\ref{lemma:terminal-pair3}.
Finally, the morphism $\pi \colon Y \to \PP^{g-6}$ given by this linear system is generically finite, 
because so is the map $\barX \dashrightarrow \PP^{g-6}$ (Corollary~\ref{cor:projection-birational}).
Thus, Proposition~\ref{proposition:q-fano} implies that $Y$ is $\kk$-unirational.
Since by construction~$Y$ is birational to~$X$, the latter is $\kk$-unirational as well.

Now assume $\g(X) = 7$, so that $\barX$ is a quartic threefold in~$\PP^4$, and $\barE$ is a $\kk$-plane on it.
Again, we consider the diagram~\eqref{eq:barx-y}.
In this case the projective space at the bottom of the diagram is just a~$\PP^1$.
Let $t \in \PP^1$ and let $\barH_t \subset \PP^4$ be the corresponding hyperplane containing $\barE$.
Then 
\begin{equation*}
\barX \cap \barH_t = \barE \cup Y_t,
\end{equation*}
where $Y_t$ is a cubic surface and the fiber of the map~$\pi \colon Y \to \PP^1$ over~$t$ is~$Y_t$, 
so that~$\pi$ is a family of cubic surfaces.
Note that a general~$Y_t$ is not a cone, because~$\barX$ is not covered by lines by Lemma~\ref{lemma:barx-detailed}.
On the other hand, the surface~$S$ dominates~$\PP^1$, 
because the divisor~$\barE$ dominates~$\PP^1$ under the map~$\barX \dashrightarrow \PP^1$
(indeed, its fiber over a point~$t \in \PP^1$ is the cubic curve~\mbox{$\barE \cap Y_t$}).
Therefore, Lemma~\ref{lemma:unirational-fibration}\ref{lemma:unirational-fibration:cubic} proves that~$Y$ is $\kk$-unirational,
and again, $X$ is $\kk$-unirational as well.
\end{proof}

A combination of Lemmas~\ref{lemma:f1xx} and~\ref{lemma:f1xx-1} 
with Propositions~\ref{proposition:f1xx-0}, \ref{proposition:f1xx-2}, and~\ref{proposition:f1xx-3}
finally proves $\kk$-unirationality of a Fano threefold satisfying~\eqref{ass:fano} which has a $\kk$-point:

\begin{corollary}
\label{corollary:unirat}
Let $X$ a Fano threefold satisfying~\eqref{ass:fano}.
If $X(\kk)\neq \emptyset$ then $X$ is $\kk$-unirational.
\end{corollary}

\subsection{Rationality}

Now we are ready to prove the sufficient conditions for rationality of Fano threefolds satisfying~\eqref{ass:fano}.

For $\g(X) \in \{7,12\}$ the result was essentially proved in Proposition~\ref{proposition:f1xx-0}.

\begin{proposition}
\label{proposition:rationality-7-12}
Let $X$ be a prime Fano threefold of genus~$\g(X) \in \{7,12\}$.
If $X(\kk) \ne \varnothing$ then~$X$ is $\kk$-rational.
\end{proposition}
\begin{proof}
As we already know, $X$ is $\kk$-unirational (see Corollary~\ref{corollary:unirat}),
hence the set of $\kk$-points on~$X$ is Zariski dense.
In particular, $X$ has a $\kk$-point $x$ such that $\rF_1(X,x) = \varnothing$.
We choose such a point, apply Theorem~\ref{th:sl:point},
and conclude that $X$ is birational to a quintic del Pezzo threefold, or to $\PP^3$, 
hence is $\kk$-rational (by Theorem~\ref{theorem:v5} in the former case).
\end{proof}

Next, we consider the genus~9 case.

\begin{proposition}
\label{proposition:rationality-9}
Let $X$ be a prime Fano threefold of genus~$\g(X) = 9$.
If $\rF_3(X)(\kk) \ne \varnothing$ then~$X$ is $\kk$-rational.
\end{proposition}
\begin{proof}
Let $C \subset X$ be a cubic curve of arithmetic genus~0 defined over~$\kk$.
Then~$C$ is one of the curves listed in Lemma~\ref{lemma:f3}.
Below we consider all of its possible types.

First, assume that $C$ is a smooth curve (type~\ref{case:f3:smooth}).
Then the intersection of $X$ with the linear span of $C$ is either $C$ itself, or the union of $C$ and a bisecant line.
In the latter case, the line is defined over~$\kk$, hence $X$ is $\kk$-rational by Lemma~\ref{corollary:rat:line}.

So, assume that $C$ is smooth and has no bisecant lines in~$X$.
We then apply the Sarkisov link of Theorem~\ref{th:sl:cubic}.
It shows that $X$ is birational to a quintic del Pezzo threefold, hence is $\kk$-rational by Theorem~\ref{theorem:v5}.

Next, assume that $C$ is of types~\ref{case:f3:2-1}, \ref{case:f3:chain}, or~\ref{case:f3:multiple}.
Then one of its irreducible components (the middle one in the case~\ref{case:f3:chain}) is a Galois-invariant line, 
hence it is defined over $\kk$, hence $X$ is $\kk$-rational by Lemma~\ref{corollary:rat:line}.

Finally, assume that $C$ is of type~\ref{case:f3:trident}, i.e., is the union of three Galois-conjugate lines meeting at a point $x$. 
The point~$x$ is Galois-invariant, hence is defined over $\kk$, and $\rF_1(X,x)$ is reduced of length~3 (Corollary~\ref{cor:length-f1}).
By Lemma~\ref{lemma:ci3} the double projection $X \dashrightarrow \barX$ out of~$x$ is a birational isomorphism
onto a complete intersection $\barX \subset \PP^6$ of three quadrics which contains the plane~\mbox{$\barE \subset \barX$}.
By Corollary~\ref{cor:projection-birational} the projection 
\begin{equation*}
\pi_\barE \colon \barX \dashrightarrow \PP^3
\end{equation*}
is birational onto its image.
Therefore, $\barX$ is birational to $\PP^3$, hence it is $\kk$-rational and $X$ is $\kk$-rational as well.
\end{proof}

Finally, we consider the genus~10 case.

\begin{proposition}
\label{proposition:rationality-10}
Let $X$ be a prime Fano threefold of genus~$\g(X) = 10$.
If $X(\kk) \ne \varnothing$ and~$\rF_2(X)(\kk) \ne \varnothing$ then~$X$ is $\kk$-rational.
\end{proposition}
\begin{proof}
Let $C$ be a conic on $X$ defined over~$\kk$.
If $C$ is a double line, or is reducible over~$\kk$, a component of $C$ is a $\kk$-line on~$X$, 
hence $X$ is $\kk$-rational by Lemma~\ref{corollary:rat:line}.

Assume now that $C$ is reduced and $\kk$-irreducible. 
Let us check that $X$ has a $\kk$-point $x$ such that
\begin{enumerate}
\item 
\mbox{$\rF_1(X,x) = \varnothing$}, 
\item 
$x$ does not lie on any conic intersecting~$C$.
\end{enumerate}
Indeed, (i) holds for a general point of~$X$ because $\dim (\rF_1(X)) = 1$.
Similarly, conics intersecting~$C$ are parameterized by a curve 
(see~\cite[Lemma~4.2.6(i)]{IP} for the case where $C$ is smooth 
and Theorem~\ref{th:sl:conic}\ref{th:sl:conic:g=10} for the case where $C$ is singular),
hence sweep a surface, hence~(ii) also holds for a general~$x$.
But $X$ is $\kk$-unirational by Corollary~\ref{corollary:unirat} since $X(\kk) \ne \varnothing$, 
therefore the set of $\kk$-points on~$X$ is Zariski dense, 
and so there is a $\kk$-point for which the conditions~(i) and~(ii) hold.

We choose such a point and apply Theorem~\ref{th:sl:point}\ref{th:sl:point:g=10}.
We conclude that $X$ is birational to a sextic del Pezzo fibration~\mbox{$\tX^+ \to X^+ = \PP^1$}.

By Theorem~\ref{th:sl:point}\ref{th:sl:point:g=10} the flopping locus of the flop~$\tpsi$ 
is the union of conics passing through~$x$, and none of these intersects $C$ by the condition~(ii) above.
Therefore, the conic $C$ does not intersect the flopping locus of~$\tpsi$.
Since the map $\tX^+ \dashrightarrow X^+$ is given by the linear system $|H - 3E|$ 
it follows that its restriction to the strict transform of~$C$ in~$\tX$ is regular and has degree~2.

Thus, $C$ provides a 2-section for the del Pezzo fibration $\sigma_+ \colon \tX^+ \to \PP^1$.
On the other hand, by Theorem~\ref{th:sl:point}\ref{th:sl:point:g=10} a general line in $E$ provides a 3-section for $\sigma_+$.
Therefore, $\tX^+$ is rational over~$\PP^1$ 
by~\cite[Proposition~8]{Add-Has-Tsch-VA}, 
hence it is $\kk$-rational, and hence $X$ is $\kk$-rational as well.
\end{proof}

\section{Obstructions to rationality}
\label{section:obstructions}

In the previous sections we showed that the criteria of Theorem~\ref{theorem:main}\ref{main:v4x16x18} 
for rationality of threefolds~$V_4$, $X_{16}$, and~$X_{18}$ are sufficient.
Starting from this section we prove they are also necessary.
Since the existence of a $\kk$-point is a necessary condition by Lemma~\ref{lemma:unirationality-point}, 
we only have to show that the existence of a line, a conic, and a cubic curve 
on the threefolds $V_4$, $X_{18}$, and $X_{16}$, respectively, are necessary conditions for rationality.
In this section we prove a general result of this sort, and in the remaining part of the paper 
we check that it applies to the varieties of our interest.

\subsection{Rationality criterion}
\label{subsection:obstruction}

In this section we state a general criterion of rationality.
We assume that $X$ is a smooth projective geometrically rationally connected threefold
and $H \in \Pic(X)$ is a very ample divisor class (defined over~$\kk$).
To state the criterion we need some notation.

Let $\Jac(X)$ denote the intermediate Jacobian of $X$ (we refer to~\cite[\S2]{BW} for details).
Recall that~$\Jac(X)$ is an abelian variety defined over~$\kk$ endowed with a natural principal polarization and with the Abel--Jacobi map 
\begin{equation*}
\AJ \colon \CH^2_{\mathrm{alg}}(X_\bkk) \larrow \Jac(X_\bkk) \cong \Jac(X)_\bkk
\end{equation*}
from the group of algebraically trivial codimension~2 cycles on $X_\bkk$ 
to the group of~$\bkk$-points of~$\Jac(X)$.

As before we denote by~$\rF_d(X)$ the Hilbert scheme of curves of degree~$d$ with respect to~$H$ 
and arithmetic genus~$0$ (see~\S\ref{subsection:Hilbert}). 
If~$\rF_d(X_\bkk)$ is connected, the cycles associated with all curves parameterized by~$\rF_d(X_\bkk)$ are algebraically equivalent,
hence for each curve $[C_0] \in \rF_d(X_{\bkk})$ the map~$\AJ$ induces the (shifted) Abel--Jacobi map 
\begin{equation*}
\AJ = \AJ_{C_0} \colon \rF_d(X_\bkk) \larrow \Jac(X_\bkk),
\qquad 
[C] \longmapsto \AJ([C] - [C_0]).
\end{equation*}
Changing the curve $C_0$ results in a composition of this map with a translation in $\Jac(X_\bkk)$.
Furthermore, by the universal property of the Albanese variety, the map $\AJ_{C_0}$ factors as the composition
\begin{equation}
\label{eq:alpha-d}
\rF_d(X_\bkk) \xrightarrow{\ \alb_{C_0}\ } \Alb(\rF_d(X_\bkk)) \xrightarrow{\ \alpha_d\ } \Jac(X_\bkk),
\end{equation}
where $\alpha_d$ is a morphism of abelian varieties that does not depend on the choice of $C_0$ and is Galois-invariant.
In particular, $\alpha_d$ is defined over~$\kk$ as a morphism $\alpha_d \colon \Alb(\rF_d(X)) \to \Jac(X)$.

\begin{definition}
\label{def:nice-degenerations}
We will say that the Hilbert scheme $\rF_d(X_\bkk)$ {\sf admits nice degenerations} if there is a point $[C] \in \rF_d(X_\bkk)$
such that every irreducible component $C' \subset C$ taken with its reduced scheme structure 
corresponds to a point of $\rF_{d'}(X_{\bkk})$ for some $d' < d$, i.e., $[C'_{\mathrm{red}}] \in \rF_{d'}(X_{\bkk})$.
\end{definition}

Now we can formulate the main result of this section.

\begin{theorem}
\label{theorem:obstruction-general}
Let $X$ be a smooth projective geometrically rationally connected $\kk$-variety with a very ample divisor class~$H \in \Pic(X)$.
Assume that the following conditions are satisfied for~$X_\bkk$:
\begin{itemize}
\item[(T)] 
There exists $\dtor \in \ZZ_{>0}$ such that the Abel--Jacobi map 
\begin{equation*}
\AJ \colon \rF_{\dtor}(X_\bkk) \larrow \Jac(X_\bkk)
\end{equation*}
is an isomorphism.
\item[(C)] 
There exists $\dc \in \ZZ_{>0}$ such that the Hilbert scheme $\rF_{\dc}(X_\bkk)$ 
admits a map~$\rF_{\dc}(X_\bkk) \to \Gamma_\bkk$ with rationally connected fibers 
to a curve $\Gamma_\bkk$ of genus $g(\Gamma_\bkk) > 1$, 
such that the induced map 
\begin{equation*}
\Pic^0(\Gamma_\bkk) \cong \Alb(\rF_{\dc}(X_\bkk)) \xrightarrow{\ \alpha_{\dc}\ } \Jac(X_\bkk)
\end{equation*}
is an isomorphism of principally polarized abelian varieties.
\item[(H)]
For each $1 \le d \le \max(\dc,\dtor)$ the Hilbert scheme $\rF_d(X_\bkk)$ is connected and when $d > 1$
it admits nice degenerations.
\item[(N)] 
The maximal divisor of $2g(\Gamma_\bkk) - 2$ coprime to $\dc$ divides $\dtor$.
\end{itemize}
If $X$ is $\kk$-rational then $\rF_{\dtor}(X)(\kk) \neq \varnothing$, 
i.e., $X$ contains a curve of degree~$\dtor$ and arithmetic genus~$0$ defined over~$\kk$.
\end{theorem}

Note that all the conditions of the theorem are \emph{geometric}, i.e., they only depend on $X_\bkk$.

\subsection{Rationality criterion of Benoist--Wittenberg}

The proof of Theorem~\ref{theorem:obstruction-general} is based on the results of~\cite{BW2};
we briefly remind these results in this subsection.

Let $Y$ be a smooth projective $\kk$-variety of dimension~$n$.
For each $1 \le k \le n$ we consider the N\'eron--Severi group $\NS^k(Y_\bkk) = \CH^k(Y_\bkk) / \sim_{\mathrm{alg}}$
of algebraic equivalence classes of codimension~$k$ cycles on $Y_\bkk$ and the group~$\NS^k(Y_\bkk)^\Gal$
of Galois-invariant classes in $\NS^k(Y_\bkk)$.

If~$Y$ is geometrically connected
the degree map gives an isomorphism $\NS^n(Y_\bkk)^\Gal \cong \ZZ$,
and for each $d \in \ZZ$ there is a torsor $\Alb_d(Y)$ 
over the Albanese variety $\Alb(Y)$ and a Galois-invariant Albanese morphism 
\begin{equation*}
\alb_d \colon \CH^n_d(Y_\bkk) \larrow \Alb_d(Y)_\bkk
\end{equation*}
from the coset of 0-cycles of degree~$d$
which is universal among morphisms to torsors over abelian varieties.
In particular, for $d = 1$ we obtain the morphism
\begin{equation*}
\alb_1 \colon Y \larrow \Alb_1(Y).
\end{equation*}
Note also that when $Y = \Gamma$ is a geometrically connected curve, 
we have
\begin{equation}
\label{eq:alb-pic}
\Alb(\Gamma) \cong \Pic^0(\Gamma),
\qquad 
\Alb_d(\Gamma) \cong \Pic^d(\Gamma).
\end{equation} 

On the other hand, if~$X$ is a threefold then
for each class $\gamma \in \NS^2(X_\bkk)^\Gal$ Benoist and Wittenberg define 
a torsor $\Jac_\gamma(X)$ over the intermediate Jacobian $\Jac(X)$ and a Galois-invariant morphism 
\begin{equation*}
\AJ_\gamma \colon \CH^2_\gamma(X_\bkk) \larrow \Jac_\gamma(X)_\bkk
\end{equation*}
from the coset of codimension~2 cycles algebraically equivalent to $\gamma$,
so that the morphism $\AJ_\gamma$ is universal among morphisms to torsors over abelian varieties.

Note that by definition $\Jac_0(X) = \Jac(X)$, and for any $\gamma_1,\gamma_2 \in \NS^2(X_\bkk)^\Gal$ 
the universal property implies the equality
\begin{equation}
\label{eq:jac-gamma-sum}
[\Jac_{\gamma_1 + \gamma_2}(X)] = [\Jac_{\gamma_1}(X)] + [\Jac_{\gamma_2}(X)],
\end{equation} 
of $\Jac(X)$-torsor classes. 
Here for a torsor~$T$ over an abelian $\kk$-variety~$A$ we denote by~$[T]$ its class in the Weil--Ch\^atelet group~$H^1(\Gal,A)$ of~$A$-torsors.

The following result was proved by Benoist and Wittenberg.

\begin{theorem}[{\cite[Theorem~3.10]{BW2}}]
\label{theorem:bw}
Let $X$ be a smooth projective threefold over~$\kk$.
If~$X$ is~$\kk$-rational and there is an isomorphism~$\Jac(X) \cong \Pic^0(\Gamma)$ of principally polarized abelian varieties, 
where~$\Gamma$ is a curve of genus~$g > 1$ then for any~$\gamma \in \NS^2(X_\bkk)^\Gal$, 
there is an integer~$m$ such that there is an isomorphism
\begin{equation*}
\Jac_\gamma(X) \cong \Pic^{m}(\Gamma)
\end{equation*}
of torsors over $\Jac(X) \cong \Pic^0(\Gamma)$.
\end{theorem}

In the next subsection we use this result to prove Theorem~\ref{theorem:obstruction-general}.

\subsection{Proof of Theorem~\ref{theorem:obstruction-general}}

To prove the theorem we need some preparations.

First, we derive some consequences of condition~(H).
Assume the Hilbert scheme $\rF_d(X)$ is geometrically connected.
Then all curves in $X$ parameterized by the scheme $\rF_d(X)$ are algebraically equivalent
and the class of these curves in $\NS^2(X_\bkk)$ is Galois-invariant.
We denote the corresponding torsor over $\Jac(X)$ simply by $\Jac_d(X)$.
The Abel--Jacobi map then restricts to the map 
\begin{equation*}
\AJ_d \colon \rF_d(X) \larrow \Jac_d(X).
\end{equation*}

\begin{lemma}
\label{lemma:jac-d}
Assume condition~\textnormal{(H)} of Theorem~\textup{\ref{theorem:obstruction-general}} is satisfied.
Then we have an equality of torsors over $\Jac(X)$
\begin{equation}
\label{eq:jac-d-mult:a}
[\Jac_d(X)] = d[\Jac_1(X)],
\end{equation}
for all $1 < d \le \max\{\dc,\dtor\}$.
\end{lemma}
\begin{proof}
If $d > 1$ and $\rF_d(X_{\bkk})$ admits nice degenerations, there is a point $[C] \in \rF_d(X_{\bkk})$ such that 
\begin{equation*}
C = C_1 \cup \dots \cup C_k,
\end{equation*}
where $C_i$ has multiplicity $m_i$ and $[(C_i)_{\mathrm{red}}] \in \rF_{d_i}(X_{\bkk})$ with $1 \le d_i < d$ and $\sum m_id_i = d$.
It follows that we have an equality of classes $[C] = \sum m_i[C_i]$ in $\NS^2(X_\bkk)^{\Gal}$, therefore 
\begin{equation*}
[\Jac_d(X)] = \sum m_i[\Jac_{d_i}(X)]
\end{equation*}
by~\eqref{eq:jac-gamma-sum}.
Iterating this argument, we eventually obtain~\eqref{eq:jac-d-mult:a}.
\end{proof}

On the other hand, we can use condition~(C) of Theorem~\ref{theorem:obstruction-general} 
to get a description of the torsor~$\Jac_{\dc}(X)$.

\begin{proposition}
\label{proposition:jac-dc}
Assume condition~\textnormal{(C)} of Theorem~\textup{\ref{theorem:obstruction-general}} is satisfied 
and the scheme~$\rF_{\dc}(X_\bkk)$ is connected.
Then there is a geometrically connected $\kk$-curve $\Gamma$ and isomorphisms
\begin{equation}
\label{eq:jac-pic-iso}
\Jac(X) \cong \Pic^0(\Gamma)
\qquad\text{and}\qquad
[\Jac_{\dc}(X)] = [\Pic^1(\Gamma)].
\end{equation} 
of principally polarized abelian varieties and torsors over them.
\end{proposition}
\begin{proof}
Let $\xi \colon \rF_{\dc}(X) \to Z$ be a maximal rationally connected fibration of~$\rF_{\dc}(X)$,
which recall is defined over~$\kk$ and is unique up to birational equivalence~\cite[Theorem~IV.5.4]{Kollar96:curves}; 
in particular, one can assume~$Z$ to be normal.
Then~$\xi_\bkk \colon \rF_{\dc}(X_\bkk) \to Z_\bkk$ is a maximal rationally connected fibration for~$\rF_{\dc}(X_\bkk)$,
condition~\textnormal{(C)} implies that~$Z_\bkk$ is birational to~$\Gamma_\bkk$,
and since~$\Gamma_\bkk$ is a smooth curve, it follows that~$Z_\bkk \cong \Gamma_\bkk$.
Thus, the curve~$\Gamma_\bkk$ is actually defined over~$\kk$ and there is a morphism
\begin{equation*}
\xi \colon \rF_{\dc}(X) \larrow \Gamma
\end{equation*}
such that $\xi_\bkk$ coincides with the morphism in condition~(C).
The curve $\Gamma$ is geometrically connected since~$\Gamma_\bkk$ is connected.

Furthermore, $\rF_{\dc}(X)$ is geometrically connected, so $\Jac_{\dc}(X)$ is well-defined.
Consider the Abel--Jacobi map $\AJ_{\dc} \colon \rF_{\dc}(X) \to \Jac_{\dc}(X)$.
Since the fibers of~$\xi$ are geometrically rationally connected, it factors as a composition
\begin{equation*}
\rF_{\dc}(X) \xrightarrow{\ \xi\ } \Gamma \larrow \Jac_{\dc}(X).
\end{equation*}
By the universal property of the Albanese map and~\eqref{eq:alb-pic}, the second map factors as 
\begin{equation*}
\Gamma \xrightarrow{\ \alb_1\ } \Pic^1(\Gamma) \larrow \Jac_{\dc}(X).
\end{equation*}
Over $\bkk$ the second map induces the morphism~$\alpha_{\dc}$ of the corresponding abelian varieties.
So, if~$\alpha_{\dc}$ is an isomorphism, 
the corresponding abelian varieties $\Pic^0(\Gamma)$ and $\Jac(X)$ are isomorphic
and the above map $\Pic^1(\Gamma) \to \Jac_{\dc}(X)$ is an isomorphism of torsors.
\end{proof}

Finally, we rewrite the numerical condition~(N).

\begin{lemma}
\label{lemma:condition-n}
Assume the condition~\textnormal{(N)} of Theorem~\textup{\ref{theorem:obstruction-general}} is satisfied.
Then for any integer $m \in \ZZ$ the product $m\dtor$ is divisible by $\gcd(m\dc - 1, 2g(\Gamma_\bkk) - 2)$.
\end{lemma}
\begin{proof}
Denote $h = \gcd(m\dc - 1, 2g(\Gamma_\bkk) - 2)$. 
Since $h$ divides $m\dc - 1$, it is coprime to $\dc$.
By condition~(N) then $h$ divides $\dtor$, and a fortiori divides $m\dtor$.
\end{proof}

A simple exercise left to the reader is to show that the statement of Lemma~\ref{lemma:condition-n} 
is equivalent to condition~(N) of Theorem~\ref{theorem:obstruction-general}.

Now we are ready to prove Theorem~\ref{theorem:obstruction-general}.

\begin{proof}[Proof of Theorem~\textup{\ref{theorem:obstruction-general}}]
By Proposition~\ref{proposition:jac-dc} we have an isomorphism $\Jac(X) \cong \Pic^0(\Gamma)$ 
of principally polarized abelian varieties.
So if $X$ is rational over~$\kk$ then by Theorem~\ref{theorem:bw} we have an isomorphism of torsors
\begin{equation*}
[\Jac_1(X)] = [\Pic^m(\Gamma)] = m[\Pic^1(\Gamma)]
\end{equation*}
for some integer~$m$. 
By Lemma~\ref{lemma:jac-d} we deduce 
\begin{equation}
\label{eq:jac-d}
[\Jac_d(X)] = md[\Pic^1(\Gamma)]
\end{equation}
for any $d \le \max(\dtor,\dc)$.
Furthermore, by Proposition~\ref{proposition:jac-dc} we have isomorphisms~\eqref{eq:jac-pic-iso}.
Combining the second of them with~\eqref{eq:jac-d} for $d = \dc$, we deduce
\begin{equation*}
(m\dc - 1)[\Pic^1(\Gamma)] = 0.
\end{equation*}
On the other hand, the canonical class of the curve $\Gamma$ is defined over $\kk$, 
hence $\Pic^{2g(\Gamma_\bkk) - 2}(\Gamma)$ has $\kk$-rational points, and so
\begin{equation*}
(2g(\Gamma_\bkk) - 2)[\Pic^1(\Gamma)] = 0.
\end{equation*}
Therefore, the greatest common divisor of $m\dc - 1$ and $2g(\Gamma_\bkk) - 2$ kills the class of the torsor~$\Pic^1(\Gamma)$.
By Lemma~\ref{lemma:condition-n} we conclude that
\begin{equation*}
m\dtor [\Pic^1(\Gamma)] = 0.
\end{equation*}
Using~\eqref{eq:jac-d} for $d = \dtor$ we deduce that $\Jac_{\dtor}(X)$ is a trivial $\Jac(X)$-torsor, 
hence has a $\kk$-rational point.
Finally, condition~(T) implies that the Abel--Jacobi map $\rF_{\dtor}(X) \to \Jac_{\dtor}(X)$ is an isomorphism,
hence $\rF_{\dtor}(X)(\kk) \ne \varnothing$.
\end{proof}

\subsection{Rationality criteria for Fano threefolds}

Below, as promised, we apply Theorem~\ref{theorem:obstruction-general} to Fano threefolds $V_4$, $X_{18}$, and $X_{16}$.
The following lemma verifies the main part of condition~(H).
Definitely, the same is true in a much larger generality, but we restrict to the cases we really need.

\begin{lemma}
\label{lemma:condition-h}
If~$X = V_4$ is a quartic del Pezzo threefold then~$\rF_1(X)$ is connected and~$\rF_2(X)$ admits nice degenerations.
If~$X$ is a prime Fano threefold of genus~$9$ or~$10$ then~$\rF_1(X)$ is connected while~$\rF_2(X)$ and~$\rF_3(X)$ admit nice degenerations.
\end{lemma}

\begin{proof}
We may assume that the base field is algebraically closed.

If~$X = V_4$ then~$\rF_1(X)$ is an abelian surface (Theorem~\ref{theorem:v4-f1}), in particular it is connected.
Moreover, $\rF_1(X,x)$ is a finite scheme of length~4 (\cite[Remark~2.2.7]{KPS}) for each point~$x \in X$; 
choosing in it a subscheme of length~2 we obtain a singular conic on~$X$ (with singularity at~$x$).
This shows that~$\rF_2(X)$ admits nice degenerations.

From now on let~$X = X_{16}$ or~$X = X_{18}$ and let us show that~$\rF_2(X)$ admits nice degenerations, 
i.e., that there are singular or non-reduced conics on~$X$.

If the Hilbert scheme of lines $\rF_1(X)$ is singular, there is a \emph{special line}, i.e., a line~\mbox{$L \subset X$} 
such that $\cN_{L/X} \cong \cO(1) \oplus \cO(-2)$ (\cite[Corollary~2.1.6, proof of Proposition~2.2.8]{KPS}).
Then~$L$ admits a structure of a non-reduced conic (see~\cite[Remark~2.1.2]{KPS}).

So, assume that~$\rF_1(X)$ is smooth.
Let us show that for any connected component of~$\rF_1(X)$ there are pairs of distinct lines parameterized by it which intersect 
(hence form a singular conic).
Indeed, the natural morphism~$p \colon \cC_1(X) \to X$ from the corresponding component of the universal line cannot be a closed embedding, 
because~$\cC_1(X)$ is a smooth ruled surface, while any smooth divisor in~$X$ has nonnegative Kodaira dimension.
Therefore, either there is a pair of meeting lines parameterized by the component of~$\rF_1(X)$, or there is a special line in~$X$.
As we assumed~$\rF_1(X)$ to be smooth, the latter case is impossible, hence the former case takes place.

Now let us show that~$\rF_3(X)$ admits nice degenerations.
Let~$C$ be a general conic on~$X$.
Then~$C$ is irreducible by~\cite[Lemma~2.3.3]{KPS}.
Let~$L$ be a line on~$X$ intersecting~$C$ 
(it exists because lines on~$X$ sweep a surface, which is ample, because~\mbox{$\uprho(X) = 1$}).
The union~$C \cup L$ is a degenerate cubic curve.
It is not contained in a plane (because~$X$ is an intersection of quadrics and contains no planes) 
hence its arithmetic genus is~$0$, 
and so it corresponds to a point of~$\rF_3(X)$.

It remains to prove that~$\rF_1(X)$ is connected.
If~$X = X_{16}$ connectedness of~$\rF_1(X)$ is proved in~\cite[Corollary~5.1.b]{I03} and for~$X = X_{18}$ we argue as follows. 

First, $\rF_2(X) \cong \Jac(X)$, see Corollary~\ref{corollary:x18-conics}.
The intermediate Jacobian~$\Jac(X)$ is an abelian surface, which is general when~$X$ is general
(in fact, \cite[Theorem~4.3.7(ii) and Remark~4.3.9]{IP} show that for each curve~$\Gamma$ of genus~2 there is~$X$ with~$\Jac(X) \cong \Pic^0(\Gamma)$),
in particular, its N\'eron--Severi group has rank~1, hence any effective divisor on~$\rF_2(X)$ is ample, hence connected.

We apply this to the divisor~$\Delta \subset \rF_2(X)$ that parameterizes singular conics 
(it is nonempty, because~$\rF_2(X)$ admits nice degenerations).
Let~$\tilde\Delta \to \Delta$ be the 
double covering that parameterizes pairs~$(L,C)$ of a line and a conic on~$X$ such that~$L \subset C$.
Since~$\Delta$ is connected, $\tilde\Delta$ has at most two connected components.

On the other hand, consider the natural morphism
\begin{equation*}
\tilde\Delta \longrightarrow \rF_1(X) \times \rF_1(X),
\qquad 
(L,C) \longmapsto (L,C \setminus L).
\end{equation*}
If~$\rF_1(X)$ has at least two connected components~$\rF'_1(X)$ and~$\rF''_1(X)$ 
then~$\tilde\Delta$ should have at least four connected components ---
the preimages of 
\begin{equation*}
\rF'_1(X) \times \rF'_1(X),
\qquad 
\rF''_1(X) \times \rF''_1(X),
\qquad 
\rF'_1(X) \times \rF''_1(X),
\qquad\text{and}\qquad 
\rF''_1(X) \times \rF'_1(X).
\qquad 
\end{equation*}
Indeed, the preimage of the first two of the above products is nonempty 
because by the argument above for each component of~$\rF_1(X)$ either there is a pair of meeting lines parameterized by this component
or a special line also parameterized by this component.
Similarly, the preimage of the last two of the above products is nonempty because any line from one component of~$\rF_1(X)$ 
intersects the other since~$\uprho(X) = 1$.
This contradiction finishes the proof for general~$X$.
Finally, for arbitrary~$X$ we conclude by the argument of~\cite[Lemma~2.13]{DK}.
\end{proof}

Now we are ready to deduce Theorem~\ref{theorem:main}\ref{main:v4}--\ref{main:x18} from Theorem~\ref{theorem:obstruction-general}.
For the reader's convenience we collect the parameters $\dtor$, $\dc$, and $2g(\Gamma_\bkk) - 2$ in a table.
\begin{equation*}
\begin{array}{c|cccc}
& \dtor & \dc & 2g(\Gamma_\bkk) - 2 \\
\hline
V_4 	& 1 & 2 & 2\\
X_{18} 	& 2 & 3 & 2\\
X_{16} 	& 3 & 2 & 4
\end{array}
\end{equation*}
This should facilitate verification of the numerical condition~(N) of Theorem~\ref{theorem:obstruction-general}.

\begin{corollary}
\label{corollary:v4-rat}
Let $X = V_4$ be a quartic del Pezzo threefold.
Then~$X$ is $\kk$-rational if and only if~$X(\kk) \ne \varnothing$ and~$X$ contains a line defined over~$\kk$, i.e., $\rF_1(X)(\kk) \ne \varnothing$.
\end{corollary}
\begin{proof}
For the ``if'' direction see Theorem~\ref{theorem:v4-unirational}.
For the ``only if'' direction we apply Theorem~\ref{theorem:obstruction-general}, so we must verify all its conditions.

By Theorem~\ref{theorem:v4-f2} there is a~$\PP^3$-fibration~$\rF_2(X_\bkk) \to \Gamma_\bkk$ 
to a curve of genus~$g(\Gamma_\bkk) = 2$.
The induced Abel--Jacobi map is an isomorphism by Corollary~\ref{corollary:v4-conics},
so condition~(C) holds with~$\dc = 2$.
Furthermore, by Corollary~\ref{corollary:v4-lines} condition~(T) holds for $\dtor = 1$.
In particular, both schemes~$\rF_1(X_\bkk)$ and~$\rF_2(X_\bkk)$ are connected.
Moreover, $\rF_2(X_\bkk)$ admits nice degenerations by Lemma~\ref{lemma:condition-h}, hence condition~(H) holds.
Finally, condition~(N) is evident.
Applying Theorem~\ref{theorem:obstruction-general} we conclude that~$\rF_1(X)(\kk) \ne \varnothing$, 
hence~$X$ contains a line defined over~$\kk$.
Finally, we have~$X(\kk) \ne \varnothing$ by Lemma~\ref{lemma:unirationality-point}.
\end{proof}

\begin{corollary}
\label{corollary:x16-rat}
Let $X = X_{16}$ be a prime Fano threefold of genus~$9$.
Then~$X$ is $\kk$-rational if and only if~$\rF_3(X)(\kk) \ne \varnothing$.
\end{corollary}
\begin{proof}
For the ``if'' direction see Proposition~\ref{proposition:rationality-9}.
For the ``only if'' direction we apply Theorem~\ref{theorem:obstruction-general}, so we must verify all its conditions.

By Theorem~\ref{theorem:x16-f2} there is a $\PP^1$-fibration $\rF_2(X_\bkk) \to \Gamma_\bkk$ to a curve of genus $g(\Gamma_\bkk) = 3$
and by Corollary~\ref{corollary:x16-conics} the induced Abel--Jacobi map is an isomorphism, so condition~(C) holds with~\mbox{$\dc = 2$}.
Furthermore, by Corollary~\ref{corollary:x16-cubics} condition~(T) holds for $\dtor = 3$.
In particular, both schemes~$\rF_2(X_\bkk)$ and $\rF_3(X_\bkk)$ are connected.
Furthermore, $\rF_1(X_\bkk)$ is connected and $\rF_2(X_\bkk)$ and $\rF_3(X_\bkk)$ admit nice degenerations 
by Lemma~\ref{lemma:condition-h}, hence condition~(H) holds.
Finally, condition~(N) is evident.
Applying Theorem~\ref{theorem:obstruction-general} we conclude that $\rF_3(X)(\kk) \ne \varnothing$.
\end{proof}

\begin{corollary}
\label{corollary:x18-rat}
Let $X = X_{18}$ be a prime Fano threefold of genus~$10$.
Then~$X$ is $\kk$-rational if and only if~$X(\kk) \ne \varnothing$ and~$X$ contains a conic defined over~$\kk$, i.e., $\rF_2(X)(\kk) \ne \varnothing$.
\end{corollary}
\begin{proof}
For the ``if'' direction see Proposition~\ref{proposition:rationality-10}.
For the ``only if'' direction we apply Theorem~\ref{theorem:obstruction-general}, so we must verify all its conditions.

By Theorem~\ref{theorem:x18-f3} there is a $\PP^2$-fibration $\rF_3(X_\bkk) \to \Gamma_\bkk$ to a curve of genus $g(\Gamma_\bkk) = 2$
and by Corollary~\ref{corollary:x18-cubics} the induced Abel--Jacobi map is an isomorphism, so condition~(C) holds with~\mbox{$\dc = 3$}.
Furthermore, by Corollary~\ref{corollary:x18-conics} condition~(T) holds for $\dtor = 2$.
In particular, both schemes $\rF_2(X_\bkk)$ and $\rF_3(X_\bkk)$ are connected.
Furthermore, $\rF_1(X_\bkk)$ is connected and $\rF_2(X_\bkk)$ and $\rF_3(X_\bkk)$ admit nice degenerations 
by Lemma~\ref{lemma:condition-h}, hence condition~(H) holds.
Finally, condition~(N) is evident.
Applying Theorem~\ref{theorem:obstruction-general} we conclude that $\rF_2(X)(\kk) \ne \varnothing$, 
i.e., $X$ contains a conic defined over~$\kk$.
Finally, we have~$X(\kk) \ne \varnothing$ by Lemma~\ref{lemma:unirationality-point}.
\end{proof}

\section{Semiorthogonal decompositions and Abel--Jacobi maps}
\label{section:sod-aj}

To prove the structural results for the Hilbert schemes $\rF_d(X)$ of Fano threefolds $X$ used
in Corollaries~\ref{corollary:v4-rat}, \ref{corollary:x16-rat} and~\ref{corollary:x18-rat} 
we will use the technique of derived categories.
In this section we prove several useful results that allow to deduce the required properties of Abel--Jacobi maps
from semiorthogonal decompositions.
In this section we work over an algebraically closed field~$\kk$ (of zero characteristic).

\subsection{Null-algebraic correspondences}

We will need some results about Chow groups.

\begin{definition}
\label{def:nil}
Let $X_1$ and $X_2$ be smooth projective varieties.
A cycle $\alpha \in \CH^p_\QQ(X_1 \times X_2)$ is a~{\sf null-algebraic correspondence}
if for all $i$ it induces zero maps
\begin{equation*}
\CH^i_{\alg,\QQ}(X_1) \xrightarrow{\ \alpha\ } \CH^{i + p - \dim X_1}_{\alg,\QQ}(X_2)
\qquad\text{and}\qquad 
\CH^i_{\alg,\QQ}(X_2) \xrightarrow{\ \alpha\ } \CH^{i + p - \dim X_2}_{\alg,\QQ}(X_1)
\end{equation*}
on algebraically trivial cycles with rational coefficients.
\end{definition}

We denote by 
\begin{equation}
\label{eq:nil}
\CH^\bullet_{\nil,\QQ}(X_1 \times X_2) \subset \CH^\bullet_\QQ(X_1 \times X_2)
\end{equation}
the graded subgroup of null-algebraic correspondences.
The following property is evident.

\begin{lemma}
\label{lemma:ch-nil}
Let $X_i$, $X'_i$ be smooth projective varieties.
\begin{enumerate}
\item 
\label{nil:ideal}
Let $\alpha_{12} \in \CH^\bullet_\QQ(X_1 \times X_2)$ and $\alpha_{23} \in \CH^\bullet_\QQ(X_2 \times X_3)$.
If either $\alpha_{12} \in \CH^\bullet_{\nil,\QQ}(X_1 \times X_2)$ or~$\alpha_{23} \in \CH^\bullet_{\nil,\QQ}(X_2 \times X_3)$ then
\begin{equation*}
p_{13*}(p_{12}^*\alpha_{12} \cdot p_{23}^*\alpha_{23}) \in \CH^\bullet_{\nil,\QQ}(X_1 \times X_3),
\end{equation*}
where $p_{ij} \colon X_1 \times X_2 \times X_3 \to X_i \times X_j$ are the projections.
\item 
\label{nil:pb-pf}
Let $\alpha \in \CH^\bullet_{\nil,\QQ}(X_1 \times X_2)$, $\alpha' \in \CH^\bullet_{\nil,\QQ}(X'_1 \times X'_2)$;
then for any morphisms $f_i \colon X'_i \to X_i$ 
\begin{equation*}
(f_1 \times f_2)^*\alpha \in \CH^\bullet_{\nil,\QQ}(X'_1 \times X'_2)
\qquad\text{and}\qquad 
(f_1 \times f_2)_*\alpha' \in \CH^\bullet_{\nil,\QQ}(X_1 \times X_2).
\end{equation*}
\item 
\label{nil:tp}
The map $\CH^\bullet_\QQ(X_1) \otimes \CH^\bullet_\QQ(X_2) \to \CH^\bullet_\QQ(X_1 \times X_2)$ 
factors through $\CH^\bullet_{\nil,\QQ}(X_1 \times X_2)$.
\end{enumerate}
In other words, the subgroups $\CH^\bullet_{\nil,\QQ} \subset \CH^\bullet_\QQ$ form an ideal with respect to the convolution product 
which is closed under pullbacks and pushforwards with respect to morphisms of factors and contains all decomposable classes. 
\end{lemma}
\begin{proof}
Convolution product of cycles corresponds to composition of maps induced by correspondences.
If one of the maps is zero on algebraically trivial cycles then so is the composition.
This proves~(i).
Furthermore, pullback and pushforward of correspondences with respect to morphisms of factors
can be rewritten as their convolutions with the graphs of the morphisms, so~(i) implies~(ii).
Finally, the tensor product map can be rewritten as the convolution of correspondences to and from $\Spec(\kk)$,
so~(i) also implies~(iii).
\end{proof}

We will also need the following simple observation.

\begin{lemma}
\label{lemma:nil-x-gamma}
Let $\Gamma$ be a smooth projective curve and let $X$ be a smooth projective rationally connected threefold.
Then
\begin{enumerate}
\item 
\label{nil:gamma-gamma}
$\CH^p_{\nil,\QQ}(\Gamma \times \Gamma) = \CH^p_{\QQ}(\Gamma \times \Gamma)$ for $p \ne 1$;
\item 
\label{nil:gamma-x}
$\CH^p_{\nil,\QQ}(\Gamma \times X) = \CH^p_{\QQ}(\Gamma \times X)$ for $p \ne 2$;
\item 
\label{nil:x-x}
$\CH^p_{\nil,\QQ}(X \times X) = \CH^p_{\QQ}(X \times X)$ for $p \ne 3$.
\end{enumerate}
\end{lemma}
\begin{proof}
Clearly, $\CH^i_{\alg,\QQ}(\Gamma) = 0$ for $i \ne 1$.
Furthermore, since $X$ is rationally connected, we have~$\CH^i_{\alg,\QQ}(X) = 0$ for $i \ne 2$.
Now all equalities follow by degree reasons.
\end{proof}

\subsection{From semiorthogonal decompositions to intermediate Jacobians}

Now let $X$ be a smooth projective threefold.
Assume there is a semiorthogonal decomposition
\begin{equation}
\label{eq:sod-general}
\Db(X) = \langle \Phi_\cE(\Db(\Gamma)), E_1, \dots, E_n \rangle,
\end{equation} 
where 
$\Gamma$ is a smooth projective curve, 
$E_1,\dots,E_n$ is an exceptional collection in $\Db(X)$,
and
\begin{equation*}
\Phi_\cE \colon \Db(\Gamma) \larrow \Db(X)
\end{equation*}
is a fully faithful Fourier--Mukai functor with kernel $\cE \in \Db(X \times \Gamma)$.
We consider the Chern classes of $\cE$ as elements of the Chow ring $\CH^\bullet(X \times \Gamma)$,
and the Chern character as an element of~$\CH^\bullet_\QQ(X \times \Gamma)$.
Note that $\rc_2(\cE)$ gives maps $\CH_\alg^i(X) \to \CH_\alg^{i-1}(\Gamma)$ and $\CH_\alg^j(\Gamma) \to \CH_\alg^{j+1}(X)$.

\begin{proposition}
\label{proposition:aj-c2-e}
Let $X$ be a smooth projective rationally connected threefold.
If there is a semiorthogonal decomposition~\eqref{eq:sod-general} then the maps 
\begin{equation}
\xymatrix@1@C=5em{\CH_\alg^2(X)\ \ar@<.3ex>[r]^{\rc_2(\cE)} & 
\ \CH_\alg^1(\Gamma) \ar@<.3ex>[l]^{\rc_2(\cE)} }
\end{equation} 
induce an isomorphism $\Jac(X) \cong \Pic^0(\Gamma)$ of principally polarized abelian varieties.
\end{proposition}

\begin{remark}
Note that the same result for Chow groups with \emph{rational} coefficients 
(and as a consequence an \emph{isogeny} between the intermediate Jacobian of~$X$ and the Jacobian of~$\Gamma$) 
can be deduced from the theory of noncommutative motives \cite[Theorem~1.4]{BT16}.
However, it is crucial for our applications to have a result on integral level.
\end{remark}

\begin{proof}
As we already mentioned, the class $\rc_2(\cE) \in \CH^2(X \times \Gamma)$ 
defines maps between $\CH_\alg^1(\Gamma)$ and~$\CH_\alg^2(X)$ in both directions.
By the universal property of abelian varieties $\Jac(X)$ and $\Pic^0(\Gamma)$ these maps induce morphisms
\begin{equation*}
\xymatrix@1@C=5em{\Jac(X)\ \ar@<.3ex>[r]^{\rc_2(\cE)} & 
\ \Pic^0(\Gamma) \ar@<.3ex>[l]^{\rc_2(\cE)} }
\end{equation*} 
We will check that the compositions of these two are the minus identity maps, i.e., that the maps 
\begin{align*}
\rc_2(\cE) \circ \rc_2(\cE) + \id &\colon \CH^2_\alg(X) \to \CH^2_\alg(X)
&&\text{and} && 
\CH^1_\alg(\Gamma) \to \CH^1_\alg(\Gamma)
\intertext{factor through torsion, i.e., that the induced maps on Chow groups with \emph{rational coefficients}}
\rc_2(\cE) \circ \rc_2(\cE) + \id &\colon \CH^2_{\alg,\QQ}(X) \to \CH^2_{\alg,\QQ}(X)
&&\text{and} &&
\CH^1_{\alg,\QQ}(\Gamma) \to \CH^1_{\alg,\QQ}(\Gamma)
\end{align*}
are zero.

To prove this, note that full faithfulness of the functor $\Phi_\cE$ 
implies that $\Phi_\cE^! \circ \Phi_{\cE} \cong \Phi_{\cO_{\Delta_\Gamma}}$, 
where~$\Phi_\cE^! \cong \Phi_{\cE^\vee(K_\Gamma)[1]}$ is the right adjoint functor of $\Phi_\cE$.
Computing the composition of Fourier--Mukai functors in the left side of this isomorphism, we obtain an isomorphism
\begin{equation*}
p_{13*}(p_{12}^*(\cE^\vee(K_\Gamma)[1]) \otimes p_{23}^*(\cE)) \cong \cO_{\Delta_\Gamma}
\end{equation*}
in $\Db(\Gamma \times \Gamma)$, where the maps~$p_{ij}$ are the projections 
\begin{equation*}
\vcenter{\xymatrix{
& \Gamma \times X \times \Gamma \ar[dl]_{p_{12}} \ar[dr]^{p_{23}} \ar[d]^{p_{13}}
\\
\Gamma \times X &
\Gamma \times \Gamma &
X \times \Gamma
}}
\end{equation*}
and all functors (here and further on) are derived.
This implies equality of Chern characters of both sides.
Using Grothendieck--Riemann--Roch to compute the left side, we obtain
\begin{equation*}
\label{eq:left-sde:2}
- \rch(\cE) \circ (\Delta_*\rtd_X) \circ \rch(\cE^\vee(K_\Gamma)) \
= \rch(\cO_{\Delta_\Gamma}),
\end{equation*}
where the left-hand-side is the convolution of correspondences.
Using Lemma~\ref{lemma:nil-x-gamma} and Lemma~\ref{lemma:ch-nil}\ref{nil:tp}, we see that modulo null-algebraic correspondences we have
\begin{equation*}
\begin{aligned}
&\rch(\cE) \equiv \rch_2(\cE) \equiv -\rc_2(\cE),
&&\rch(\cE^\vee(K_\Gamma)) \equiv \rch(\cE^\vee(K_\Gamma))_2 \equiv -\rc_2(\cE),
\\
&\Delta_*\rtd_X \equiv (\Delta_*\rtd_X)_3 \equiv [\Delta_X],
&&\rch(\Delta_*\cO_\Gamma) \equiv \rch(\Delta_*\cO_\Gamma)_1 \equiv [\Delta_\Gamma].
\end{aligned}
\end{equation*}
Using Lemma~\ref{lemma:ch-nil}\ref{nil:ideal} we deduce the equality
\begin{equation*}
- \rc_2(\cE) \circ \rc_2(\cE) \equiv [\Delta_\Gamma]
\end{equation*}
modulo $\CH^\bullet_{\nil,\QQ}(\Gamma \times \Gamma)$, which implies that $\rc_2(\cE) \circ \rc_2(\cE) + \id$ 
acts trivially on $\CH^1_{\alg,\QQ}(\Gamma)$.

Similarly, the fact that the orthogonal of the image of $\Phi_\cE$ is generated by an exceptional collection implies 
that there is a morphism of functors 
\begin{equation*}
\Phi_\cE \circ \Phi_{\cE}^! \larrow \Phi_{\cO_{\Delta_X}} 
\end{equation*}
whose cone is the projection functor onto the subcategory of $\Db(X)$ generated by the exceptional collection.
Passing to Chern characters of the kernels, we deduce
\begin{equation*}
\label{eq:left-sde:1}
\rch(\cO_{\Delta_X}) + \rch(\cE) \circ (\Delta_*\rtd_\Gamma) \circ \rch(\cE^\vee(K_\Gamma)) + \sum \rch(E_i \boxtimes E'_i) = 0,
\end{equation*}
where $E'_i$ are appropriate exceptional objects on $X$.
Being computed modulo null-algebraic correspondences as above, this equality gives 
\begin{equation*}
[\Delta_X] + \rc_2(\cE) \circ \rc_2(\cE) \equiv 0,
\end{equation*}
which implies that $\rc_2(\cE) \circ \rc_2(\cE) + \id$ acts trivially on $\CH^1_{\alg,\QQ}(X)$.

The above arguments also show that the compositions of the maps 
\begin{equation*}
\xymatrix@1@C=5em{
H^3_{\textrm{\'et}}(X,\QQ_\ell/\ZZ_\ell(2))\ \ar@<.3ex>[r]^{\rc_2(\cE)} & 
\ H^1_{\textrm{\'et}}(\Gamma,\QQ_\ell/\ZZ_\ell(1)) \ar@<.3ex>[l]^{\rc_2(\cE)} 
}
\end{equation*}
are equal to $-1$ for any prime $\ell$. 
On the other hand, the morphisms~$\rc_2(\cE)$ above are adjoint with respect to intersection product, i.e.,
\begin{equation*}
\rc_2(\cE)(\alpha) \cdot \beta = \alpha \cdot \rc_2(\cE)(\beta)
\end{equation*}
for any~$\alpha \in H^3_{\textrm{\'et}}(X,\QQ_\ell/\ZZ_\ell(2))$ and~$\beta \in H^1_{\textrm{\'et}}(\Gamma,\QQ_\ell/\ZZ_\ell)$.
Combining this with the previous observation, we conclude that 
\begin{equation*}
\rc_2(\cE)(\alpha_1) \cdot \rc_2(\cE)(\alpha_2) = 
\alpha_1 \cdot \rc_2(\cE)(\rc_2(\cE)(\alpha_2)) = 
- \alpha_1 \cdot \alpha_2,
\end{equation*}
hence these maps are compatible up to sign with intersection products,
and therefore the isomorphism~$\Jac(X) \cong \Pic^0(\Gamma)$ is compatible with the natural principal polarizations.
\end{proof}

\subsection{Abel--Jacobi maps}

Consider again a semiorthogonal decomposition~\eqref{eq:sod-general}.
In this section we abbreviate the functor $\Phi_\cE$ to just~$\Phi$.
By base change (see~\cite{K11}) for any quasiprojective scheme $S$ it induces a semiorthogonal decomposition
\begin{equation}
\Db(X \times S) = \langle \Phi_S(\Db(\Gamma \times S)), E_1 \boxtimes \Db(S), \dots, E_n \boxtimes \Db(S) \rangle,
\end{equation} 
where $\Phi_S$ is the Fourier--Mukai functor with kernel $\cE \boxtimes \cO_{\Delta_S} \in \Db((X \times S) \times (\Gamma \times S))$.
This allows to apply the functor $\Phi_S$ and its adjoints to $S$-families of objects in~$X$.
Below we do this for ideals of curves, but the same procedure can be analogously applied to other families of objects.

Let $\cC_d(X) \subset X \times \rF_d(X)$ denote the universal curve.
Let $\Phi^*$, $\Phi^*_S$, $\Phi^!$, $\Phi^!_S$ be the left and right adjoint functors of~$\Phi$ and~$\Phi_S$ respectively.

\begin{lemma}
\label{lemma:moduli-map}
Let $X$ be a smooth projective variety with a semiorthogonal decomposition~\eqref{eq:sod-general}.
Fix some $d \in \ZZ_{>0}$.
Assume that there are integers $r,e,s \in \ZZ$ with $r \ge 0$ and $\gcd(r,e) = 1$ such that
for any curve $C \subset X$ from the Hilbert scheme $\rF_d(X)$ there is an isomorphism
\begin{equation*}
\Phi^*_\cE(I_C) \cong F[s],
\end{equation*}
for some stable sheaf $F$ of rank~$r$ and degree~$e$ on $\Gamma$.

Then there is a morphism $f \colon \rF_d(X) \to \rM_\Gamma(r,e)$ 
to the fine moduli space of stable sheaves of rank~$r$ and degree~$e$ on~$\Gamma$
such that there is a unique isomorphism
\begin{equation}
\label{eq:phi-i-f}
\Phi^*_{\rF_d(X)}(I_{\cC_d(X)}) \cong (\id_\Gamma \times f)^*(\cF) \otimes p_2^*\cL[s]
\end{equation} 
in $\Db(\Gamma \times \rF_d(X))$,
where $\cF$ is the universal sheaf on $\Gamma \times \rM_\Gamma(r,e)$ and $\cL$ is a line bundle on~$\rF_d(X)$.
The same result holds for the left adjoint functor $\Phi^*$ replaced by the right adjoint $\Phi^!$.
\end{lemma}

\begin{proof}
The proof is analogous to the proof of~\cite[Theorem~4.5]{K19}, so we skip it here.
\end{proof}

If the assumptions of Lemma~\ref{lemma:moduli-map} hold true, we will say that the morphism~$f$ is
{\sf induced by the functor $\Phi^*_\cE$} ({\sf functor $\Phi^!_\cE$}, respectively).
The next result can be used to prove that the Abel--Jacobi map associated with a universal curve is an isomorphism.
Recall the map~$\alpha_d$ defined in~\eqref{eq:alpha-d}.

\begin{proposition}
\label{proposition:aj-cd}
Let $X$ be a smooth projective rationally connected threefold with a semiorthogonal decomposition~\eqref{eq:sod-general}.
Assume that either of the functors $\Phi_\cE^*$ or $\Phi^!_\cE$ induces a morphism 
\begin{equation*}
f \colon \rF_d(X) \larrow \rM_\Gamma(r,e).
\end{equation*}
Then there is a commutative up to sign and torsion diagram
\begin{equation*}
\xymatrix@C=6em{
\CH_{\mathrm{alg}}^2(X) \ar[r]^-{{[\cC_d(X)]}} \ar[d]_{{\rc_2(\cE)}} &
\CH_{\mathrm{alg}}^1(\rF_d(X))
\\
\CH_{\mathrm{alg}}^1(\Gamma) \ar[r]^-{{\rc_1(\cF)}} &
\CH_{\mathrm{alg}}^1(\rM_\Gamma(r,e)) \ar[u]_{f^*} 
}
\end{equation*}
where $\cF$ is the universal sheaf on $\Gamma \times \rM_\Gamma(r,e)$ 
and $\rc_1(\cF) \in \CH^1(\Gamma \times \rM_\Gamma(r,e))$ is considered as a correspondence.
In other words, the morphism
\begin{equation*}
f^* \circ \rc_1(\cF) \circ \rc_2(\cE) \pm [\cC_d(X)] \colon \CH^2_\alg(X) \to \CH^1_\alg(\rF_d(X))
\end{equation*}
factors through torsion for an appropriate choice of sign.

In particular, if the maps ${\rc_1(\cF)}$ and $f^*$ in the diagram are isomorphisms then the map 
\begin{equation*}
\alpha_d \colon \Alb(\rF_d(X)) \to \Jac(X) 
\end{equation*}
defined in Section~\textup{\ref{subsection:obstruction}} is an isomorphism.
\end{proposition}

\begin{proof}
We discuss the case of the morphism $f$ induced by the functor~$\Phi^*_\cE$; the other case being analogous.
Since we aim at proving commutativity up to sign and torsion, it is enough to check commutativity up to sign with \emph{rational coefficients}.

Since the morphism $f$ is induced by the functor~$\Phi^*_\cE$, we have an isomorphism~\eqref{eq:phi-i-f} 
for some line bundle~$\cL$ on $\rF_d(X)$ and some integer~$s$.
It follows that there is a morphism
\begin{equation*}
I_{\cC_d(X)} \larrow \Phi_{\rF_d(X)}((\id_\Gamma \times f)^*(\cF) \otimes p_2^*\cL)[s]
\end{equation*}
whose cone is contained in the subcategory $\langle E_1 \boxtimes \Db(\rF_d(X)), \dots, E_n \boxtimes \Db(\rF_d(X)) \rangle$.
The Chern characters of decomposable objects are in $\CH^\bullet_{\nil,\QQ}(X \times \rF_d(X))$ by Lemma~\ref{lemma:ch-nil}\ref{nil:tp}, 
hence we obtain an equality modulo $\CH^\bullet_{\nil,\QQ}(X \times \rF_d(X))$:
\begin{equation}
\label{eq:rch-equality}
\rch(I_{\cC_d(X)}) \equiv 
(-1)^s \rch\left(\Phi_{\rF_d(X)}((\id_\Gamma \times f)^*(\cF) \otimes p_2^*\cL)\right). 
\end{equation}
By Grothendieck--Riemann--Roch the right-hand side 
can be written as convolution
\begin{equation*}
\rch(\cE) \circ (\Delta_*\rtd_\Gamma) \circ \rch((\id_\Gamma \times f)^*(\cF) \otimes p_2^*\cL).
\end{equation*}
Modulo null-algebraic correspondences this is equal to
\begin{equation*}
-\rc_2(\cE) \circ [\Delta_\Gamma] \circ (\id_\Gamma \times f)^*\rch(\cF).
\end{equation*}
So, taking the component of~\eqref{eq:rch-equality} in degree~2 modulo $\CH^2_{\nil,\QQ}(X \times \rF_d(X))$ we obtain
\begin{equation*}
-[\cC_d(X)] \equiv 
\rch_2(I_{\cC_d(X)}) \equiv
-(-1)^s \rc_2(\cE) \circ (\id_\Gamma \times f)^*\rch_1(\cF) \equiv
-(-1)^s \rc_2(\cE) \circ (\id_\Gamma \times f)^*\rc_1(\cF).
\end{equation*}
It follows that the morphism 
$f^* \circ \rc_1(\cF) \circ \rc_2(\cE) - (-1)^s [\cC_d(X)] \colon \CH^2_{\alg,\QQ}(X) \to \CH^2_{\alg,\QQ}(X)$
is zero, which implies the first claim of the proposition.

By the universal property of the intermediate Jacobian and the Picard scheme, we conclude that the composition of maps 
\begin{equation*}
\Jac(X) \xrightarrow{\ \rc_2(\cE)\ } 
\Pic^0(\Gamma) \xrightarrow{\ \rc_1(\cF)\ } 
\Pic^0(\rM_\Gamma(r,e)) \xrightarrow{\ f^*\ }
\Pic^0(\rF_d(X))
\end{equation*}
coincides up to sign with the map induced by $[\cC_d(X)]$.
Moreover, the first map is an isomorphism by Proposition~\ref{proposition:aj-c2-e}.
So if the second and the third maps are isomorphisms, then so is the composition.
Therefore, the transposed map $\Alb(\rF_d(X)) \to \Jac(X)$, again induced by $[\cC_d(X)]$, is also an isomorphism.
\end{proof}

We extract the following two corollaries.

\begin{corollary}
\label{corollary:fd-pic}
In the situation of Proposition~\textup{\ref{proposition:aj-cd}} assume that $r = 1$, 
and that the induced map $f \colon \rF_d(X) \to \rM_\Gamma(r,e) \cong \Pic^e(\Gamma) \cong \Pic^0(\Gamma)$ is an isomorphism.
Then the Abel--Jacobi map~$\AJ \colon \rF_d(X) \to \Jac(X)$ is an isomorphism.
\end{corollary}
\begin{proof}
The map $f$ is an isomorphism by assumption and the map $\CH^1_\alg(\Gamma) \to \CH^1_\alg(\Pic^0(\Gamma))$
given by the Poincar\'e line bundle $\cF$ on $\Gamma \times \Pic^0(\Gamma)$ is an isomorphism, 
hence Proposition~\ref{proposition:aj-cd} implies that~$\alpha_d \colon \Alb(\rF_d(X)) \to \Jac(X)$ is an isomorphism.
Since $\rF_d(X)$ is isomorphic (via~$f$) to an abelian variety, the Albanese map $\alb \colon \rF_d(X) \to \Alb(\rF_d(X))$ is an isomorphism,
hence the composition $\AJ = \alpha_d \circ \alb$ is an isomorphism as well.
\end{proof}

\begin{corollary}
\label{corollary:fd-gamma}
In the situation of Proposition~\textup{\ref{proposition:aj-cd}} assume that $r = 0$, $e = 1$, so that~\mbox{$M_\Gamma(r,e) \cong \Gamma$},
and that the induced map~$f \colon \rF_d(X) \to \Gamma$ has rationally connected fibers.
Then the induced map
\begin{equation*}
\Pic^0(\Gamma) \cong \Alb(\rF_d(X)) \xrightarrow{\ \alpha_d\ } \Jac(X)
\end{equation*}
is an isomorphism of principally polarized abelian varieties.
\end{corollary}
\begin{proof}
Since $f$ has rationally connected fibers, the induced map $f^* \colon \CH^1_\alg(\Gamma) \to \CH^1_\alg(\rF_d(X))$ is an isomorphism.
Furthermore, the universal sheaf $\cF$ on $\Gamma \times \Gamma$ is the structure sheaf of the diagonal,
hence the map given by $\rc_1(\cF) = [\Delta_\Gamma]$ is also an isomorphism.
Using Proposition~\ref{proposition:aj-cd} we conclude that the map $\alpha_d$ is an isomorphism.
It is compatible with the principal polarizations by Proposition~\ref{proposition:aj-c2-e}.
\end{proof}

\section{Semiorthogonal decompositions for Fano threefolds}
\label{section:sods}

Throughout this section we work over an algebraically closed field~$\kk$ of characteristic zero.
Semiorthogonal decompositions for Fano threefolds $X$ of types $V_4$, $X_{16}$, and~$X_{18}$ were constructed in~\cite{K06}.
They take the form
\begin{equation*}
\Db(X) = \langle \Phi(\Db(\Gamma)), \cO, E \rangle
\end{equation*}
where $\Gamma$ is a curve of genus~2, 3, and~2, respectively, $E$ is an exceptional vector bundle, and
\begin{equation*}
\Phi \colon \Db(\Gamma) \larrow \Db(X)
\end{equation*}
is the embedding functor (we will provide below the details in each case).
We will show that its adjoint functors~$\Phi^*$ and~$\Phi^!$ induce morphisms from appropriate Hilbert schemes of~$X$
to the curve~$\Gamma$ or its Jacobian and then apply Corollaries~\ref{corollary:fd-pic} and~\ref{corollary:fd-gamma}
to deduce the results about Abel--Jacobi maps 
that were required in Section~\ref{section:obstructions} 
for the proofs of Lemma~\ref{lemma:condition-h} and Corollaries~\ref{corollary:v4-rat}, \ref{corollary:x16-rat} and~\ref{corollary:x18-rat}.

We start with the following general result.

\begin{lemma}
\label{lemma:ic-ox}
Assume that $X \subset \PP^n$ is an intersection of quadrics and contains no planes.
Then for any $d \le 3$ and any curve $C$ corresponding to a point of $\rF_d(X)$ we have $H^\bullet(C,\cO_C) = \kk$ and
the ideal $I_C$ is contained in the orthogonal $\cO^\perp$.
\end{lemma}
\begin{proof}
The Hilbert polynomial of $C$ is equal to $dt + 1$, so to prove $H^\bullet(C,\cO_C) = \kk$ 
it is enough to show that the first cohomology vanishes.
For lines and conics this is easy (see, e.g., \cite[Lemma~2.1.1]{KPS})
and for cubic curves this follows from classification of Lemma~\ref{lemma:f3}.
Now the short exact sequence
\begin{equation*}
0 \longrightarrow I_C \longrightarrow \cO \longrightarrow \cO_C \longrightarrow 0
\end{equation*}
implies that $I_C \in \cO^\perp$.
\end{proof}

Note that the result of Lemma~\ref{lemma:ic-ox} is no longer true for $d > 3$.
For example, the union of an elliptic quartic curve on~$X$ with a point is an element in $\rF_4(X)$,
whose ideal is not contained in~$\cO^\perp$.

\subsection{Fano threefold of genus~9}
\label{subsection:genus-9}

Let $X$ be a smooth prime Fano threefold of genus~9 over an algebraically closed field of characteristic~0.
Mukai proved that $X$ admits a natural embedding into $\LGr(3,V_6)$, 
the Lagrangian Grassmannian of a 6-dimensional symplectic vector space $V_6$.
We denote by $\cU \subset V_6 \otimes \cO$ the restriction to $X$ of the rank~3 tautological subbundle on $\LGr(3,V_6)$.
By~\cite[\S6.3]{K06} we have a semiorthogonal decomposition
\begin{equation}
\label{eq:dbx16}
\Db(X) = \langle \Phi(\Db(\Gamma)), \cO, \cU^\vee \rangle,
\end{equation}
where $\Gamma$ is a curve of genus~3.
We will remind a description of the embedding functor 
\begin{equation*}
\Phi \colon \Db(\Gamma) \longrightarrow \Db(X) 
\end{equation*}
below (see~\eqref{eq:ce}).
Meanwhile, recall that $\cO$ and $\cU$ are exceptional bundles and that
\begin{equation}
\label{eq:coh-x16}
H^\bullet(X,\cU) = 0,
\qquad 
H^\bullet(X,\cO) = \kk,
\qquad 
H^\bullet(X,\cU^\vee) = V_6.
\end{equation} 

The following result is well-known.

\begin{theorem}[{\cite[Proposition~3.10]{BF}, cf.~\cite[Theorem~4.6]{I03}}]
\label{theorem:x16-f2}
The left adjoint functor 
\begin{equation*}
\Phi^* \colon \Db(X) \longrightarrow \Db(\Gamma)
\end{equation*}
induces a morphism $\rF_2(X) \to \Gamma$ which is a $\PP^1$-bundle.
\end{theorem}

Applying Corollary~\ref{corollary:fd-gamma}, we deduce 

\begin{corollary}
\label{corollary:x16-conics}
The Abel--Jacobi map $\AJ \colon \Pic^0(\Gamma) \cong \Alb(\rF_2(X)) \to \Jac(X)$ is an isomorphism
of principally polarized abelian varieties.
\end{corollary}

Below we prove an analogue of this result for rational cubic curves on $X$.

\begin{theorem}
\label{theorem:x16-f3}
The left adjoint functor $\Phi^* \colon \Db(X) \to \Db(\Gamma)$ induces a noncanonical isomorphism 
\begin{equation*}
\rF_3(X) \cong \Pic^0(\Gamma).
\end{equation*}
\end{theorem}

The proof will take the rest of the subsection.
We will show that the ideals of twisted cubic curves in $X$ belong to the component $\Db(\Gamma)$ 
of~\eqref{eq:dbx16} and correspond to line bundles on $\Gamma$.

Note that the restriction to~$X$ of the tautological exact sequence on~$\LGr(3,V_6)$
\begin{equation}
\label{eq:tautological}
0 \larrow \cU \larrow V_6 \otimes \cO \larrow \cU^\vee \larrow 0
\end{equation} 
shows that $\cU \in \langle \cO, \cU^\vee \rangle$ and $\cU^\vee \in \langle \cU, \cO \rangle$.

\begin{lemma}
\label{lemma:ic-ax}
If $C \subset X$ is a rational cubic curve then $I_C \in \langle \cO, \cU^\vee \rangle^\perp = \langle \cU, \cO \rangle^\perp$.
\end{lemma}
\begin{proof}
Let us check that $H^i(C,\cU\vert_C) = 0$ for all $i$.
Indeed, $C$ is a curve, hence the cohomology vanishes for $i \not\in \{0,1\}$.
Furthermore, the rank of $\cU$ is 3 and the first Chern class is~$-1$, 
hence up to codimension~2 classes $\cU$ is equivalent to $\cO \oplus \cO \oplus \cO(-1)$.
Therefore, 
\begin{equation*}
\chi(\cU\vert_C) = \chi(\cO_C) \oplus \chi(\cO_C) \oplus \chi(\cO(-1)\vert_C) = 2h_C(0) + h_C(-1) = 0,
\end{equation*}
where $h_C(t) = 3t + 1$ is the Hilbert polynomial of~$C$.
So, it only remains to prove $H^0(C,\cU\vert_C) = 0$.

Assume, to the contrary, that~$\cU\vert_C$ has a global section.
The short exact sequence
\begin{equation*}
0 \larrow \cU\vert_C \larrow V_6 \otimes \cO_C \larrow \cU^\vee\vert_C \larrow 0
\end{equation*}
then shows that there is a vector $v \in V_6$ such that the zero locus of the corresponding global section of $\cU^\vee$ vanishes on $C$.
But~\cite[Lemma~3.6]{BF} shows that this is impossible.

This proves that $\cO_C$ is in the orthogonal to~$\cU^\vee$.
By~\eqref{eq:coh-x16} the same is true for $I_C$.
Finally, orthogonality of~$I_C$ to $\cO$ is proved in Lemma~\ref{lemma:ic-ox}.
\end{proof}

\begin{corollary}
\label{corollary:huc}
If $C \subset X$ is a rational cubic curve then the restriction morphism
\begin{equation*}
H^0(X,\cU^\vee) \larrow H^0(C,\cU^\vee\vert_C)
\end{equation*}
is an isomorphism and $H^{>0}(C,\cU^\vee\vert_C) = 0$.
\end{corollary}
\begin{proof}
Tensoring the standard exact sequence
\begin{equation*}
0 \larrow I_C \larrow \cO \larrow \cO_C \larrow 0
\end{equation*} 
with $\cU^\vee$, taking into account that
\begin{equation*}
H^i(X,\cU^\vee \otimes I_C) \cong \Ext^i(\cU,I_C) = 0,
\end{equation*}
by Lemma~\ref{lemma:ic-ax}, and using~\eqref{eq:coh-x16},
we deduce the claim.
\end{proof}

For the description of the functor~$\Phi \colon \Db(\Gamma) \to \Db(X)$ from~\eqref{eq:dbx16} we refer to~\cite[\S6.3]{KPS} and~\cite[\S3.2.1]{BF}
which show that~$\Phi$
is the Fourier--Mukai functor
with kernel a vector bundle~$\cE$ of rank~2 on~$X \times \Gamma$ such that there is an exact sequence
\begin{equation}
\label{eq:ce}
0 \larrow \cE \larrow \cO \boxtimes \cF_6 \larrow \cU^\vee \boxtimes \cF_2 \larrow \cE(H + K_\Gamma) \larrow 0,
\end{equation}
where $\cF_2$ and $\cF_6$ are vector bundles of ranks~2 and~6 on $\Gamma$.

\begin{proposition}
\label{proposition:phi-s-ic}
If $C \subset X$ is a rational cubic curve then $\Phi^!(I_C)[-1]$ is a line bundle on~$\Gamma$.
\end{proposition}
\begin{proof}
The functor $\Phi^!(-)[-1]$ is the Fourier--Mukai functor with kernel
$\cE^\vee \otimes \omega_\Gamma$.
For $y \in \Gamma$ let $\cE_y$ be the rank~2 vector bundle on $X$ corresponding to the point $y$.
Note that 
\begin{equation}
\label{eq:cey-x16}
\cE_y \cong \Phi(\cO_y) \in \langle \cO, \cU^\vee \rangle^\perp
\end{equation}
and $\rc_1(\cE_y) = -H$.
Restricting~\eqref{eq:ce} to the fiber over the point $y$ 
and taking into account isomorphism $\cE_y^\vee \cong \cE_y(H)$ we obtain an exact sequence
\begin{equation}
\label{eq:cey}
0 \larrow \cE_y \larrow \cO^{\oplus 6} \larrow \cU^\vee \oplus \cU^\vee \larrow \cE_y^\vee \larrow 0.
\end{equation}
Using~\eqref{eq:coh-x16} and~\eqref{eq:cey-x16} we obtain
\begin{equation}
\label{eq:hpos-cey}
H^{>0}(X,\cE_y^\vee) = 0.
\end{equation}
Furthermore, it follows from~\eqref{eq:cey} and~\eqref{eq:tautological} that~$\cE_y^\vee$ is globally generated, hence
\begin{equation}
\label{eq:hpos-cey-c}
H^{>0}(C,\cE_y^\vee\vert_C) = 0.
\end{equation}
Let us also show that the restriction homomorphism $H^0(X,\cE_y^\vee) \to H^0(C,\cE_y^\vee\vert_C)$ is surjective.
Indeed, we have a commutative diagram
\begin{equation*}
\xymatrix{
H^0(X,\cU^\vee \oplus \cU^\vee) \ar[r] \ar[d] &
H^0(C,(\cU^\vee \oplus \cU^\vee)\vert_C) \ar[d]
\\
H^0(X,\cE_y^\vee) \ar[r] &
H^0(C,\cE_y^\vee\vert_C),
}
\end{equation*}
where the horizontal arrows are induced by restriction and the vertical arrows are induced by the morphism from~\eqref{eq:cey}.
The top horizontal arrow is an isomorphism by Corollary~\ref{corollary:huc}.
Furthermore, the restriction of~\eqref{eq:cey} to $C$ and the vanishing of $H^{>0}(C,\cO_C)$ (Lemma~\ref{lemma:ic-ox})
imply that the right vertical arrow is surjective.
Therefore, the bottom horizontal arrow is surjective as well.

Combining this surjectivity with~\eqref{eq:hpos-cey} and~\eqref{eq:hpos-cey-c} we deduce that 
\begin{equation*}
H^{>0}(X,\cE_y^\vee \otimes I_C) = 0.
\end{equation*}
On the other hand, computing the Euler characteristic $\chi(\cE_y^\vee) = 6$ 
(this follows from~\eqref{eq:cey}, \eqref{eq:cey-x16}, and~\eqref{eq:coh-x16})
and $\chi(\cE_y^\vee\vert_C) = 5$ (this follows from $\rank(\cE_y) = 2$ and $\rc_1(\cE_y) = -H$), we obtain 
\begin{equation*}
\dim H^{0}(X,\cE_y^\vee \otimes I_C) = \chi(\cE_y^\vee \otimes I_C) = 1.
\end{equation*}
These two facts imply (see~\cite[Lemma~4.4]{K19} for an argument) that the Fourier--Mukai functor with kernel~$\cE^\vee$ takes~$I_C$ to a line bundle,
hence the same is true for the functor~$\Phi^!(-)[-1]$.
\end{proof}

Now we can complete the proof of the theorem.

\begin{proof}[Proof of Theorem~\textup{\ref{theorem:x16-f3}}]
By Proposition~\ref{proposition:phi-s-ic} and Lemma~\ref{lemma:moduli-map} the functor~$\Phi^!$ induces a morphism
\begin{equation*}
\rF_3(X) \larrow \Pic(\Gamma).
\end{equation*}
Arguing as in~\cite[Proposition~B.5.5]{KPS} we check that it is \'etale, proper, and has inverse (given by the functor~$\Phi(-)[1])$.
Therefore, it is an isomorphism onto a connected component of~$\Pic(\Gamma)$.
In particular, $\rF_3(X) \cong \Pic^0(\Gamma)$.
\end{proof}

The particular component of $\Pic(\Gamma)$ onto which $\rF_3(X)$ is mapped by $\Phi^!$ depends on the choice of the bundle~$\cE$,
which is defined by~\eqref{eq:cey} only up to twist by a line bundle on~$\Gamma$.

Applying Corollary~\ref{corollary:fd-pic} we deduce

\begin{corollary}
\label{corollary:x16-cubics}
The Abel--Jacobi map $\AJ \colon \rF_3(X) \cong \Alb(\rF_3(X)) \to \Jac(X)$ is an isomorphism.
\end{corollary}

\subsection{Fano threefold of genus~10}
\label{subsection:genus-10}

Let $X$ be a smooth prime Fano threefold of genus~10 over an algebraically closed field of characteristic~0.
Mukai proved that $X$ admits a natural embedding into $\GTGr(2,V_7)$, 
the Grassmannian of the simple algebraic group of type $\rG_2$ that parameterizes 2-dimensional 
subspaces in a 7-dimensional vector space $V_7$ annihilated by a $\rG_2$-invariant 3-form~\mbox{$\lambda \in \wedge^3 V_7^\vee$}.
We denote by $\cU \subset V_7 \otimes \cO$ the restriction to $X$ of the rank~2 tautological subbundle on $\Gr(2,V_7)$.
By~\cite[\S6.4]{K06} we have a semiorthogonal decomposition
\begin{equation}
\label{eq:dbx18}
\Db(X) = \langle \Phi(\Db(\Gamma)), \cO, \cU^\vee \rangle,
\end{equation}
where $\Gamma$ is a curve of genus~2.
We will remind a description of the embedding functor 
\begin{equation*}
\Phi \colon \Db(\Gamma) \to \Db(X) 
\end{equation*}
below (see~\eqref{eq:cey-x18}).
Meanwhile, recall that $\cO$ and $\cU$ are exceptional bundles and that
\begin{equation}
\label{eq:coh-x18}
H^\bullet(X,\cU) = 0,
\qquad 
H^\bullet(X,\cO) = \kk,
\qquad 
H^\bullet(X,\cU^\vee) = V_7^\vee.
\end{equation} 
Note also that $\cU^\vee \cong \cU(1)$.
Recall the following result.

\begin{theorem}[{\cite[Proposition~B.5.5]{KPS}}]
\label{theorem:x18-f2}
The left adjoint functor $\Phi^* \colon \Db(X) \to \Db(\Gamma)$ induces a noncanonical isomorphism 
\begin{equation*}
\rF_2(X) \cong \Pic^0(\Gamma).
\end{equation*}
\end{theorem}

Applying Corollary~\ref{corollary:fd-pic}, we deduce 

\begin{corollary}
\label{corollary:x18-conics}
The Abel--Jacobi map $\AJ \colon 
\rF_2(X) \cong \Alb(\rF_2(X)) \to \Jac(X)$ is an isomorphism.
\end{corollary}

Below we prove an analogue of this result for rational cubics on $X$.

\begin{theorem}
\label{theorem:x18-f3}
The Hilbert scheme $\rF_3(X)$ of rational cubic curves on $X$ is a $\PP^2$-bundle over~$\Gamma$.
\end{theorem}

The proof will take the rest of the subsection.
We will show that appropriate extensions of ideals of rational cubic curves in $X$ belong to the component $\Db(\Gamma)$ 
of~\eqref{eq:dbx18} and correspond to structure sheaves of points in $\Gamma$.

\begin{lemma}
\label{lemma:fc-ax}
If $C \subset X$ is a rational cubic curve then 
\begin{equation*}
\Ext^1(I_C,\cU) \cong \kk.
\end{equation*}
Let $F_C$ be the corresponding nontrivial extension, so that we have an exact sequence
\begin{equation}
\label{eq:fc-sequence}
0 \larrow \cU \larrow F_C \larrow \cO \larrow \cO_C \larrow 0.
\end{equation}
Then $F_C$ is a stable vector bundle of rank~$3$ on~$X$ with $\rc_1(F_C) = -H$.
Moreover, $F_C \in \langle \cO, \cU^\vee \rangle^\perp$.
\end{lemma}
\begin{proof}
The argument of~\cite[Lemma~B.3.3]{KPS} shows that $H^0(C,\cU\vert_C) = 0$.
Since $C$ is a curve, the only non-vanishing cohomology of $\cU\vert_C$ is $H^1(C,\cU\vert_C)$.
Furthermore, the rank of $\cU$ is~2 and the first Chern class is~$-1$, 
hence up to codimension~2 classes $\cU$ is equivalent to $\cO \oplus \cO(-1)$.
Therefore, 
\begin{equation*}
\chi(\cU\vert_C) = \chi(\cO_C) \oplus \chi(\cO(-1)\vert_C) = h_C(0) + h_C(-1) = -1
\end{equation*}
where $h_C(t) = 3t + 1$ is the Hilbert polynomial of~$C$.
It follows that 
\begin{equation*}
\Ext^1(I_C,\cU) \cong \Ext^2(\cU^\vee,I_C)^\vee \cong H^2(X,\cU \otimes I_C)^\vee \cong 
H^1(X, \cU \otimes \cO_C)^\vee \cong H^1(C,\cU\vert_C)^\vee = \kk
\end{equation*}
(in the first isomorphism we used Serre duality, and in the third we used~\eqref{eq:coh-x18}).

This also proves that $\Ext^i(I_C,\cU) = 0$ for $i \ne 1$.
Therefore, if $F_C$ is defined as the corresponding extension, then $\Ext^\bullet(F_C,\cU) = 0$, 
hence by Serre duality $\Ext^\bullet(\cU^\vee,F_C) = 0$, hence $F_C \in {\cU^\vee}^\perp$.
Furthermore, as $\cU \in \cO^\perp$ by~\eqref{eq:coh-x18} and $I_C \in \cO^\perp$ by Lemma~\ref{lemma:ic-ox},
we deduce $F_C \in \cO^\perp$.

The rank and first Chern class computations for $F_C$ are straightforward, 
and local freeness can be proved by the same argument that proves local freeness of Serre's construction.
So, it remains to prove stability of $F_C$.
For this we use Hoppe's criterion.
According to it it is enough to check the vanishing of $H^0(X,F_C)$ and of $H^0(F_C^\vee(-1)) \cong H^3(X,F_C)^\vee$.
But both follow from the containment~$F_C \in \cO^\perp$ proved above.
\end{proof}

\begin{proposition}
\label{proposition:phi-s-fc}
If $C \subset X$ is a rational cubic curve then $\Phi^!(F_C) \cong \cO_y$, where $y \in \Gamma$.
\end{proposition}
\begin{proof}
Since~$\Phi$ is fully faithful, it is enough to show that $F_C \cong \Phi(\cO_y)$.
Recall from~\cite[\S6.4]{K06} and~\cite[\S B.5]{KPS} that~$\Phi$ is the Fourier--Mukai functor 
whose kernel is a rank~3 vector bundle~$\cE$ on~$X \times \Gamma$
such that for each point~\mbox{$y \in \Gamma$} the corresponding bundle~$\cE_y$ on~$X$ fits into an exact sequence~(see~\cite[(B.5.2)]{KPS})
\begin{equation}
\label{eq:cey-x18}
0 \larrow \cE_y \larrow \cO^{\oplus 6} \larrow \cU^\vee \oplus \cU^\vee \oplus \cU^\vee \larrow \cE_y(1) \larrow 0.
\end{equation}
Applying the functor $\Ext^\bullet(-,F_C)$ and taking into account 
that $F_C \in \langle \cO, \cU^\vee \rangle^\perp$ by Lemma~\ref{lemma:fc-ax},
we conclude that $\Ext^i(\cE_y,F_C) = 0$ for $i \ge 2$.
Using Riemann--Roch it is easy to check that 
\begin{equation*}
\chi(\cE_y,F_C) = 0,
\end{equation*}
hence either $\Ext^\bullet(\cE_y,F_C) = 0$ or $\Hom(\cE_y,F_C) \ne 0$.

If the first holds for all $y \in \Gamma$ then $\Ext^\bullet(\cO_y,\Phi^!(F_C)) = 0$ for all $y \in \Gamma$, 
hence $\Phi^!(F_C) = 0$, hence~$F_C = 0$, which is absurd.
Therefore, there is $y \in \Gamma$ and a nontrivial morphism $\cE_y \to F_C$.
Since both $\cE_y$ and $F_C$ are stable of the same slope, this is an isomorphism.
\end{proof}

This proposition already proves that the functor $\Phi^!$ induces a morphism $\rF_3(X) \to \Gamma$.
It remains to identify it with a $\PP^2$-bundle.

\begin{lemma}
\label{lemma:phi-star-cu}
The object $\Phi^*(\cU)$ is a vector bundle of rank~$3$ on $\Gamma$.
\end{lemma}
\begin{proof}
Recall that the functor $\Phi^*$ is given by the kernel $\cE^\vee(K_X)[3]$ on $X \times \Gamma$.
To check that it takes $\cU$ to a rank~3 bundle, it is enough (see~\cite[Lemma~4.4]{K19} for an argument) to show that
\begin{equation*}
H^\bullet(X,\cU \otimes \cE_y^\vee(-1)) \cong \kk^{\oplus 3}[-3]
\end{equation*}
for each $y \in \Gamma$.
For this we dualize~\eqref{eq:cey-x18} and tensor it by $\cU$:
\begin{equation*}
0 \larrow \cU \otimes \cE_y^\vee(-1) \larrow (\cU \otimes \cU)^{\oplus 3} \larrow \cU^{\oplus 6} \larrow \cU \otimes \cE_y^\vee \larrow 0.
\end{equation*}
Note that $\cU$ is acyclic by~\eqref{eq:coh-x18}.
Using Serre duality and exceptionality of $\cU$ we obtain
\begin{equation*}
H^\bullet(X,\cU \otimes \cU) \cong
\Ext^\bullet(\cU^\vee,\cU) \cong
\Ext^{3-\bullet}(\cU,\cU)^\vee \cong \kk[-3].
\end{equation*}
Finally, using Serre duality and the fact that~$\cE_y \in \Phi(\Db(\Gamma)) \subset {\cU^\vee}^\perp$, we obtain
\begin{equation*}
H^\bullet(X,\cU \otimes \cE_y^\vee) \cong
\Ext^\bullet(\cE_y,\cU) \cong
\Ext^{3 - \bullet}(\cU^\vee,\cE_y)^\vee = 0.
\end{equation*}
Combining all the above together, we deduce the lemma.
\end{proof}

Now we can complete the proof of the theorem.

\begin{proof}[Proof of Theorem~\textup{\ref{theorem:x18-f3}}]
We associate with a rational cubic curve $C$ the pair $(y,\phi)$, 
where $y \in \Gamma$ is the point such that $\Phi^!(F_C) \cong \cO_y$ and
\begin{equation*}
\phi \in \Hom(\cU,F_C) = \Hom(\cU,\Phi(\cO_y)) \cong \Hom(\Phi^*(\cU),\cO_y)
\end{equation*}
is the first map in~\eqref{eq:fc-sequence}, which is nonzero and determined by $C$ up to rescaling.
This defines a morphism
\begin{equation*}
\rF_3(X) \larrow \PP_\Gamma(\Phi^*(\cU)^\vee).
\end{equation*}
The target is a $\PP^2$-bundle over $\Gamma$, so it remains to check that the morphism is an isomorphism.
For this we will construct its inverse.

Let $(y,\phi)$ be a pair as above. 
Then~$\phi$ gives a nonzero morphism~$\cU \to \Phi(\cO_y) = \cE_y$.
Both~$\cU$ and~$\cE_y$ are stable vector bundles (see~\cite[Theorem~B.1.1 and Remark~B.5.3]{KPS}) 
of ranks~2 and~3 and~$\rc_1(\cU) = \rc_1(\cE_y) = -H$,
hence the morphism is injective.
The kernel of the dual map~$\cE_y^\vee \to \cU^\vee$ is a reflexive sheaf of rank~1 with~$\rc_1 = 0$, hence isomorphic to $\cO$.
Thus, we have an exact sequence
\begin{equation*}
0 \larrow\cO \larrow \cE_y^\vee \larrow \cU^\vee \larrow T \larrow 0
\end{equation*}
for some coherent sheaf $T$.
The rank and the first Chern class of~$T$ vanish, hence the support of~$T$ is 1-dimensional.
Therefore, after dualization we obtain an exact sequence
\begin{equation*}
0 \larrow \cU \larrow \cE_y \larrow \cO \larrow T' \larrow 0.
\end{equation*}
Now $T'$ is the structure sheaf of a subscheme of $X$, and Riemann--Roch shows that its Hilbert polynomial is $3t + 1$.
Thus, $T' \cong \cO_C$ for a rational cubic curve.
This defines a map 
\begin{equation*}
\PP_\Gamma(\Phi^*(\cU)^\vee) \larrow \rF_3(X).
\end{equation*}
It is easy to see that the two maps constructed above are mutually inverse.
\end{proof}

Applying Corollary~\ref{corollary:fd-gamma} we deduce

\begin{corollary}
\label{corollary:x18-cubics}
The Abel--Jacobi map $\AJ \colon 
\Pic^0(\Gamma) 
\cong \Alb(\rF_3(X)) \to \Jac(X)$ is an isomorphism
of principally polarized abelian varieties.
\end{corollary}

\subsection{Del Pezzo threefold of degree~4}
\label{subsection:v4}

Let $X$ be a smooth del Pezzo threefold of degree~4 over an algebraically closed field of characteristic~0,
i.e., a complete intersection of two quadrics in $\PP(V_6) = \PP^5$.
By~\cite{K06} we have a semiorthogonal decomposition
\begin{equation}
\label{eq:sod-v4}
\Db(X) = \langle \Phi(\Db(\Gamma)), \cO, \cO(1) \rangle,
\end{equation}
where $\Gamma$ is a curve of genus~2.
We will remind a description of the embedding functor 
\begin{equation*}
\Phi \colon \Db(\Gamma) \longrightarrow \Db(X) 
\end{equation*}
below (see~\eqref{eq:cey-v4}).
Meanwhile, recall that $\cO$ and $\cO(1)$ are exceptional bundles and that
\begin{equation}
\label{eq:coh-v4}
H^\bullet(X,\cO(-1)) = 0,
\qquad 
H^\bullet(X,\cO) = \kk,
\qquad 
H^\bullet(X,\cO(1)) = V_6^\vee.
\end{equation} 

Recall the following result.

\begin{theorem}[{\cite[Lemma~5.5]{K12}}]
\label{theorem:v4-f1}
The left adjoint functor $\Phi^* \colon \Db(X) \to \Db(\Gamma)$ induces a noncanonical isomorphism $\rF_1(X) \cong \Pic^0(\Gamma)$.
\end{theorem}

Applying Corollary~\ref{corollary:fd-pic} we deduce 

\begin{corollary}
\label{corollary:v4-lines}
The Abel--Jacobi map $\AJ \colon \rF_1(X) \cong \Alb(\rF_1(X)) \to \Jac(X)$ is an isomorphism.
\end{corollary}

Below we prove an analogue of this result for conics on $X$.

\begin{theorem}
\label{theorem:v4-f2}
The Hilbert scheme $\rF_2(X)$ of conics on $X$ is a $\PP^3$-bundle over~$\Gamma$.
\end{theorem}

We will show that appropriate extensions of ideals of conics in $X$ belong to the component~$\Db(\Gamma)$ 
of~\eqref{eq:sod-v4} and correspond to structure sheaves of points in $\Gamma$.
The proof is analogous to the proof of Theorem~\ref{theorem:x18-f3}.

\begin{lemma}
\label{lemma:fc-ax-v4}
If $C \subset X$ is a conic then 
\begin{equation*}
\Ext^1(I_C,\cO(-1)) \cong \kk.
\end{equation*}
Let $F_C$ be the corresponding nontrivial extension, so that we have an exact sequence
\begin{equation}
\label{eq:fc-sequence-v4}
0 \larrow \cO(-1) \larrow F_C \larrow \cO \larrow \cO_C \larrow 0.
\end{equation}
Then $F_C$ is a stable vector bundle of rank~$2$ with $\rc_1(F_C) = -H$. 
Moreover, $F_C \in \langle \cO, \cO(1) \rangle^\perp$.
\end{lemma}
\begin{proof}
Analogous to the proof of Lemma~\ref{lemma:fc-ax}.
\end{proof}

\begin{proposition}
\label{proposition:phi-s-fc-v4}
If $C \subset X$ is a conic then $\Phi^!(F_C) \cong \cO_y$, where $y \in \Gamma$.
\end{proposition}
\begin{proof}
Analogous to the proof of Proposition~\ref{proposition:phi-s-fc} with exact sequence
\begin{equation}
\label{eq:cey-v4}
0 \larrow \cE_y \larrow \cO^{\oplus 4} \larrow \cO(1)^{\oplus 4} \larrow \cE_y(1) \larrow 0
\end{equation}
used instead of~\eqref{eq:cey-x18}.
\end{proof}

This proposition already proves that the functor $\Phi^!$ induces a morphism $\rF_2(X) \to \Gamma$.
It remains to identify it with a $\PP^3$-bundle.

\begin{lemma}
The object $\Phi^*(\cO(-1))$ is a vector bundle of rank~$4$ on $\Gamma$.
\end{lemma}
\begin{proof}
Analogous to the proof of Lemma~\ref{lemma:phi-star-cu}.
\end{proof}

A combination of the above results proves Theorem~\ref{theorem:v4-f2} along the lines of the proof of Theorem~\ref{theorem:x18-f3}.

Of course, the isomorphism $\rF_2(X)$ and a $\PP^3$-bundle over $\Gamma$ can be established geometrically
(see, e.g., \cite[Proposition~2.3.8(ii)]{KPS}).
However, the above argument gives a categorical flavor to this isomorphism, 
and also allows to apply Corollary~\ref{corollary:fd-gamma} and deduce the following

\begin{corollary}
\label{corollary:v4-conics}
The Abel--Jacobi map $\AJ \colon \Pic^0(\Gamma) \cong \Alb(\rF_2(X)) \to \Jac(X)$ is an isomorphism
of principally polarized abelian varieties.
\end{corollary}

\appendix

\section{Tetragonal property}

\makeatletter
\@addtoreset{equation}{subsection}
\makeatother
\renewcommand{\theequation}
{\Alph{section}.\arabic{subsection}.\arabic{equation}}
\renewcommand{\thesubsection}{\Alph{section}.\arabic{subsection}}

Recall that a curve $C$ is called \emph{tetragonal} if it has a 
positive dimensional linear system of degree~$4$.
The geometric Riemann--Roch theorem implies that a smooth canonical curve~\mbox{$C \subset \PP^g$} is tetragonal, 
if and only if it has a subscheme of length~4 contained in a plane.
This motivates the following

\begin{definition}
A projective variety~$X \subset \PP^n$ is {\sf not tetragonal} if it is an intersection of quadrics
and there is no plane $\Pi \subset \PP^n$ such that $X \cap \Pi$ is a finite scheme of length~4.
\end{definition}

The following lemma is obvious.

\begin{lemma}
\label{lemma:tetragonality-section}
An intersection of quadrics $X \subset \PP^n$ is not tetragonal if an only if 
the same is true for any \textup(not necessarily transverse\textup) linear section of~$X$.
\end{lemma}

On the other hand, recall the property~$\bN_2$ of Green and Lazarsfeld~\cite{GL86}.

\begin{definition}
A projective variety~$X \subset \PP^n$ {\sf satisfies the property~$\bN_2$} if its structure sheaf has a locally free resolution of the form
\begin{equation*}
\dots \longrightarrow \cO_{\PP^{n}}(-3)^{\oplus m_3} \longrightarrow \cO_{\PP^{n}}(-2)^{\oplus m_2} \longrightarrow \cO_{\PP^{n}} \longrightarrow \cO_X \longrightarrow 0.
\end{equation*}
\end{definition}

We have the following implication.

\begin{proposition}
\label{prop:n2-tetragonality}
If a projective variety satisfies the property~$\bN_2$, it is not tetragonal.
\end{proposition}

\begin{proof}
Assume~$X$ satisfies the property~$\bN_2$.
The form of the resolution implies that~$X$ is an intersection of quadrics.
Assume there is a plane~$\Pi \subset \PP^n$ such that~$Z = X \cap \Pi$ is a finite scheme of length~4.
Tensoring the resolution of~$\cO_X$ with~$\cO_\Pi$ we obtain a complex 
\begin{equation}
\label{eq:resolution-restricted}
\dots \longrightarrow \cO_\Pi(-3)^{\oplus m_3} \longrightarrow \cO_\Pi(-2)^{\oplus m_2} \longrightarrow \cO_\Pi,
\end{equation}
whose cohomology in $p$-th term is~$\Tor_p(\cO_\Pi,\cO_X)$, which is an Artinian sheaf supported on~$Z$.
Moreover, the cohomology at the right term is~$\cO_Z$, 
hence the rightmost map in~\eqref{eq:resolution-restricted} is an epimorphism~$\cO_\Pi(-2)^{\oplus m_2} \twoheadrightarrow I_{Z,\Pi}$.

On the other hand, the scheme~$Z$ is an intersection of conics in~$\Pi$, hence it is a complete intersection of two conics.
Therefore, the ideal of~$Z$ on~$\Pi$ has a resolution
\begin{equation*}
0 \longrightarrow \cO_\Pi(-4) \longrightarrow \cO_\Pi(-2)^{\oplus 2} \longrightarrow I_{Z,\Pi} \longrightarrow 0.
\end{equation*}
The epimorphism~$\cO_\Pi(-2)^{\oplus m_2} \twoheadrightarrow I_{Z,\Pi}$ lifts canonically to a morphism~$\cO_\Pi(-2)^{\oplus m_2} \to \cO_\Pi(-2)^{\oplus 2}$,
which also has to be surjective (otherwise the composition~$\cO_\Pi(-2)^{\oplus m_2} \to \cO_\Pi(-2)^{\oplus 2} \to I_{Z,\Pi}$ 
factors through~$\cO_\Pi(-2)$ and thus cannot be surjective).
Therefore, the kernel of the rightmost arrow in~\eqref{eq:resolution-restricted} is isomorphic to $\cO_\Pi(-2)^{\oplus m_2 - 2} \oplus \cO_\Pi(-4)$, 
and hence we have an exact sequence
\begin{equation*}
\cO_\Pi(-3)^{\oplus m_3} \longrightarrow \cO_\Pi(-2)^{\oplus m_2 - 2} \oplus \cO_\Pi(-4) \longrightarrow \Tor_1(\cO_\Pi,\cO_X) \longrightarrow 0.
\end{equation*}
But $\Hom(\cO_\Pi(-3), \cO_\Pi(-4)) = 0$, hence $\cO_\Pi(-4)$ is a summand of $\Tor_1(\cO_\Pi,\cO_X)$,
which is absurd since the latter is an Artinian sheaf.
This contradiction shows that~$X$ is not tetragonal.
\end{proof}

\begin{lemma}
\label{lemma:hv-n2}
If~$\bX$ is one of the following homogeneous varieties
\begin{equation*}
\OGr_+(5,10) \subset \PP^{15},
\quad 
\Gr(2,6) \subset \PP^{14},
\quad 
\LGr(3,6) \subset \PP^{13},
\quad 
\GTGr(2,7) \subset \PP^{13},
\quad 
\Gr(3,7) \subset \PP^{34}
\end{equation*}
and~$x \in \bX$ is any point then the Hilbert scheme~$\rF_1(\bX,x) \subset \PP(T_x\bX)$ of lines on~$\bX$ through the point~$x$ satisfies the property~$\bN_2$.
\end{lemma}

\begin{proof}
By~\cite{LM78} the Hilbert schemes of lines~$\rF_1(\bX,x)$ for the above homogeneous varieties have an explicit description, namely
\begin{equation*}
\Gr(2,5) \subset \PP^{9},
\quad 
\PP^1 \times \PP^3 \subset \PP^{7},
\quad 
v_2(\PP^2) \subset \PP^{5},
\quad 
v_3(\PP^1) \subset \PP^{3},
\quad 
\PP^2 \times \PP^3 \subset \PP^{11},
\end{equation*}
where~$v_2$ and~$v_3$ stand for the Veronese embeddings of degree~2 and~3, respectively.
All of their structure sheaves have explicitly known locally free resolutions which take the following forms:
\begin{equation*}
\xymatrix@R=0.1em@C=1.7em{
0 \ar[r] & \cO(-5) \ar[r] & \cO(-3)^{\oplus 5} \ar[r] & \cO(-2)^{\oplus 5} \ar[r] & \cO \ar[r] & \cO_{\Gr(2,5)} \ar[r] & 0,
\\
0 \ar[r] & \cO(-4)^{\oplus 3} \ar[r] & \cO(-3)^{\oplus 8} \ar[r] & \cO(-2)^{\oplus 6} \ar[r] & \cO \ar[r] & \cO_{\PP^1 \times \PP^3} \ar[r] & 0,
\\
0 \ar[r] & \cO(-4)^{\oplus 3} \ar[r] & \cO(-3)^{\oplus 8} \ar[r] & \cO(-2)^{\oplus 6} \ar[r] & \cO \ar[r] & \cO_{v_2(\PP^2)} \ar[r] & 0,
\\
&0 \ar[r] & \cO(-3)^{\oplus 2} \ar[r] & \cO(-2)^{\oplus 3} \ar[r] & \cO \ar[r] & \cO_{v_3(\PP^1)} \ar[r] & 0,
\\
&\dots \ar[r] & \cO(-3)^{\oplus 52} \ar[r] & \cO(-2)^{\oplus 12} \ar[r] & \cO \ar[r] & \cO_{\PP^2 \times \PP^3} \ar[r] & 0,
} 
\end{equation*}
see~\cite{Wey},
in particular, the property~$\bN_2$ holds for~$\rF_1(\bX,x)$.
\end{proof}

\begin{corollary}
\label{cor:length-f1}
For any prime Fano threefold~$X$ of genus~$g \ge 7$ the length of the Hilbert scheme~$\rF_1(X,x)$ is at most~$3$.
\end{corollary}

\begin{proof}
By~\cite{Mukai92} any prime Fano threefold~$X$ of genus~$g \ge 7$ is a linear section (not necessarily transverse) 
of one of homogeneous varieties~$\bX$ of Lemma~\ref{lemma:hv-n2}.
Therefore, for any~$x \in X$ the scheme~$\rF_1(X,x)$ is a linear section (not necessarily transverse) of~$\rF_1(\bX,x)$.
The scheme~$\rF_1(\bX,x)$ satisfies the property~$\bN_2$ by Lemma~\ref{lemma:hv-n2}, 
hence~$\rF_1(\bX,x)$ is not tetragonal by Proposition~\ref{prop:n2-tetragonality}, 
and hence~$\rF_1(X,x)$ is not tetragonal by Lemma~\ref{lemma:tetragonality-section}.

On the other hand, by Lemma~\ref{lemma:f1-x-x} the scheme~$\rF_1(X,x) \subset \PP(T_xX) \cong \PP^2$ is an intersection of quadrics, 
so either it contains a conic, or it contains a line, or it is finite of length at most~4.
The first two cases are impossible, because~$X$ contains neither quadratic cones of dimension~2, nor planes by Theorem~\ref{th:bht}.
In the last case, the length cannot be equal to~4 because~$\rF_1(X,x)$ is not tetragonal.
Therefore, $\rF_1(X,x)$ is a finite scheme of length at most~3.
\end{proof}


\end{document}